\documentclass[12pt]{amsart}

\usepackage{amsmath,amssymb,amsxtra,amsthm,hyperref}

\DeclareMathOperator*\uplim{\overline{lim}}
\renewcommand{\limsup}{\uplim} 

\usepackage{multicol}
\usepackage{mathtools}
\usepackage{calc}
\usepackage{graphicx,tikz}
    \usetikzlibrary{calc,decorations.markings,arrows.meta}

\usepackage[
    margin=2cm
    ]{geometry}
\usepackage{enumitem} 
    \makeatletter
    			\def\namedlabel#1#2{\begingroup#2%
    			\def\@currentlabel{#2}%
    			\phantomsection\label{#1}\endgroup
    		}
    		\makeatother
\usepackage{tensor} 
\usepackage{tikz-cd} 
\usepackage{upref}  
\usepackage{mathrsfs} 
\usepackage{MnSymbol} 

\theoremstyle{plain}
    \newtheorem{theorem}{Theorem}[section]
    \newtheorem{corollary}[theorem]{Corollary}
    \newtheorem{proposition}[theorem]{Proposition} 
    \newtheorem{lemma}[theorem]{Lemma}

\theoremstyle{definition}
    \newtheorem{definition}[theorem]{Definition}
    \newtheorem{example}[theorem]{Example}
    \newtheorem{remark}[theorem]{Remark}
    \newtheorem*{notation*}{Notation}
    \newtheorem{notation}[theorem]{Notation}
\numberwithin{equation}{section}

\newcommand{\assumption}{}

\newcommand{\cB}{\mathscr{B}}
\newcommand{\cC}{\mathscr{C}}
\newcommand{\cD}{\mathscr{D}}
\newcommand{\cL}{\mathscr{L}}
\newcommand{\cM}{\mathscr{M}}
\newcommand{\cN}{\mathscr{N}}
\newcommand{\cP}{\mathscr{P}}
\newcommand{\cK}{\mathscr{K}}
\newcommand{\usc}{up\-per se\-mi-con\-ti\-nuous}
\newcommand{\USCBb}{USC Ba\-nach bun\-dle}
\newcommand{\ib}{im\-prim\-i\-tiv\-ity bi\-mod\-u\-le}

\newcommand{\cG}{\mathcal{G}}
\newcommand{\cH}{\mathcal{H}}
\newcommand{\etale}{\'{e}tale}
\newcommand{\LCH}{lo\-cal\-ly com\-pact Haus\-dorff}

\newcommand{\X}{\mathcal{X}}
\newcommand{\Y}{\mathcal{Y}}
\newcommand{\Z}{\mathcal{Z}}
\newcommand{\hequiv}{hypo-equi\-va\-lence}
\newcommand{\Hequiv}{Hypo-Equi\-va\-lence}
\newcommand{\pe}{pre-equi\-va\-lence} 
\newcommand{\Pe}{Pre-Equi\-va\-lence} 
\newcommand{\peshort}{GP} 
\newcommand{\peadjective}{pre-equi\-va\-lent}

\newcommand{\Rel}{\mathbin{\mathcal{R}}} 
\newcommand{\FBRel}{\mathbin{\mathscr{R}}} 

\newcommand{\dbtilde}[1]{\tilde{\raisebox{0pt}[0.85\height]{$\tilde{#1}$}}} 

\newcommand{\sme}{\,\mathord{\mathop{\text{--}}\nolimits_{\relax}}\,}

\let\ipscriptstyle=\scriptscriptstyle
\def\lipsqueeze{{\mskip -2.0mu}}
\def\ripsqueeze{{\mskip -2.0mu}}
\def\ipcomma{\nobreak \mid\nobreak} 
\newbox\ipstrutbox
\setbox\ipstrutbox=\hbox{\vrule
    height5.5pt
    depth 2pt
    width 0pt
}

\newcommand{\norm}[1]{\left\| #1 \right\|}
\newcommand{\abs}[1]{\left| #1 \right|}

\newcommand{\inner}[2]{\langle #1 \ipcomma #2 \rangle}
\newcommand{\Inner}[2]{\llangle #1 \ipcomma #2 \rrangle}

\newcommand{\linner}[4][]{
    {
    \relax_{\ipscriptstyle
                #2
                }^{\ipscriptstyle
                #1}
                \lipsqueeze
                \langle #3 \ipcomma #4 \rangle
    }
}
    \def\lip#1<#2,#3>{\linner{#1}{#2}{#3}}

\newcommand{\lInner}[4][]{
    {
    \relax_{\ipscriptstyle
    #2
    }^{\ipscriptstyle
    #1}
    \lipsqueeze
    \llangle #3 \ipcomma #4 \rrangle
    }
}

\newcommand{\rinner}[4][]{
    {
        \langle #3
        \ipcomma
        #4 \rangle_{\ripsqueeze
        \ipscriptstyle
        #2
        }^{\ripsqueeze
        \ipscriptstyle
        #1}
    }
}
    \def\rip#1<#2,#3>{\rinner{#1}{#2}{#3}}

\newcommand{\rInner}[4][]{
    {
        \llangle #3
        \ipcomma
        #4 {\rrangle_{\ripsqueeze
        \ipscriptstyle
        #2
        }^{\ripsqueeze
        \ipscriptstyle
        #1}}
    }
}

\newcommand{\inv}{^{-1}}
\newcommand{\z}{^{(0)}}

\newcommand{\bfp}[2]{\lipsqueeze\tensor*[_{\ipscriptstyle #1}]{\ast}{_{\ipscriptstyle #2}}\ripsqueeze} 

\newcommand{\leoq}[4][]{\tensor*[_{\ipscriptstyle #2}^{\ipscriptstyle #1}]{\left\{ #3 \ipcomma #4 \right\}}{}} 

\newcommand{\reoq}[4][]{\tensor*[]{\left\{ #2 \ipcomma #3 \right\}}{_{\ipscriptstyle #4}^{\ipscriptstyle #1}}} 

\newcommand{\Leoq}[4][]{L_{\ipscriptstyle #2}^{\ipscriptstyle #1}\mathopen{}\left( #3 \ipcomma #4 \right)\mathclose{}} 

\newcommand{\Leoqempty}[2][]{L_{\ipscriptstyle #2}^{\ipscriptstyle #1}} 

\newcommand{\Reoq}[4][]{R_{\ipscriptstyle #4}^{\ipscriptstyle #1}\mathopen{}\left( #2 \ipcomma #3 \right)\mathclose{}} 

\newcommand{\Reoqempty}[2][]{R_{\ipscriptstyle #2}^{\ipscriptstyle #1}} 

\newcommand\Item[1][]{%
  \ifx\relax#1\relax  \item \else \item[#1] \fi
  \abovedisplayskip=0pt\abovedisplayshortskip=0pt~\vspace*{-\baselineskip}}

\title{Equivalence of Fell bundles is an equivalence relation} 

\author{Anna Duwenig}
\address{School of Mathematics and Applied Statistics, University of Wollongong, Wollongong, Australia}
\email{aduwenig@uow.edu.au}

\author{Boyu Li}
\address{Department of Pure Mathematics, University of Waterloo, Waterloo, ON., Canada}
\email{b32li@uwaterloo.ca}
\date{\today}


\begin{document}



\allowdisplaybreaks

\maketitle

\begin{abstract}
    We introduce the notion of {groupoid \pe s} and prove that they give rise to groupoid equivalences by taking certain quotients. Then, given an equivalence of Fell bundles $\cB$ and $\cC$ and another equivalence between $\cC$ and $\cD$ over \LCH\ \etale\ groupoids, we construct an equivalence between $\cB$ and $\cD$ out of the tensor product bundle. As a consequence, we obtain that Fell bundle equivalence is indeed an equivalence relation.
\end{abstract}

\section{Introduction}

The concept of groupoid equivalence was first introduced in \cite{MRW:Grpd} to connect with the notion of strong Morita equivalence of groupoid $\textrm{C}^*$-algebras. Roughly speaking, a groupoid equivalence $\X$ for two groupoids $\cG$ and $\cH$ is a left $\cG$- and right $\cH$-space that satisfies several tracable algebraic and topological conditions. These conditions allow us to build an \ib\ for $\textrm{C}^*(\cG)$ and $\textrm{C}^*(\cH)$, proving that they are strongly Morita equivalent. A related notion of Fell bundle equivalence was first outlined in unpublished work by Yamagami and later formalized by Muhly and Williams in~\cite{MW2008}. One of the main results in~\cite{MW2008}, Theorem 6.4, is that an equivalence of Fell bundles, similarly to an equivalence of groupoids, implies the strong Morita equivalence of their associated $\textrm{C}^*$-algebras.

The notion of equivalence for groupoids is easily seen to be a reflexive and symmetric relation: Any groupoid $\cG$ is an equivalence from itself to itself, and given a $(\cG,\cH)$-equivalence~$\X$, it is easy to construct an $(\cH,\cG)$-equivalence $\X^{\mathrm{op}}$ whose left $\cH$-action, for example, is defined out of the right $\cH$-action on~$\X$. Showing that equivalence of groupoids is transitive is slightly more involved. Given a second groupoid equivalence $\Y$ between $\cH$ and $\mathcal{K}$, the fibre product $\X\bfp{s}{r}\Y$ is generally not a $(\cG,\mathcal{K})$-equivalence; instead, one has to take a quotient of $\X\bfp{s}{r}\Y$ by balancing over the middle groupoid~$\cH$. The resulting balanced product $\X\bfp{}{\cH}\Y$ is a known $(\cG,\mathcal{K})$-equivalence \cite[p.~6]{MRW:Grpd}, which shows that groupoid equivalence is indeed an equivalence relation.

For equivalence of Fell bundles, it is similarly easy to show reflexivity and symmetry \cite[Example 6.6 and 6.7]{MW2008}. For (not necessarily saturated) Fell bundles over \LCH\ {\em groups}, this relation was furthermore shown to be transitive in \cite{AF:EquivFb}. However, for Fell bundles over groupoids, an analogue of the balanced product construction in the world of groupoids has so far been missing in the literature, and so, despite its name, it was unclear whether equivalence of Fell bundles over groupoids is a transitive relation. 
We point out that, while strong Morita equivalence of $\textrm{C}^*$-algebras is known to be an equivalence relation, it may not be true that two strongly Morita equivalent Fell bundle $\textrm{C}^*$-algebras have arisen from Fell bundles that are  equivalent in the sense of \cite{MW2008}. 

\smallskip

Thus, the main motivation behind this paper is to prove that
the notion of equivalences of  Fell bundles over \LCH\ \etale\ groupoids
is indeed a transitive relation, 
see Theorem~\ref{thm:cP-is-equiv}. To do so, we mimic the construction in the world of groupoids, by taking a quotient of the tensor product of two Fell bundle equivalences. Along the way, we first need to introduce two new relations that are better behaved under taking products than the known notions of equivalence; we call these new relations {\em groupoid \pe} and {\em Fell bundle \hequiv} (Definition~\ref{def:preequiv} resp.~\ref{def:FBword}). 
We show that the fibre product of two {\pe s} is another \pe\ (Lemma~\ref{lem:product-of-preequiv}), and that the tensor product of two 
\hequiv s is another \hequiv\ (Theorem~\ref{thm:cK-is-word}). We prove that any groupoid \pe\ gives rise to an equivalence by taking an appropriate quotient (Proposition~\ref{prop:from-preequiv-to-equivalence}); this generalizes the above-mentioned result that
$\X\bfp{s}{r}\Y$ 
gives rise to
the equivalence $\X\bfp{}{\cH}\Y$.
Ideally, we would like to have a similar result for an appropriate notion of `\peadjective\ Fell bundles', but while the authors conjecture that a slight strengthening of 
\hequiv\ may allow such a result, this article will not be concerned with that question. Instead, we will only focus on the case when the \hequiv\ arises as the tensor product $\cM\otimes_{\cC}\cN$ of two Fell bundle equivalences $\cM$ from $\cB$ to $\cC$ and $\cN$ from $\cC$ to $\cD$, in which case it indeed can be rigged to give an equivalence from $\cB$ to $\cD$, which is the content of Theorem~\ref{thm:cP-is-equiv}.

\section{Definitions and notation}

\begin{notation*}
    Suppose we have maps $s\colon M\to Z$, $r\colon N\to Z$. We define the {\em fibre product} of $M$ and $N$ as
    \[
        M\bfp{s}{r}N := \{(m,n)\in M\times N : s(m)=r(n)\}.
    \]
    If the spaces involved are topological and the maps continuous, we will always equip $M\bfp{s}{r}N$ with the subspace topology. If $M$ and $N$ are the total spaces of bundles $\cM$ resp.\ $\cN$, we will sometimes write $\cM\bfp{s}{r}\cN$ for $ M\bfp{s}{r}N$ instead. If the maps carry subscripts, we will refrain from writing those in the  fibre product; in other words, we will prefer writing $M\bfp{s}{r}N$  
    over  $M\bfp{s_{\cM}}{r_{\cN}}N$,  
    as the maps involved should always be clear from context.
    Given $U\subseteq M$ and $V\subseteq N$, we write $U\bfp{}{}V$ or $U\bfp{s}{r}V$ for $(U\times V)\cap M\bfp{s}{r}N$. 
\end{notation*}

\begin{definition}[{\cite[p.~5]{MRW:Grpd}}]\label{df.left.action} 
		A \LCH\ groupoid $\mathcal{G}$ {\em acts on the left} on a \LCH\ space  $X $ if there is a continuous, open surjection $r_X \colon X \to \mathcal{G}\z $ \textup(called \emph{momentum map}\textup) and a continuous map
		\(
			\mathcal{G}\bfp{s}{r} X  
			\to
			X ,
			(g , x) \mapsto g \cdot x,
		\)
		such that
		\begin{enumerate}[leftmargin=1.7cm, label={(GA\arabic*)}]
			\item\label{item:def-GA:range} $r_X (g . x) = r_\mathcal{G} (g )$ for all $(g , x)\in \mathcal{G}\bfp{s}{r} X $,
			\item\label{item:def-GA:assoc} if $(g , x)\in \mathcal{G}\bfp{s}{r} X $ and $(g ',g )\in \mathcal{G}^{(2)}$, then $(g 'g , x)\in \mathcal{G}\bfp{s}{r}X $ and $(g 'g )\cdot x = g '\cdot (g  \cdot x)$, and
			\item\label{item:def-GA:unit} $r_X (x)\cdot x = x$ for all $x\in X $.
		\end{enumerate}
		We call $X$ a {\em left $\cG$-space}. We further call
		$X $ {\em principal} if the action is
		\begin{enumerate}[leftmargin=1.7cm, resume, label={(GA\arabic*)}]
			\item\label{item:def-GA:free} {\em free}, meaning that $g  \cdot x=x$ implies $g  = r_X  (x)$; and
			\item\label{item:def-GA:proper} {\em proper}, meaning that, if $(g_\lambda\cdot x_\lambda, x_\lambda)\to (y,x)$ in $X\times X$, then the net $(g_\lambda)_\lambda$ in $\cG$ has a convergent subnet (see \cite[Proposition 2.7]{Wil2019}).
		\end{enumerate}
		If we want to exchange left by right, we need to swap range and source maps.
	\end{definition}
	
	We point out that there are other definitions of groupoid actions in the literature in which the momentum map $r_{X}$ is not necessarily assumed to be open. In this article, however,  
	openness of $r_X$ will be used frequently. 
	
	\begin{definition}[{\cite[Def.\ 2.1]{MRW:Grpd}}]\label{def:GE}
		Two \LCH\ groupoids $\cG$ and $\cH$ are {\em equivalent} if there is a topological space $\X$ such that
		\begin{enumerate}[leftmargin=1.7cm, label={(GE\arabic*)}]
			\item\label{item:def-GE:principal} $\X$ is a principal left $\cG$- and principal right $\cH$-space,
			\item\label{item:def-GE:commute} the left and right action commute, and
			\item\label{item:def-GE:homeos} $r_\X$ induces a homeomorphism between $\X\slash \mathcal{H}$ and $\cG\z $, and $s_\X$ induces a homeomorphism between $\cG\! \text{\rotatebox[origin=c]{35}{$\slash$}}\! \X$ and $\mathcal{H}\z $.
		\end{enumerate}
		We then call $\X$ a {\em $(\mathcal{G},\mathcal{H})$-equivalence}.
	\end{definition}

As observed in \cite{MRW:Grpd}, whenever $x,x'\in \X$ with $s_{X}(x)=s_{X}(x')$, there exists a unique $g\in \cG$ such that $x=g\cdot x'$. This element is denoted by $\leoq[X]{\cG}{x}{x'}$. Similarly, when $r_{X}(x)=_{X}r(x')$, then $h=\reoq[X]{\cdot}{\cdot}{\cH}$ is the unique element in $\cH$ such that $x\cdot h=x'$.

It is well known (see {\cite[p.~6]{MRW:Grpd}}
) that groupoid equivalence is indeed an equivalence relation. Let us recall why this relation is transitive: Suppose 
$\X  $ is a $(\cG ,\cH )$-equivalence and $ \Y $ is an $(\cH ,\mathcal{K} )$-equivalence. Then the quotient $\X  \bfp{}{\cH}  \Y$ of $\X  \bfp{s}{r}  \Y $ by forcing the equality
\[
	[x \cdot h, h\inv \cdot  y]_{\ipscriptstyle\cH} = [x , y ]_{\ipscriptstyle\cH}
	\quad
	\text{for all $h\in\cH$ for which $\cdot$ makes sense,}
\]
is a $(\cG ,\mathcal{K})$-equivalence with momentum maps
\[
	\begin{tikzcd}
			 \X  \ast_\cH  \Y  \ar[r, "{r}"] 	& \cG \z	&[-25pt]\text{and} &[-25pt]  \X  \ast_\cH  \Y  \ar[r, "{s}"] 	& \mathcal{K}\z \\[-20pt]
		\left[x , y \right]_{\ipscriptstyle\cH} \ar[r, mapsto] 			& r_\X   (x )	& & \left[x , y \right]_{\ipscriptstyle\cH} \ar[r, mapsto] 			& s_\Y  ( y )
	\end{tikzcd}
\]
	and the actions
\[
		g \cdot  \left[x ,  y \right]_{\ipscriptstyle\cH} := \left[g\cdot  x ,  y \right]_{\ipscriptstyle\cH}
		\qquad\text{resp.}\qquad
        \left[x ,  y \right]_{\ipscriptstyle\cH} \cdot  k := \left[x ,  y  \cdot  k\right]_{\ipscriptstyle\cH}.
\]

\begin{definition}[{\cite[Definition 2.1]{BE2012}}]\label{def.USCBdl}
Suppose $X$ and $M$ are topological spaces and $q_{\cM}\colon M\to X$ a continuous surjection.
We call $\cM=(q_{\cM}\colon M\to X)$ an {\em \usc\ \textup(USC\textup) Banach bundle} 
if its fibres $\cM_{x}=M(x)=q_{\cM}\inv(x)$ have the structure of complex Banach spaces and if the following hold:
\begin{enumerate}[leftmargin=2cm,label=(USC\arabic*)]
    \item\label{cond.B-USC} The map $M\to \mathbb{R}_{\geq 0}$, $m\mapsto\norm{m}$, is \usc.
    \item\label{cond.B-plus} The map  $M\bfp{q}{q} M \to M$, $(m,m')\mapsto m+m'$, is continuous.
    \item\label{cond.B-times} For each $\lambda\in\mathbb{C},$ the map $M\to M,$ $m\mapsto \lambda m$, is continuous.
    \item\label{cond.B-nets} If $(m_{i})_{i}$ is a net in $M$ such that $q_{\cM}(m_{i})$ converges to $x\in X$ and $\norm{m_{i}}\to 0$, then $(m_{i})_{i}$ converges to $0\in M(x)$ in $M$.
\end{enumerate}
\end{definition}
We will not always make a clear distinction between the total space $M$ and the bundle $\cM$ itself.

\begin{remark}\label{rmk:enough-sections}
    We point out that our USC bundles always contain {\em enough sections}, i.e., for any given $x\in X$ and $m\in \cM_x$,  there is a continuous section $\sigma$ such that $\sigma(x) = m$; see also \cite[Appendix A]{MW2008} and \cite[Corollary 2.10]{LAZAR2018448}. 
\end{remark}

\begin{notation}\label{notation:s-and-r}
    If $\cB=(p_{\cB}\colon B\to \cG)$ is a bundle over a groupoid $\cG$, we write
    \[
        s_{\cB}:= s_{\cG}\circ p_{\cB}\colon \cB\to \cG\z
        \quad\text{and}\quad
        r_{\cB}:= r_{\cG}\circ p_{\cB}\colon \cB\to \cG\z.
    \]
    Similarly, if $\cM=(q_{\cM}\colon M\to X)$ is a bundle over a left $\cG$-space $X$ with momentum map $r_{X}\colon X\to \cG\z$, we write
    \[
        r_{\cM}:= r_{X}\circ q_{\cM}\colon M\to \cG\z.
    \]
    If $X$ is a space with a right action instead, so that its momentum map is a source map, we analogously define $s_{\cM}.$ And 
    if there is no ambiguity, we may drop the subscripts.
\end{notation}

\begin{definition}\label{def:beefing-up-a-USC-bundle}
    Let $\cM=(q_{\cM}\colon M\to X)$ be a \USCBb\ and $\cG$ a groupoid. Suppose we are given a continuous map $t\colon X\to \cG\z$. Consider the continuous map
    \(
        f\colon \cG\bfp{s}{t}X
        \to X,
        (g,x)\mapsto x.
    \)
    The {\em pull-back bundle $f^* (\cM)$ of $\cM$ via $f$} is the bundle over $\cG\bfp{s}{t}X$  defined by
\begin{align*}
    f^* (\cM)
    &=
    \{
    (g,x,m)\in (\cG\bfp{s}{t}X)\times \cM
    \,\vert\,
    f(g,x)=q_{\cM}(m)
    \}
    \\
    &\cong
    \{
        (g,m)\in \cG\times \cM
    \,\vert\,
     s_{\cG}(g)=t(q_{\cM}(m))
    \}.
\end{align*}
We will denote this pull-back bundle by $\cG \bfp{s}{t}  \cM$ rather than $f^* (\cM)$, and we make an analogous definition for $\cM \bfp{t}{r}  \cG$.
\end{definition}
We want to give a description of how to topologize these pull-back bundles. To this end, we first need to know how to topologize arbitrary bundles:

\begin{remark}[{see \cite[Corollary 3.7]{Hofmann1977}}]\label{rmk:topology-in-USC-bundles}
    If we are given a bundle $\cM=(q_{\cM}\colon M\to X)$ of Banach spaces, where $M$ is just 
    an untopologized set,
    then there is a standard trick of inducing a topology on~$M$ via a space of sections. To be more precise,
    suppose that $\Gamma$ is a vector space of sections of~$\cM$ such that
    \begin{enumerate}[label=\textup{(\arabic*)}]
      \item\label{item:sections:usc} for each $\gamma\in \Gamma$, the map $X \to\mathbb{R}_{\geq 0}, x\mapsto \norm{\gamma(x)},$ is \usc, and
      \item\label{item:sections:dense} for each $x\in X $, the set $\Gamma(x):=\{\gamma(x)\,\mid\,\gamma\in\Gamma\}$ is dense in $\cM_x $.
    \end{enumerate}
    Then 
    there exists a unique topology on  $M$
    which makes $\cM$ a \USCBb\ such that all elements of $\Gamma$ are continuous sections; see \cite[13.18]{FellDoranVol1} for a proof for continuous Banach bundles and \cite[Theorem C.25]{Wil2007} for a proof for \usc\ $\textrm{C}^*$-bundles which can be readily adapted to {\USCBb s}. We give a description of net convergence with respect to this topology in  Lemma~\ref{lm.conv.equivalence.conditions}.
\end{remark}

\begin{lemma}\label{lem:beefing-up-a-USC-bundle}
    The pull-back bundle $\cG \bfp{s}{t}  \cM$ has a unique topology that makes it a \USCBb\ such that for any continuous cross section $\sigma$ of $\cM$, 
    the map
    \[
    \mathrm{id}\ast\sigma\colon\quad \cG \bfp{s}{t}  X \to \cG \bfp{s}{t}  \cM,\quad
    (g,x)\mapsto (g,\sigma(x)),
    \]
    is a continuous cross-section. Moreover, the sets  $U_{1} \bfp{}{} U_{2}$
    for  $U_{1}\subseteq\cG, U_{2}\subseteq \cM$ basic open sets, form a basis for this topology.
\end{lemma}

Note that, the bundle automatically has enough continuous cross-sections (see Remark~\ref{rmk:enough-sections}).

\begin{proof}
    Note that $\cG \bfp{s}{t}   X $ is \LCH, as it is a closed subspace of the \LCH\ space $\cG\times  X $.
    By construction, the projection map $\pi\colon \cG \bfp{s}{t}  \cM \to \cG\bfp{s}{t}X, (g,m)\mapsto (g,q_{\cM}(m)),$  is a surjection and each fibre
    \[
       \pi\inv(g,x)
       =
       \{
           (g,m) \,\mid\,q_{\cM}(m)=x
       \}
       \cong
       \cM_{x}
    \]
    is a Banach space.
    We now let $\Gamma$ denote the linear span of all elements $
       \mathrm{id}\ast\sigma$. If we can show 
        the conditions in Remark~\ref{rmk:topology-in-USC-bundles}, then it follows that there is such a unique topology.
    
    \smallskip
    
    For \ref{item:sections:usc}, first fix $\gamma =\mathrm{id}\ast\sigma.$ We have $\norm{\gamma (g,x)} = \norm{\sigma(x)}.$ As $\sigma$ is continuous and as $m\mapsto\norm{m}$ is \usc\ since $\cM$ is a \USCBb, $(g,x)\mapsto \norm{\gamma (g,x)}$ is indeed \usc. It follows that the same is true for any complex linear combination of such elements, i.e., any $\gamma \in\Gamma.$  
    
    For \ref{item:sections:dense}, we compute
    \begin{align*}
       \pi\inv (g,x)
       \supseteq
       \Gamma (g,x)
       &\supseteq
       \{(g,\sigma(x))\,\mid\,\sigma\in\Gamma_{0} ( X ;\cM)\}
       \\
       &\overset{(*)}{=}
       \{(g,m)\,\mid\,m\in\cM_x\}
       =
       \pi\inv (g,x),
    \end{align*}
    where $(*)$ follows since $\cM$ has enough continuous cross-sections.
    
    \smallskip
    
    It now remains to show that sets of the form $U_{1} \bfp{}{}  U_{2}$ form basic open subsets. From \cite[Proof of Theorem C.25]{Wil2007}, we know that basic open sets with respect to the newly constructed topology are of the form
    \begin{align*}  
       W(\mathrm{id}\ast\sigma,V,\epsilon):=\left\{
       ( g , m ) \in \cG \bfp{s}{t}  \cM \,\mid\,
       ( g ,q_{\cM}( m ))\in V, \norm{( g , m ) - ( \mathrm{id}\ast\sigma) ( g ,q_{\cM}( m ))} < \epsilon 
    \right\}
    \end{align*}
    for open $V\subseteq  \cG \bfp{s}{r }   X $, a fixed section $\sigma$ of $\cM$, and $\epsilon >0.$ As it suffices to take {\em basic} open sets in $\cG \bfp{s}{r }   X $, we may take $V:= U_{1} \bfp{}{} V'$ for some open $V'\subseteq  X $. We get
    \begin{align*}  
       W(\mathrm{id}\ast\sigma, U_{1} \bfp{}{}  V' ,\epsilon)
       =
       \left\{
           (h,\xi)\in \cG \bfp{s}{t}  \cM \,\mid\,
           h\in U_{1}, q_{\cM}(\xi)\in V', \norm{\xi - \sigma (q_{\cM}(\xi))} < \epsilon 
       \right\}
    \end{align*}
    But that set is exactly $U_{1} \bfp{}{}  U_{2}$ for $U_{2}:= W(\sigma, V' ,\epsilon)$, a basic open subset of~$\cM$. 
\end{proof}

\begin{definition}[{\cite[Definition 2.8]{BE2012}}]
An \usc\ Banach bundle $\cB=(B,p_{\cB})$ over a (\LCH\ \etale) groupoid $\cG$ is called {\em Fell bundle} 
if it comes with continuous maps
\[
    \cdot\;\colon
     \cB^{(2)}:=
     \left\{
        (a,b)\in B\times B  : (p_{\cB}(a),p_{\cB}(b))\in \cG^{(2)}
     \right\}
    \to
    B
    \quad\text{and}\quad
    {}^{\ast}\colon B \to B
\]
such that:
\begin{enumerate}[label=\textup{(F\arabic*)}]
\item\label{cond.F1} For each $(x,y)\in \cG^{(2)}$, $\cB_{x}\cdot \cB_y\subseteq\cB_{xy}$, i.e.\ $p_{\cB}(b\cdot c)=p_{\cB}(b) p_{\cB}(c)$ for all $(b,c)\in  \cB^{(2)}$.
\item\label{cond.F2} The multiplication is bilinear.
\item\label{cond.F3} The multiplication is associative, whenever it is defined.
\item\label{cond.F4} If $(b,c)\in  \cB^{(2)}$, then $\|b\cdot c\|\leq\|b\|\|c\|$, where the norm is the Banach norm of the respective fibre. 
\item\label{cond.F5} For any $x\in \cG$, $\cB_{x}^*\subseteq \cB_{x\inv }$.
\item\label{cond.F6} The involution map $b\mapsto b^*$ is conjugate linear.
\item\label{cond.F7} If $(b,c)\in  \cB^{(2)}$, then $(b\cdot c)^*=c^*\cdot  b^*$. 
\item\label{cond.F8} For any $b\in B$, $b^{**}=b$.
\item\label{cond.F9} For any $b\in B$, $\|b^*\cdot  b\|=\|b\|^2=\|b^*\|^2$. 
\item\label{cond.F10} For any $b\in B$, $b^*\cdot  b\geq 0$ in the fibre of $\cB$ over $s_{\cB}(b)$. 
\end{enumerate}
We call $\cB$ {\em saturated} if we have an equality of sets in Condition~\ref{cond.F1}. We will often write $bc$ for $b\cdot c$.
\end{definition}

\begin{definition}[{\cite{MW2008}}]\label{def:USCBb-action}
    Suppose that $\cB=(p_{\cB}\colon B\to \cG)$ is a Fell bundle over a (\LCH\ \etale) groupoid $\cG$, $X$ is a left $\cG$-space, and $\cM=(q_{\cM}\colon M \to X)$ is a \USCBb.  Then we say that {\em $\cB$
    acts on \textup(the left\textup) of $\cM$} if there is a continuous map 
    \(
      \cB\bfp{s}{r} \cM \to \cM,\,(b,m)\mapsto
    b\cdot m,\) such that
    \begin{enumerate}[leftmargin=2cm,label=\textup{(FA\arabic*)}]
    \item\label{item:FA:fibre} $q_{\cM}(b\cdot m)=p_{\cB}(b)\cdot q_{\cM}(m)$,
    \item\label{item:FA:assoc} $a\cdot (b\cdot m)=(ab)\cdot m$ for all appropriate $a\in B$, and
    \item\label{item:FA:norm} $\norm{b\cdot m}\leq\norm{b}\,\norm{m}$. 
    \end{enumerate}
\end{definition}
We point out that in \cite{MW2008}, Condition~\ref{item:FA:norm} had a typo: it should read $\;\leq\;$ instead of $\;=\;$.
An analogous definition can be made for a right action of a Fell bundle.

\begin{definition}[{\cite[Definition 6.1]{MW2008}}]\label{def:FBequivalence}
  Suppose that $\cG$, $\cH$ are \LCH\ \etale\ groupoids, that $\X$ is a $(\cG ,\cH )$-equivalence, that $\cB=(p_{\cG }\colon B\to \cG)$ and $\cC=(p_{\cH}\colon C\to \cH)$ are saturated Fell
  bundles, and that $\cM=(q_{\cM}\colon M\to \X)$ is a \USCBb.  We say that $\cM$ is a {\em $\cB\sme\cC$-equivalence}
  if the following conditions hold.
  \begin{enumerate}[leftmargin=2cm,label=\textup{(FE\arabic*)}]
  \item\label{item:FE:actions} There is a left $\cB$-action and a right $\cC$-action on $\cM$
    such that $b\cdot (m\cdot c)=(b\cdot m)\cdot c$ for all $b\in B $,
    $m\in M $, and $c\in C $, wherever it makes sense.
  \item\label{item:FE:ip} There are sesquilinear maps
  \[
  \begin{tikzcd}[row sep=tiny, column sep = small]
    \lip\cB<\cdot,\cdot>\colon&[-15pt] M \bfp{s}{s} M \ar[r]& B, && \rip\cC<\cdot,\cdot>\colon&[-15pt] M \bfp{r}{r} M \ar[r]& C
    \\
    &(m_{1},m_{2})\ar[r,mapsto]& \lip\cB<m_{1},m_{2}>, && &(m_{1},m_{2})\ar[r,mapsto]& \rip\cC<m_{1},m_{2}>
  \end{tikzcd}
  \]
  such that for all appropriately chosen $m_{i}\in M$, $b\in B$, and $c\in C$, we have
    \begin{enumerate}[leftmargin=1cm,label=\textup{(FE2.\alph*)}]
    \item\label{item:FE:ip:fibre}
      $p_{\cG }\bigl(\lip\cB<m_{1},m_{2}>\bigr)=\leoq{\cG}{q_{\cM}(m_{1})}{q_{\cM}(m_{2})}$\; and \;
      $p_{\cH}\bigl(\rip\cC<m_{1},m_{2}>\bigr)=\reoq{q_{\cM}(m_{1})}{q_{\cM}(m_{2})}{\cH}$,
    \item\label{item:FE:ip:adjoint} $\lip\cB<m_{1},m_{2}>^{*}=\lip\cB<m_{2},m_{1}>$\; and \;
      $\rip\cC<m_{1},m_{2}>^{*}=\rip\cC<m_{2},m_{1}>$,
    \item\label{item:FE:ip:C*linear} $\lip\cB<b\cdot m_{1},m_{2}>=b\lip\cB<m_{1},m_{2}>$\; and \;
      $\rip\cC<m_{1},m_{2}\cdot c>=\rip\cC<m_{1},m_{2}>c$, and
    \item\label{item:FE:ip:compatibility} $\lip\cB<m_{1},m_{2}>\cdot m_{3}=m_{1}\cdot\rip\cC<m_{2},m_{3}>$.
    \end{enumerate}
  \item\label{item:FE:SMEs} With the actions coming from \ref{item:FE:actions} and the inner products
    coming from \ref{item:FE:ip}, each $M(x)$ is a $B\bigl(r(x)\bigr)\sme
    C\bigl(s(x)\bigr)$-\ib.
  \end{enumerate}
\end{definition}

\begin{remark}[{\cite[Lemma 6.2]{MW2008}}]\label{rmk:MW-6.2}
    For $(g,x)\in \cG\bfp{s}{r} \X$, we have that $\cM_{g\cdot x}$ is isomorphic to $\cB_{g}\otimes_{\cB_{s(g)}} \cM_{x}$  as $\cB_{r(g)}-\cC_{s(x)}$-\ib s. We denote by $\cB_{g}\cdot \cM_{x}$ the dense subspace of $\cM_{g\cdot x}$ given by the image of $\cB_{g}\odot_{\cB_{s(g)}} \cM_{x}$ under that isomorphism.
    
    Similarly for $(x,h)\in \X\bfp{s}{r}\cH$,  $\cM_{x\cdot h}$ is isomorphic to $\cM_{x}\otimes_{\cC_{r(h)}} \cC_{h}$  as $\cB_{r(x)}-\cC_{s(h)}$-\ib s, and we let $\cM_{x}\cdot\cC_{h}$ be the corresponding dense subspace of $\cM_{x\cdot h}$.
\end{remark}

\section{Groupoid \Pe s}

In Section~\ref{sec:Product}, we will construct, out of Fell bundle equivalences, new \USCBb s over spaces that are not groupoid equivalences. We therefore need to be able to talk about spaces that are only `almost' equivalences, and so we start with the following definition.

\begin{definition}\label{def:preequiv}
   Two \LCH\ groupoids $\cG$ and $\cH$ are called {\em \peadjective} if there is a topological space $X$ 
  such that
\begin{enumerate}[leftmargin=1.7cm, label={(\peshort\arabic*)}]
	\item\label{item:def-preequiv:proper} $X$ is a proper left $\cG$- and proper right $\cH$-space,
	\item\label{item:def-preequiv:commute} the left and right action commute, and
	\item\label{item:def-preequiv:Leoq-Reoq} there exist  continuous surjective	maps
  \begin{align*}
      \Leoqempty{\cG}\colon X\bfp{s}{s}X \to \cG \quad\text{and}\quad  \Reoqempty{\cH}\colon X\bfp{r}{r}X\to \cH
  \end{align*}
  which satisfy the following for $x,x'\in X$, $g\in\cG, h\in\cH$ wherever it makes sense:
  \begin{enumerate}[label={(\peshort3.\alph*)}]
    \Item\label{item:def-preequiv:r-and-s}
         \begin{align*}
            r_{\cG}\left(\Leoq{\cG}{x}{x'} \right) = r_{X}(x)
            \quad\text{and}\quad
            s_{\cH}\left(\Reoq{x}{x'}{\cH} \right) = s_{X}(x'),
        \end{align*} 
    \Item\label{item:def-preequiv:inverse}
        \begin{align*}
            \Leoq{\cG}{x}{x'}\inv = \Leoq{\cG}{x'}{x}
            \quad\text{and}\quad
            \Reoq{x}{x'}{\cH}\inv = \Reoq{x'}{x}{\cH},
        \end{align*}
    \Item\label{item:def-preequiv:equivariant}
    \begin{align*}
        \Leoq{\cG}{g\cdot x}{x'} = 
        g\Leoq{\cG}{x}{x'}
        \quad\text{and}\quad 
        \Reoq{x}{x'\cdot h}{\cH} = 
        \Reoq{x}{x'}{\cH}h,
    \end{align*}
    \Item\label{item:def-preequiv:balancing}
    \begin{align*}
        \Leoq{\cG}{x\cdot h\inv}{x'}
        =
        \Leoq{\cG}{x}{x'\cdot h}
        \quad\text{and}\quad
        \Reoq{g\inv \cdot x}{x'}{\cH}
        =
        \Reoq{x}{g\cdot x'}{\cH},
    \end{align*}
    \Item\label{item:def-preequiv:transitive}
    \begin{align*}
        \Leoq{\cG}{x}{x'}\Leoq{\cG}{x'}{x''}
        &=
        \Leoq{\cG}{x}{x''}
        \quad\text{and}\quad
        \\
        \Reoq{x}{x'}{\cH}&\Reoq{x'}{x''}{\cH}
        =
        \Reoq{x}{x''}{\cH},
    \end{align*}
    \Item\label{item:def-preequiv:kernels}
    \begin{align*}
        \ker(\Leoq{\cG}{\cdot}{\cdot})
        &=
        \ker(\Reoq{\cdot}{\cdot}{\cH}).
    \end{align*}
    \end{enumerate}
	\end{enumerate}
		We then call $X$ a {\em $(\mathcal{G},\mathcal{H})$-\pe}. 
    Sometimes, we will equip the maps $\Leoqempty{\cG}$ and $\Reoqempty{\cH}$ with a superscript-$X$ to avoid ambiguity.
\end{definition}

One should think of $\Leoq{\cG}{\cdot}{\cdot}$ and $\Reoq{\cdot}{\cdot}{\cH}$ as weakened analogues of the maps $\leoq{\cG}{\cdot}{\cdot}$ and $\reoq{\cdot}{\cdot}{\cH}$, respectively, that a groupoid equivalence carries; however, we allow that
  \begin{align*}
      \Leoq{\cG}{x}{x'} \cdot x' \neq x
      \quad\text{ or }\quad 
      x\cdot \Reoq{x}{x'}{\cH} \neq x'.
  \end{align*}
 In fact, we will see in Corollary~\ref{cor:preequiv-almost-equivalence} that this is the only difference between a groupoid equivalence and a {groupoid \pe}.

\begin{remark}
    The axioms of a groupoid \pe\ have some easy consequences; here are the ones we will use frequently.
    \begin{description}
     \Item[\normalfont{
     \namedlabel{item:def-preequiv:s-and-r}
     {\ref*{item:def-preequiv:r-and-s}$'$}
     }]
        \begin{align*}s_{\cG}\left(\Leoq{\cG}{x}{x'} \right) = r_{X}(x')
            \quad\text{and}\quad
            r_{\cH}\left(\Reoq{x}{x'}{\cH} \right) = s_{X}(x),
        \end{align*}
    \Item[\normalfont{
     \namedlabel{item:def-preequiv:units}
     {\ref*{item:def-preequiv:inverse}$'$}
     }] 
    $$\Leoq{\cG}{x}{x}=r_{X}(x)\in \cG\z\quad\text{and}\quad\Reoq{x}{x}{\cH}=s_{X}(x)\in \cH\z.$$
    \end{description}
    In fact, \mbox{\ref{item:def-preequiv:s-and-r}} can be  proved using \mbox{\ref{item:def-preequiv:r-and-s}} and \mbox{\ref{item:def-preequiv:inverse}}, and it was implicitly used in \mbox{\ref{item:def-preequiv:transitive}}.
\end{remark}

\begin{example}\label{ex:equivalences-are-preequivs}
    Any $(\cG,\cH)$-equivalence $\X$ is a $(\cG,\cH)$-\pe\ if we let 
    \[
        \Leoq{\cG}{x}{x'}:=\leoq{\cG}{x}{x'}
        \quad\text{and}\quad 
        \Reoq{x}{x'}{\cH}:=\reoq{x}{x'}{\cH}.
    \]
\end{example}
\begin{proof}
	Conditons~\ref{item:def-preequiv:proper} and~\ref{item:def-preequiv:commute} hold by assumption (see Definition~\ref{def:GE}).
    We will only verify continuity and surjectivity of $\Leoqempty{\cG}$; all algebraic properties directly follow from the fact that the element $g=\leoq{\cG}{x}{x'}\in\cG$ {\em uniquely} satisfies $x=g\cdot x',$ and by symmetry, all properties then also follow for $\Reoqempty{\cG}$.
    
     For continuity of $\Leoqempty{\cG}$, suppose we are given a convergent net $(x_{\lambda},x'_{\lambda})\to (x,x')$ in $\X\bfp{s}{s}\X$. For each $\lambda$, $g_{\lambda}:=\leoq{\cG}{x_{\lambda}}{x'_{\lambda}}\in \cG$ satisfies $x_{\lambda}=g_{\lambda}\cdot x'_{\lambda},$ just as $g=\leoq{\cG}{x}{x'}\in\cG$ satisfies $x=g\cdot x'.$ Thus,
    \[
        (g_{\lambda}\cdot x'_{\lambda}, x'_{\lambda}) = (x_{\lambda},x'_{\lambda}) \to (x,x') = (g\cdot x', x').
    \]
    Since the actions on an equivalence are proper (see \ref{item:def-GA:proper} in Definition~\ref{df.left.action}), this implies that a subnet of
    \[
        \Leoq{\cG}{x_{\lambda}}{x_{\lambda}'}=\leoq{\cG}{x_{\lambda}}{x_{\lambda}'}= g_{\lambda}
        \text{ converges to }
        g = \leoq{\cG}{x}{x'}=\Leoq{\cG}{x}{x'}.
    \]
    This suffices to conclude that $\Leoqempty{\cG}$ is continuous.
\end{proof}

\begin{lemma}\label{lem:product-of-preequiv}
    Suppose $X$ is a $(\cG,\cH)$-\pe\ and $Y$ is an $(\cH,\mathcal{K})$-\pe.
    Then $ Z :=X\bfp{s}{r}Y$ is a $(\cG,\cH)$-\pe: its momentum maps are given by $r_{ Z }(x,y)=r_{X}(x)$ and $s_{ Z }(x,y)=s_{Y}(y)$; the actions are defined by $g\cdot(x,y)=(g\cdot x,y)$ and $(x,y)\cdot k = (x,y\cdot k)$; it carries the maps
    \[
        \Leoq[ Z ]{\cG}{(x,y)}{(x',y')}:=\Leoq[X]{\cG}{x}{x'\cdot \Leoq[Y]{\cH}{y'}{y}}
        \quad\text{and}\quad 
        \Reoq[ Z ]{(x,y)}{(x',y')}{\mathcal{K}}:=\Reoq[Y]{\Reoq[X]{x'}{x}{\cH}\cdot y}{y'}{\mathcal{K}}.
    \]
\end{lemma}

\begin{proof}
Continuity of $\Leoqempty[Z]{\cG}$ follows from continuity of $\Leoqempty[Y]{\cH}$, of the right $\cH$-action on $X$, and of $\Leoqempty[X]{\cG}.$ Surjectivity follows from surjectivity of both $\Leoqempty[Y]{\cH}$ and of $\Leoqempty[X]{\cG}.$ The algebraic properties are again easy to verify, using the corresponding properties of $\Leoqempty[Y]{\cH}$ and  $\Leoqempty[X]{\cG}.$ The same is true for $\Reoqempty[Y]{\mathcal{K}}.$
\end{proof}

\begin{remark}\label{rm.product.is.not.eq}
If $\X$ and $\Y$ are {\em equivalences} with $\Leoqempty[\X]{\cG}=\leoq[\X]{\cG}{\cdot}{\cdot}$ etc.\ as in Example~\ref{ex:equivalences-are-preequivs}, then $ Z = \X\bfp{s}{r}\Y$ is {\em not} automatically an equivalence: if we let  
$$g:=\Leoq[ Z ]{\cG}{(x,y)}{(x',y')} = \leoq[\X]{\cG}{x}{x'\cdot  \leoq[\Y]{\cH}{y'}{y}},$$ then the element 
$g\cdot (x',y')=(g\cdot x', y')$ does not necessarily equal $(x,y)$.
This is exactly our motivation for the definition of groupoid \pe.

However, we do have that the pair $((g\cdot x', y'),(x,y))$  belongs to the kernel of $\Leoqempty[Z]{\cG}$: since $\X$ is a groupoid equivalence, we have $\leoq[\X]{\cG}{x}{x'}\cdot x'=x$, so
\begin{alignat*}{2}
    g\cdot x' &= \leoq[\X]{\cG}{x}{x'\cdot  \leoq[\Y]{\cH}{y'}{y}} \cdot x' 
    \\
    &= \leoq[\X]{\cG}{x\cdot \leoq[\Y]{\cH}{y'}{y}\inv }{x'} \cdot x' 
    = x\cdot \leoq[\Y]{\cH}{y}{y'}.
\end{alignat*}
Thus, we see that 
\[
    \Leoq[ Z ]{\cG}{(g\cdot x',y')}{(x,y)} = 
    \leoq[\X]{\cG}{g\cdot x'}{x\cdot \leoq[\Y]{\cH}{y}{y'}}
    =
    \leoq[\X]{\cG}{g\cdot x'}{g\cdot x'}
\]
is an element of $\cG\z$ by Condition~\mbox{\ref{item:def-preequiv:units}}.
\end{remark}

The remark motivates the following:

\begin{lemma}\label{lem:preequiv:taking-the-quotient}
    If $X$ is a $(\cG,\cH)$-\pe, then we define the relation $\Rel$ on $X$ by
    \[
        x \Rel  x' \text{ if } (x,x')\in \ker(\Leoq{\cG}{\cdot}{\cdot}).
    \]
    This defines a closed equivalence relation; we let  $\mathcal{X}$ denote the quotient of $X$ by $\Rel$. The quotient map $\pi\colon X\to \mathcal{X}$ is open, and so $\mathcal{X}$ is \LCH. 
    For $x\in X$, we will sometimes write $\tilde{x}$ for $\pi(x).$
\end{lemma}
\begin{proof}~ 
    Reflexivity, symmetry, and transitivity of $\Rel$ follow easily from \mbox{\ref{item:def-preequiv:units}}, \mbox{\ref{item:def-preequiv:inverse}}, and \mbox{\ref{item:def-preequiv:transitive}} respectively. 
    To show that $\Rel$ is closed, notice first that the domain $X\bfp{s}{s}X$ of $\Leoq{\cG}{\cdot}{\cdot}$ is closed in $X\times X$, so it suffices to check that $\Rel $ is closed in $X\bfp{s}{s}X$. If $(x_{\lambda}, x'_{\lambda})\to (x,x')$ and $x_{\lambda}\Rel  x'_{\lambda}$, then by continuity of $\Leoq{\cG}{\cdot}{\cdot}$ and since $\cG\z$ is closed in $\cG$,
    \[
        \Leoq{\cG}{x}{x'} =  \lim_{\lambda} \Leoq{\cG}{x_{\lambda}}{x'_{\lambda}} \in \cG\z,
    \]
    showing that $x\Rel  x'$ also.
    
    To see that the quotient map $\pi$ is open, note that $\Leoqempty{\cG}{}\inv(\cG\z)$ is open in $X\bfp{s}{s}X$  since $\cG$ is \etale\ and since $\Leoqempty{\cG}$ is continuous. By definition of the subspace topology, there thus exists an open set $V\subseteq X\times X$ such that
    \(
        \Leoqempty{\cG}{}\inv(\cG\z) = V \cap (X\bfp{s}{s}X).
    \) Now, suppose  $U\subseteq X$ is open, so that $V\cap (U\times X)$ is open in $X\times X$ and hence $V\cap (U\bfp{s}{s} X) = (V\cap (U\times X)) \cap X\bfp{s}{s}X$ is open in $X\bfp{s}{s}X$. We compute
    \begin{align*}
        \mathrm{pr}_{2} (V\cap (U\bfp{s}{s} X))
        &=
        \{
            x\in X \,\mid\,
            \exists x'\in U\text{ such that } s_{X}(x')=s_{X}(x) \quad\text{and}\quad (x',x)\in V
        \}\\
        &=
        \{
            x\in X \,\mid\,
            \exists x'\in U\text{ such that } s_{X}(x')=s_{X}(x) \quad\text{and}\quad \Leoq{\cG}{x'}{x}\in \cG\z
        \}\\
        &=
        \{
            x\in X \,\mid\,
            \exists x'\in U\text{ such that } \pi(x')=\pi(x)
        \}
        =
        \pi\inv (\pi(U)).
    \end{align*}
    Since $s_{X}$ is open and continuous, the map $\mathrm{pr}_{2}\colon X\bfp{s}{s}X\to X$ is open by Lemma~\ref{lemma:proj-open-from-bfp}, so that $\mathrm{pr}_{2} (V\cap (U\bfp{s}{s} X)) = \pi\inv(\pi(U))$ is open as claimed.
    
    Since $X$ is \LCH\ and $\pi$ is open, it follows immediately that $\mathcal{X}$ is also \LCH.
\end{proof}

\begin{proposition}\label{prop:from-preequiv-to-equivalence}
    Suppose $X$ is a $(\cG,\cH)$-\pe\ and $\mathcal{X}$ the quotient constructed in Lemma~\ref{lem:preequiv:taking-the-quotient}. Then $\mathcal{X}$ is an equivalence of groupoids if we define
    \begin{align*}
        \tilde{r}\colon \mathcal{X}\to \cG\z, \; \tilde{r}( \tilde{x}):= r_X(x);&\qquad 
        g\cdot 
        \tilde{x}:= \pi(g\cdot x)
        \\
        \tilde{s}\colon \mathcal{X}\to \cH\z, \; \tilde{s}( \tilde{x}):= s_X(x); &\qquad
        \tilde{x} \cdot h:= \pi(x\cdot h).
    \end{align*}
    Furthermore, 
    for all $(x,x')\in X\bfp{s}{s}X$ and all $(y,y')\in X\bfp{r}{r}X$, we have
    \[
        \leoq[\mathcal{X}]{\cG}{\pi(x)}{\pi(x')} = \Leoq[X]{\cG}{x}{x'}
        \quad\text{and}\quad
        \reoq[\mathcal{X}]{\pi(y)}{\pi(y')}{\cH}=\Reoq[X]{y}{y'}{\cH}.
    \]
\end{proposition}

\begin{proof}
    We will prove all things only for the left side; the right side will follow by symmetry and since the actions clearly commute. So let us first check that the \LCH\ space $\mathcal{X}=X/\Rel$ is a left $\cG$-space.
    
    To see that the action is well defined, suppose $\tilde{x}=\tilde{y}$, i.e., $\Leoq[X]{\cG}{x}{y}\in\cG\z.$ Then 
        \[
            \Leoq[X]{\cG}{g\cdot x}{g\cdot y} = g\,\Leoq[X]{\cG}{x}{y}\,g\inv = gg\inv \in \cG\z,
        \]
        which shows that $\pi(g\cdot x)=\pi(g\cdot y).$ 
        
        Since $r_X\colon X\to \cG\z$ is continuous, open, and surjective, and since $\pi$ is open by Lemma~\ref{lem:preequiv:taking-the-quotient}, we have that $\tilde{r}$ is continuous, open, and surjective. 
        
        Next, let $f_0\colon \cG\bfp{s}{r}\mathcal{X}\to \mathcal{X}$ denote the newly defined action. To show that $f_0$ is continuous, suppose $U_0\subseteq \mathcal{X}$ is open, i.e., $U:=\pi\inv (U_0)$ is open in $X$.
        As the action $f\colon \cG\bfp{s}{r}X\to X$ on $X$ is continuous, $f\inv (U)$ is open in $\cG\bfp{s}{r}X$, so we may write it as $\bigcup_{i\in I} V_{i}\ast U_{i}$ for some open sets $V_{i}\subseteq \cG$ and $U_{i}\subseteq X$. But then
        \[
            f_0\inv (U_0) =  (\mathrm{id}\ast\pi) (f\inv (\pi\inv(U_0))=  (\mathrm{id}\ast\pi) \Bigl(\bigcup_{i\in I} V_{i}\ast U_{i}\Bigr) = \bigcup_{i\in I} V_{i}\ast \pi(U_{i}),
        \]
        which is open in $\cG\bfp{s}{r}\mathcal{X}$ since $\pi$ is an open map by Lemma~\ref{lem:preequiv:taking-the-quotient}.
        
        \smallskip 
        
        We now verify that this action satisfies all the axioms listed in Definition~\ref{df.left.action}. 
        The items \mbox{\ref{item:def-GA:range}}, \mbox{\ref{item:def-GA:assoc}}, \mbox{\ref{item:def-GA:unit}} follow immediately from the corresponding property of the $\cG$-action on $X$. Now, if $g\cdot\tilde{x}=\tilde{x}$, then $\Leoq[X]{\cG}{g\cdot x}{x}\in\cG\z$ by definition of $\Rel$, so by {\mbox{\ref{item:def-preequiv:equivariant}}}
        and \mbox{\ref{item:def-preequiv:units}}, that means $g\in \cG\z$; in other words, the action on $\mathcal{X}$ is free.
        To see that the action is proper, it thus suffices to check that
        $\Phi_0\colon \cG\bfp{s}{r}\mathcal{X}\to \mathcal{X}\times \mathcal{X}, (g,\tilde{x})\mapsto (g\cdot \tilde{x}, \tilde{x})$, is a closed map.
        For any closed $A    \subseteq \cG\bfp{s}{r}\mathcal{X}$, we must show that 
        \begin{align*}
            (\pi\times\pi)\inv (\Phi_0(A))
            &=
            \left\{
                (y,x)\in X\times X \,\mid\,
                \exists g\in \cG, \tilde{y}=\pi(g\cdot x) \text{ and } (g,\tilde{x})\in A
            \right\}
        \end{align*}
        is closed in $X\times X$. So suppose that $\{(y_{\lambda},x_{\lambda})\}_{\lambda}$ is a net in $(\pi\times\pi)\inv (\Phi_0(A))$ that converges to $(y,x)$ in $X\times X.$ 
        By the above description of the preimage, it follows that, for every $\lambda$, there exists $g_{\lambda}\in \cG$ such that $\tilde{y}_{\lambda}=\pi(g_{\lambda}\cdot x_{\lambda})$ and $ (g_{\lambda},\tilde{x}_{\lambda})\in A$. The equality in $\mathcal{X}$ implies
        \(
            \Leoq[X]{\cG}{g_{\lambda}\cdot x_{\lambda}}{y_{\lambda}}
            \in \cG\z.
        \)
        By Conditions~\mbox{\ref{item:def-preequiv:equivariant}} and \mbox{\ref{item:def-preequiv:inverse}}, that means
        \[
            g_{\lambda} = \Leoq[X]{\cG}{x_{\lambda}}{y_{\lambda}}\inv = \Leoq[X]{\cG}{y_{\lambda}}{x_{\lambda}}.
        \]
        By continuity of $\Leoq[X]{\cG}{\cdot}{\cdot}$, we conclude that $g_{\lambda}\to \Leoq[X]{\cG}{y}{x}=:g$. As $x_{\lambda}\to x$ by assumption, it follows that the net $\{(g_{\lambda}, \tilde{x}_{\lambda})\}_{\lambda}$ converges to $(g,\tilde{x})$. 
        Since $A$ contains the net and is closed, we conclude 
        $(g,\tilde{x})\in A$. By definition of $g$ and using Conditions~\mbox{\ref{item:def-preequiv:equivariant}} and \mbox{\ref{item:def-preequiv:inverse}}, we see
		\begin{align*}
			\Leoq[X]{\cG}{g\cdot x}{y}
			=
			g\cdot \Leoq[X]{\cG}{x}{y}
			=
			gg\inv \in \cG\z,
		\end{align*}
        so that $\pi(g\cdot x)=\tilde{y}$ by definition of $\Rel$. In particular, the limit $(g\cdot x,x)=(y,x)$ of $\{(y_{\lambda},x_{\lambda})\}_{\lambda}$ is an element of $(\pi\times\pi)\inv (\Phi_0(A)).$ This concludes our proof that this set is closed.
        
        \smallskip 
        
        Lastly, we need to check that the map $r\colon \mathcal{X}/\cH \to \cG\z, \tilde{x}\cdot \cH \mapsto r_{X}(x),$ is well defined, injective, continuous, and open. We let $q\colon \mathcal{X}\to\mathcal{X}/\cH$ denote the quotient map.
        Suppose $\tilde{x}\cdot \cH=\tilde{y}\cdot \cH$, so there exists $h\in\cH$ such that $v:=\Leoq[X]{\cG}{x}{y\cdot h}\in\cG\z$. Then $v=r_{\cG}(v)=r_{X}(x)$ and $v=s_{\cG}(v)=r_{X}(y\cdot h)=r_{X}(y)$, showing that $r_{X}(x)=r_{X}(y)$ and thus $r$ is well defined on $\mathcal{X}/\cH$.
        
        To see that $r$ is injective, suppose $r_{X}(x)=r_{X}(y)$ and let $h:= \Reoq[X]{y}{x}{\cH}$. Then using Conditions~\mbox{\ref{item:def-preequiv:s-and-r}} and \mbox{\ref{item:def-preequiv:transitive}}, we have
        \[ 
        \Reoq[X]{x}{y\cdot h }{\cH}=\Reoq[X]{x}{y}{\cH} h = \Reoq[X]{x}{x}{\cH} \in \cH\z,
        \]
        showing that $(x,y\cdot h)\in \ker (\Reoqempty[X]{\cH})= \ker (\Leoqempty[X]{\cG}).$ In particular, $\tilde{x}=\pi(x)=\pi(y\cdot h)= \pi(y)\cdot h = \tilde{y}\cdot h$. That means that $\tilde{x}\cdot \cH =\tilde{y}\cdot \cH$, as claimed.

        For continuity, take $U\subseteq \cG\z$. Then
        \begin{align*}
            r\inv (U) \text{ is open in } \mathcal{X}/\cH
            &
            \iff 
            q\inv( r\inv (U))\text{ is open in } \mathcal{X}
            \\&
            \iff 
            \pi\inv(q\inv (r\inv (U)))\text{ is open in } X.
        \end{align*}
        But note that
        \begin{align*}
            \pi\inv(q\inv (r\inv (U)))
            &=
            \{
                x\in X
                \,\mid\,
                r(q(\pi(x)))\in U
            \}
            =
            \{
                x\in X
                \,\mid\,
                r_{X}(x)\in U
            \}
            =
            r_{X}\inv (U),
        \end{align*}
        which, if $U$ is open, is open by continuity of $r_{X}$.

        Finally, to see that $r$ is an open map, suppose $V\subseteq  \mathcal{X}/\cH$ is open, i.e., $q\inv (V)$ is open in $\mathcal{X}$, i.e., $\pi\inv(q\inv (V))$ is open in $X$. Since $r_{X}$ is an open map, $r_{X}(\pi\inv(q\inv (V)))$ is open in $\cG\z$. We compute:
        \begin{align*}
            r_{X}(\pi\inv(q\inv (V)))
            &=
            \{
                u\in\cG\z\,\mid\,
                \exists x\in X\text{ such that } q(\pi(x))\in V \text{ and } r_{X}(x)=u
            \}\\
            &=
            \{
                u\in\cG\z\,\mid\,
                \exists \tilde{x}\in \mathcal{X}\text{ such that } q(\tilde{x})\in V \text{ and } \tilde{r}(\tilde{x})=u
            \}\\
            &=
            \{
                u\in\cG\z\,\mid\,
                \exists \tilde{x}\cdot\cH\in \mathcal{X}/\cH\text{ such that } \tilde{x}\cdot\cH\in V \text{ and } r(\tilde{x}\cdot\cH)=u
            \}
            =
            r (V).
        \end{align*}
    We conclude that $\mathcal{X}$ is an equivalence between the groupoids $\cG$ and~$\cH$. 

    \smallskip
    
    It remains to show that \(
        \leoq[\mathcal{X}]{\cG}{\tilde{x}}{\tilde{y}} = \Leoq[X]{\cG}{x}{y}
    \). We have that
    \begin{align*}
        \Leoq[X]{\cG}{\Leoq[X]{\cG}{x}{y} \cdot y}{x}
        =
        \Leoq[X]{\cG}{x}{y} \,
        \Leoq[X]{\cG}{y}{x}
        =
        r_{\cG} \left(\Leoq[X]{\cG}{x}{y}\right)
        \in\cG\z,
    \end{align*}
    which means that $\Leoq[X]{\cG}{x}{y} \cdot y$ is related to $x$ via $\Rel$. In other words,
    \[ 
        \Leoq[X]{\cG}{x}{y} \cdot \tilde{y} = \pi\left(\Leoq[X]{\cG}{x}{y} \cdot y\right) = \tilde{x}.
    \]
    This means that $\Leoq[X]{\cG}{x}{y}$ satisfies the property that uniquely determines $\leoq[\mathcal{X}]{\cG}{\tilde{x}}{\tilde{y}}$.
\end{proof}

\begin{remark}
    It is clear that a groupoid equivalence is also a \pe.  Proposition~\ref{prop:from-preequiv-to-equivalence} proves 
    that any \pe\ gives rise to an equivalence, 
    and therefore, two groupoids are equivalent if and only if they allow a groupoid \pe. 
\end{remark}

\begin{corollary}[cf.\ \cite{MRW:Grpd}]
\label{cor:bfp-of-gpd} 
    If $\X$ is a $(\cG,\cH)$-equivalence and $\Y$ is an $(\cH,\mathcal{K})$-equivalence, then their balanced fibre product $\X\bfp{}{\cH}\Y$ is a $(\cG,\mathcal{K})$-equivalence and the quotient map $[\cdot]_{\ipscriptstyle\cH}\colon X\bfp{s}{r}Y\to X\bfp{}{\cH}Y$ is open.
\end{corollary}

\begin{proof}
    We have seen in Example~\ref{ex:equivalences-are-preequivs} that $\X$ and $\Y$ are groupoid \pe s, so by
     Lemma~\ref{lem:product-of-preequiv},  $ Z := \X\bfp{s}{r}\Y$ is a $(\cG,\mathcal{K})$-\pe.
    We claim that $\Rel$ is exactly the equivalence relation that gives rise to $\X\bfp{}{\cH}\Y$
    , so let us unravel what exactly $\Rel$ means here:
    \begin{align*}
		(x,y) \Rel  (x',y') 
		&\iff
        ((x,y),(x',y'))\in \ker(\Leoq[ Z ]{\cG}{\cdot}{\cdot})
        \\&\iff
        \leoq[\X]{\cG}{x}{x'\cdot \leoq[\Y]{\cH}{y'}{y}}
        \in\cG\z
        \\
        &\iff
        x = x'\cdot \leoq[\Y]{\cH}{y'}{y}
        \\
        &\iff
        \exists h \in\cH, \,
        x = x'\cdot h
        \text{ and }
        y'=h\cdot y
        \\
        &\iff
        \exists h \in\cH, \,
        (x,y) = (x'\cdot h, h\inv \cdot y')
        \\&\iff
        [x,y]_{\ipscriptstyle\cH} = [x',y']_{\ipscriptstyle\cH} \text{ in } \X\bfp{}{\cH}\Y.
    \end{align*}
    We conclude that $\X\bfp{}{\cH}\Y$ is exactly $(\X\bfp{s}{r}\Y) / \Rel$, which is a $(\cG,\mathcal{K})$-equivalence by  Proposition~\ref{prop:from-preequiv-to-equivalence}.
\end{proof}

\begin{corollary}\label{cor:preequiv-almost-equivalence}
Suppose $X$ is a $(\cG,\cH)$-\pe. Then the following  are equivalent:
\begin{enumerate}[label=\textup{(\arabic*)}]
    \item\label{item:preequiv:L-nondeg} For all $x,x'\in X $ such that $ s_{X}(x)=s_{X}(x'),$  we have $ \Leoq{\cG}{x}{x'} \cdot x' = x$.
    \item\label{item:preequiv:R-nondeg} For all $ y,y'\in X $ such that $r_{X}(y)=r_{X}(y'), $ we have $y\cdot \Reoq{y}{y'}{\cH} = y' $.
    \item\label{item:preequiv:X-equiv} $X$ is a groupoid equivalence for which $\leoq{\cG}{\cdot}{\cdot}=\Leoq{\cG}{\cdot}{\cdot}$ and $\reoq{\cdot}{\cdot}{\cH}=\Reoq{\cdot}{\cdot}{\cH}$.
\end{enumerate}
\end{corollary}

\begin{proof}
    Clearly, \ref{item:preequiv:X-equiv} implies  \ref{item:preequiv:L-nondeg} and \ref{item:preequiv:R-nondeg}. By symmetry, it now suffices to prove that \ref{item:preequiv:L-nondeg} implies \ref{item:preequiv:X-equiv}:
    
    Suppose $x\Rel y$ with $\Rel$ as defined in Lemma~\ref{lem:preequiv:taking-the-quotient}, so $\Leoq{\cG}{x}{y}\in \cG\z$. Then in particular, $\Leoq{\cG}{x}{y}=s_{\cG}(\Leoq{\cG}{x}{y})= r_{X}(y)$. It follows from our assumption \ref{item:preequiv:L-nondeg} that
    \(
        y = \Leoq{\cG}{x}{y} \cdot y = x.
    \)
    In other words, $\Rel$ is the diagonal in $X\times X$.     Thus, $X$ equals $\mathcal{X},$ which we have shown in Proposition~\ref{prop:from-preequiv-to-equivalence} to be an equivalence such that $\leoq{\cG}{\cdot}{\cdot}=\Leoq{\cG}{\cdot}{\cdot}$ and $\reoq{\cdot}{\cdot}{\cH}=\Reoq{\cdot}{\cdot}{\cH}$.
\end{proof}

\section{Fell Bundle \Hequiv s}
Just like we weakened the notion of groupoid equivalences to account for the fact that their fibre product is not another equivalence, we need to weaken the notion of Fell bundle equivalence as follows.

\begin{definition}\label{def:FBword}
  Suppose that $X$ is a $(\cG,\cH)$-\pe, that $\cB=(p_{\cG }\colon B\to \cG)$ and $\cC=(p_{\cH}\colon C\to \cH)$ are saturated Fell
  bundles, and that $\cM=(q_{\cM}\colon M\to X)$ is a \USCBb. We say that $\cM$ is a {\em $\cB\sme\cC$-\hequiv}
  if it satisfies the conditions in Definition~\ref{def:FBequivalence}, only that Conditions~\ref{item:FE:ip:fibre} and~\ref{item:FE:ip:compatibility} 
  are replaced by the following:
  \begin{description}
     \item[{
     \namedlabel{item:FHE:ip-fibre}
     {\ref*{item:FE:ip:fibre}${}_{0}$}
     }]
     $p_{\cG }\bigl(\lip\cB<m_{1},m_{2}>\bigr)=\Leoq{\cG}{q_{\cM}(m_{1})}{q_{\cM}(m_{2})}$\; and \;
      $p_{\cH}\bigl(\rip\cC<m_{1},m_{2}>\bigr)=\Reoq{q_{\cM}(m_{1})}{q_{\cM}(m_{2})}{\cH}$.
     \item[{
     \namedlabel{item:FHE:ip:adjointables}
     {\ref*{item:FE:ip:compatibility}${}_{0}$}
     }]
     $\linner{\cB}{m_{1}\cdot c}{m_{2}}
       =
       \linner{\cB}{m_{1}}{m_{2}\cdot c^*}$ and $
       \rinner{\cC}{b\cdot m_{1}}{m_{2}}
       =
       \rinner{\cC}{m_{1}}{b^*\cdot m_{2}}$
    \end{description}
    for all appropriately chosen $m_{i}\in M$, $c\in C$, and $b\in B.$ We furthermore require the following condition: 
    \begin{description}
    \item[\namedlabel{item:FHE:2e}{(FE2.e)}] For any $(m_1,m_2),(m_3,m_4)\in  M \bfp{s}{s} M$ with $q_\cM(m_2)=q_\cM(m_3)$, we have 
    \begin{align*}
        \lip\cB<m_{1},m_{2}> \cdot \lip\cB<m_{3},m_{4}> &= \lip\cB<m_{1}\cdot {\rip\cC<m_{2},m_{3}>},m_{4}> , \text{ and}       \\
        \rip\cC<m_{1},m_{2}> \cdot \rip\cC<m_{3},m_{4}> &= \rip\cC<m_{1},  {\lip\cB<m_{2},m_{3}>}\cdot m_{4}>.
    \end{align*}
    \end{description}
\end{definition}

\begin{remark}
    Condition~\ref{item:FHE:2e} is the Fell bundle analogue of Condition~\ref{item:def-preequiv:transitive} for groupoid {\pe s}. Indeed, by Condition~\ref{item:FHE:ip-fibre}, we have 
    \[
    p_{\cG }\bigl(\lip\cB<m_{i},m_{j}>\bigr)= \Leoq{\cG}{q_{\cM}(m_{i})}{q_{\cM}(m_{j})}.
    \]
    Since $q_\cM(m_2)=q_\cM(m_3)$, we have by Condition~\ref{item:def-preequiv:transitive} that
    \begin{align*}
        p_{\cG }\bigl(\lip\cB<m_{1},m_{2}>\lip\cB<m_{3},m_{4}>\bigr) &= \Leoq{\cG}{q_{\cM}(m_{1})}{q_{\cM}(m_{2})} \Leoq{\cG}{q_{\cM}(m_{3})}{q_{\cM}(m_{4})} \\
        &=\Leoq{\cG}{q_{\cM}(m_{1})}{q_{\cM}(m_{4})}
        .
    \end{align*}
    On the other hand, by Conditions~\ref{item:FA:fibre} and~\ref{item:FHE:ip-fibre},
    \[
    q_\cM(m_{1}\cdot {\rip\cC<m_{2},m_{3}>})
    =
    q_\cM(m_{1})\cdot p_{\cC}(\rip\cC<m_{2},m_{3}>)
    =
    q_\cM(m_1)\cdot \Reoq{q_{\cM}(m_{2})}{q_{\cM}(m_{3})}{\cH}
    .\]
    Again, since $q_\cM(m_2)=q_\cM(m_3)$,  we have
    by Condition~\ref{item:def-preequiv:units} that
    $
    \Reoq{q_{\cM}(m_{2})}{q_{\cM}(m_{3})}{\cH}
    =s_\cM(m_2)
    ,$
    which equals $s_\cM(m_1)$ since $(m_{1},m_{2})\in M\bfp{s}{s}M$.
    Therefore, $q_\cM(m_{1}\cdot {\rip\cC<m_{2},m_{3}>}) = q_\cM(m_1)$, and thus by Condition~\ref{item:FHE:ip-fibre}
    \[
    p_\cB(\lip\cB<m_{1}\cdot {\rip\cC<m_{2},m_{3}>},m_{4}>)=\Leoq{\cG}{q_{\cM}(m_{1})}{q_{\cM}(m_{4})},\]
    so the equalities in Condition~\ref{item:FHE:2e} make sense.
\end{remark}

\begin{remark}\label{rmk:MW-6.2-for-words}
    Remark~\ref{rmk:MW-6.2} still holds in the case of \hequiv s.
    Note further that any Fell bundle equivalence is a \hequiv. Indeed, Condition~\ref{item:FHE:ip-fibre} is exactly Condition~\ref{item:FE:ip:fibre} in the case that $X$ and $\cM$ are equivalences of groupoids resp.\ Fell bundles, with $\Leoq[X]{\cG}{\cdot}{\cdot}=\leoq[X]{\cG}{\cdot}{\cdot}$ and $\Reoq[X]{\cdot}{\cdot}{\cH}=\reoq[X]{\cdot}{\cdot}{\cH}$. 
    Moreover, we will see in Corollary~\ref{cor:adjointable} that
    Condition~\ref{item:FHE:ip:adjointables} is weaker than \ref{item:FE:ip:compatibility}.
    Lastly, Condition~\ref{item:FHE:2e} follows as well because
    \begin{alignat*}{2}
        \lip\cB<m_{1},m_{2}> \cdot \lip\cB<m_{3},m_{4}> 
        &= \lip\cB<{\lip\cB<m_{1},m_{2}>} \cdot m_{3},m_{4}>
        \quad &&\text{by Condition~\ref{item:FE:ip:C*linear}}
        \\
        &=\lip\cB<m_{1}\cdot {\rip\cC<m_{2},m_{3}>},m_{4}>
        &&
        \text{by \ref{item:FE:ip:compatibility}}.
    \end{alignat*}
\end{remark}

We will use this section to establish some basic results about \hequiv s. In the next section, we will then prove our first main result, namely that we can take the product of two \hequiv s $\cM$ and $\cN$ to construct another \hequiv; see Theorem~\ref{thm:cK-is-word} for the exact statement. This result is analogous to Lemma~\ref{lem:product-of-preequiv} in the land of groupoids.

Our second main result will then show that, when the original data $\cM,\cN$ are actually {\em equivalences}, their product gives rise to another {\em equivalence} by taking an appropriate quotient of the product \hequiv; see Theorem~\ref{thm:cP-is-equiv}. This result, in turn, is analogous to Corollary~\ref{cor:bfp-of-gpd} in groupoid land.

One groupoid result that currently does not have an analogue is Proposition~\ref{prop:from-preequiv-to-equivalence} which states that {\em any} groupoid \pe\ gives rise to an equivalence. The authors are convinced that \hequiv\ is not quite strong enough to force the existence of a Fell bundle equivalence, which is the reason for the chosen terminology.

\medskip
Let us  first proof some basic results. For the remainder of this section unless otherwise indicated, $\cM=(q_{\cM}\colon M\to X)$ will denote a fixed but arbitrary $(\cB,\cC)$-\hequiv.

\begin{lemma}\label{lem:non-trivial}
    For any $b\in \cB$ with $\norm{b}\neq 0$, there exists $m\in M $ with $r_{\cM}(m)=s_{\cB}(b)$ and $\norm{m}\neq 0$ in the Banach space norm of $\cM_x$ for $x=q_{\cM}(m)$.
\end{lemma}

\begin{proof}
     Let $g:= p_{\cB}(b) \in\cG$. By assumption, $r_{X}\colon X\to \cG\z$ is onto, so there exist $x\in X$ such that $r_{X} (x)=s_{\cG}(g)$. Now, by Condition~\ref{item:FE:SMEs} on $\cM$, we know in particular that $M(x)$ is a full left Hilbert module over
     \[
       B(r_{X} (x)) = B(s_{\cG}(g)) = B(s_{\cB}(b)).
     \]
     By fullness, it follows that $\mathrm{span}\left\{\linner{\cB}{m}{m'}: m, m' \in M(x)\right\}$ is dense in $B(s_{\cB}(b))$.
     
     By assumption on $b$, we have $0 < \epsilon := \norm{b}^2= \norm{b^*b}.$ Since $b^*b\in B(s_{\cB}(b))$,  there exist $m_{i}, m'_{i} \in M(x)$ with $\norm{b^*b - \sum_{i=1}^{n} \linner{\cB}{m_{i}}{m'_{i}}} < \frac{\epsilon}{2}$. In particular, $r_{\cM}(m) = r_{X}(q_{\cM}(m)) = r_{X} (x) = s_{\cB} (b)$, and further $\norm{\sum_{i=1}^{n} \linner{\cB}{m_{i}}{m'_{i}}} > \frac{\epsilon}{2} $ by the reverse triangle inequality. In particular, for some $1\leq i\leq n,$ we must have $\norm{\linner{\cB}{m_{i}}{m'_{i}}} > \frac{\epsilon}{2n}>0. $ By the Cauchy--Schwarz inequality for Hilbert modules \cite[Lemma 2.5]{RW:Morita}, we have
     \begin{align*}
        \linner{\cB}{m_{i}}{m'_{i}}^* \,\linner{\cB}{m_{i}}{m'_{i}} \leq \norm{\linner{\cB}{m_{i}}{m_{i}}} \,\linner{\cB}{m_{i}'}{m_{i}'}.
     \end{align*}
     Recall that, if $a, b$ in some $\textrm{C}^*$-algebra $A$ are self-adjoint with $a\leq b$, then $\norm{a}\leq \norm{b}$, 
     so the above implies that
     \begin{align*}
        0 < \norm{\linner{\cB}{m_{i}}{m'_{i}}}^2 \leq  \norm{\linner{\cB}{m_{i}}{m_{i}}} \,\norm{\linner{\cB}{m_{i}'}{m_{i}'}}
        = \norm{m_{i}}_{M(x)}^2 \,\norm{m_{i}'}_{M(x)}^2,
     \end{align*}
     where the last equality again stems from the fact that $M(x)$ is a Hilbert module, so its norm comes from the inner product. As the norm is exactly the norm with respect to which $\cM$ is a \USCBb, this concludes our proof.
\end{proof}

\begin{corollary}\label{cor:left-nondeg}
    If $b\cdot m = b'\cdot m$ for all $m\in \cM$ with $r_{\cM}(m)=s_{\cB}(b)=s_{\cB}(b')$, then $b=b'$.
\end{corollary}

\begin{proof}
     Since $b\cdot m - b'\cdot m = (b- b')\cdot m$, we have 
     \[
       0
       =
       \norm{b\cdot m - b'\cdot m}
       =
       \norm{(b- b')\cdot m}
       =
       \norm{b-b'} \,\norm{m}
     \]
     for all $m$.
     By Lemma~\ref{lem:non-trivial}, it follows that $\norm{b-b'}=0$ in $\cB_{g},$ meaning that $b=b'$ as claimed.
\end{proof}

\begin{corollary}\label{cor:adjointable}
    If $\cM$ is an equivalence of Fell bundles, then Condition~\ref{item:FE:ip:compatibility} implies that
    \[
       \linner{\cB}{m_{1}\cdot c}{m_{2}}
       =
       \linner{\cB}{m_{1}}{m_{2}\cdot c^*}
    \quad\text{and}\quad
       \rinner{\cC}{b\cdot m_{1}}{m_{2}}
       =
       \rinner{\cC}{m_{1}}{b^*\cdot m_{2}}
    \] 
    whenever each side makes sense for $m_{i}\in M $, $c\in C $, and $b\in B $.
\end{corollary}

{Corollary~\ref{cor:adjointable}} is not as trivial as it seems, despite Condition~\ref{item:FE:SMEs}, since $c$ is allowed to be in a fibre that does not live over a unit in~$\cH$.

\begin{proof}
     We do the proof for the $\cB$-valued inner product; the other one follows {\em mutatis mutandis}.
     
     First, a sanity check:
     \begin{align*}
        \linner{\cB}{m_{1}\cdot c}{m_{2}} \text{ is defined}
        &\iff s_{\cM}(m_{1})=r_{\cC}(c)\text{ and } s_{\cM}(m_{1}\cdot c)=s_{\cM} (m_{2})
        \\&\iff
        s_{\cM}(m_{1})=r_{\cC}(c)
        \text{ and }
        s_{\cC}(c)=s_{\cM}(m_{2})
        \\&\iff 
        s_{\cM}(m_{1})=s_{\cC}(c^*)
        \text{ and }
        r_{\cC}(c^*)=s_{\cM}(m_{2})
        \\&\iff 
        \linner{\cB}{m_{1}}{m_{2}\cdot c^*} \text{ is defined}.
     \end{align*}     
     Furthermore, 
     \begin{alignat*}{2}
        p_{\cB}\Bigl(\linner{\cB}{m_{1}\cdot c}{m_{2}} \Bigr)
        &=
        \leoq{\cG}{q_{\cM}(m_{1}\cdot c)}{q_{\cM}(m_{2})}
        &&\text{{ by \ref{item:FE:ip:fibre}}}
        \\&=
        \leoq{\cG}{q_{\cM}(m_{1})\cdot p_{\cC}(c)}{q_{\cM}(m_{2})}
        &&\text{{ by \ref{item:FA:fibre}}}
        \\&=
        \leoq{\cG}{q_{\cM}(m_{1})}{q_{\cM}(m_{2})\cdot p_{\cC}(c)\inv}
        &
        \\&=
        \leoq{\cG}{q_{\cM}(m_{1})}{q_{\cM}(m_{2}\cdot c^*)}
        &&\text{{ by \ref{item:FA:fibre} and \ref{cond.F5}}}
        \\&=
        p_{\cB}\Bigl(\linner{\cB}{m_{1}}{m_{2}\cdot c^*}\Bigr),&&\text{{ by \ref{item:FE:ip:fibre}}}
     \end{alignat*}
     so the two inner products live in the same fibre.
    Condition~\ref{item:FE:ip:compatibility} implies the first and last equality in the following computation:
    \begin{align*}
       \linner{\cB}{m_{1} \cdot c}{m_{2}} \cdot m_{3}
       &=
       (m_{1} \cdot c) \cdot \rinner{\cC}{m_{2}}{m_{3}}
       =
       m_{1} \cdot \Bigl(c \rinner{\cC}{m_{2}}{m_{3}}\Bigr)
       \\&=
       m_{1} \cdot \rinner{\cC}{m_{2}\cdot  c^*}{m_{3}}
       =
       \linner{\cB}{m_{1}}{m_{2}\cdot  c^*} \cdot m_{3}.
    \end{align*}
    Since $m_{3}\in M $ was arbitrary, it follows from Corollary~\ref{cor:left-nondeg} that $\linner{\cB}{m_{1} \cdot c}{m_{2}}=\linner{\cB}{m_{1}}{m_{2}\cdot  c^*}$, as claimed. 
\end{proof}

Even the `inhomogeneous' inner product on the \hequiv\ $\cM$, which allows to take in elements from two different fibres, satisfies a Cauchy--Schwarz inequality; one can compare our result below with the classical Cauchy--Schwarz inequality on an inner product module {\cite[Lemma 2.5]{RW:Morita}}.

\begin{lemma}[Cauchy--Schwarz] \label{lem:ineq of inner products}
    If $(m,m')\in \cM\bfp{s}{s}\cM$, then 
    \begin{align*}
       \linner{\cB}{m}{m'}\; \linner{\cB}{m}{m'}^*
       \leq
       \norm{m'}^2\;
       \linner{\cB}{m}{m} .
    \end{align*}
\end{lemma}
\begin{proof}
    First, we want to establish that the inequality makes sense. So note that, if
    \[
       g:=p_{\cB}\Bigl(\linner[\cM]{\cB}{m}{m'}\Bigr)=\Leoq{\cG}{q_{\cM}(m)}{q_{\cM}(m')},
    \]
    then by Item~\ref{item:FE:ip:adjoint} and by  Item~\mbox{\ref{item:def-preequiv:inverse}} in  Definition~\ref{def:preequiv},
    \[
       p_{\cB}\Bigl(\linner[\cM]{\cB}{m}{m'}^*\Bigr)
       =
       p_{\cB}\Bigl(\linner[\cM]{\cB}{m'}{m}\Bigr)
       =
       \Leoq{\cG}{q_{\cM}(m')}{q_{\cM}(m)}
       =
       \Leoq{\cG}{q_{\cM}(m)}{q_{\cM}(m')}\inv
       =
       g\inv,
    \]
    so that the product of 
    \(
       \linner[\cM]{\cB}{m}{m'}
     \) and
     \(  \linner[\cM]{\cB}{m}{m'}^*
    \)
    lives in the fibre over $gg\inv=r_{\cG}(g)$. Furthermore, it follows from the assumptions in Definition~\ref{def:preequiv} that
    \[
       p_{\cB}(\linner[\cM]{\cB}{q_{\cM}(m)}{q_{\cM}(m)}) = \Leoq{\cG}{q_{\cM}(m)}{q_{\cM}(m)} = r(g).
    \]
    Thus, the claimed inequality
    {\em makes sense} in the $\textrm{C}^*$-algebra $\cB_{r(g)}.$
    To prove it, we follow the proof of \cite[Lemma 2.5]{RW:Morita}. Let $\rho$ be a state on $\cB_{r(g)}$; it will suffice to prove that
    \begin{align}\label{ineq:rho}
       \rho\Bigl(\linner{\cB}{m}{m'}\; \linner{\cB}{m}{m'}^*\Bigr)
       \leq
       \norm{m'}^2\;
       \rho\Bigl(\linner{\cB}{m}{m}\Bigr).
    \end{align}
    Let $x=q_{\cM}(m)$. The map 
    \[ 
       \cM_{x}\times \cM_{x}\to \mathbb{C},\quad
       (m_{1},m_{2})\mapsto \rho\Bigl(\linner{\cB}{m_{1}}{m_{2}}\Bigr)
       ,
    \]
    is a (semi-definite) positive sesquilinear form on $\cM_{x}$, so the ordinary Cauchy-Schwarz inequality implies that
    \[
       \abs{\rho\Bigl(\linner{\cB}{m_{1}}{m_{2}}\Bigr)}
       \leq 
       \rho\Bigl(\linner{\cB}{m_{1}}{m_{1}}\Bigr)^{\frac{1}{2}}
       \,
       \rho\Bigl(\linner{\cB}{m_{2}}{m_{2}}\Bigr)^{\frac{1}{2}}.
    \]
    Notice that
    \[
       q_{\cM} (\linner{\cB}{m}{m'}m')
       =
       p_{\cB} (\linner{\cB}{m}{m'})\,q_{\cM}(m')
       =
       g\,q_{\cM}(m')
       =
       q_{\cM}(m)=x,
    \]
    so we may let
    \[
       m_{1}:= m
       \quad\text{and}\quad
       m_{2}:= \linner{\cB}{m}{m'}m',
    \]
    which yields:
    \begin{align*}
       0
       &\leq
       \rho\Bigl(\linner{\cB}{m}{m'}\; \linner{\cB}{m}{m'}^*\Bigr)
       =
       \rho\Bigl( 
           \linner{\cB}{m}{\linner{\cB}{m}{m'}m'}
       \Bigr)
       \\
       &
       =
       \rho\Bigl( 
           \linner{\cB}{m_{1}}{m_{2}}
       \Bigr)
       \leq 
       \rho\Bigl(\linner{\cB}{m_{1}}{m_{1}}\Bigr)^{\frac{1}{2}}
       \,
       \rho\Bigl(\linner{\cB}{m_{2}}{m_{2}}\Bigr)^{\frac{1}{2}}
       \\
       &= 
       \rho\Bigl(
           \linner{\cB}{m}{m}
       \Bigr)^{\frac{1}{2}}
       \,
       \rho\Bigl(
           \linner{\cB}{\linner{\cB}{m}{m'}m'}{\linner{\cB}{m}{m'}m'}
       \Bigr)^{\frac{1}{2}}
       \\
       &= 
       \rho\Bigl(
           \linner{\cB}{m}{m}
       \Bigr)^{\frac{1}{2}}
       \,
       \rho\Bigl(
           \linner{\cB}{m}{m'}\,
           \linner{\cB}{m'}{m'}\,
           \linner{\cB}{m}{m'}^*
       \Bigr)^{\frac{1}{2}}
       .
    \end{align*}
    With $b:= \linner{\cB}{m}{m'}$ and $c:=\linner{\cB}{m'}{m'}$, it follows from \cite[Corollary 2.22]{RW:Morita} that
    \begin{align*}
       \rho\Bigl(\linner{\cB}{m}{m'}\,\linner{\cB}{m}{m'}^*\Bigr)
       \leq
       \rho\Bigl(
           \linner{\cB}{m}{m}
       \Bigr)^{\frac{1}{2}}
       \,
       \norm{
           \linner{\cB}{m'}{m'}
       }^{\frac{1}{2}}\,
       \rho\Bigl(
           \linner{\cB}{m}{m'}\,
           \linner{\cB}{m}{m'}^*
       \Bigr)^{\frac{1}{2}},
    \end{align*}
    which yields the desired Inequality~\eqref{ineq:rho} after squaring and cancelling a factor of $$\rho\left(\linner{\cB}{m}{m'}\,\linner{\cB}{m}{m'}^*\right)$$ on both sides.
\end{proof}

\section{Product of Fell Bundle \Hequiv s}\label{sec:Product}
The goal of this section is to show an analogous result of Lemma~\ref{lem:product-of-preequiv} for Fell bundle \hequiv s, namely that the tensor product of two \hequiv s is again a \hequiv.
Throughout this section, we make the following assumptions:
\begin{itemize}
    \item We fix three saturated Fell bundles  $\cB=(p_{\cB}\colon B \to \cG),\cC=(p_{\cC}\colon C \to \cH), \cD=(p_{\cD}\colon D \to \mathcal{K})$ over \LCH\ \etale\ groupoids $\cG,\cH,\mathcal{K}$.
    \item Let $X$ be a $(\cG,\cH)$-\pe\  and $Y$ an $(\cH,\mathcal{K})$-\pe; we let $ Z := X\bfp{s}{r}Y$, which is a $(\cG,\mathcal{K})$-\pe\ by Lemma~\ref{lem:product-of-preequiv}.
    \item Let $\cM=(q_{\cM}\colon M \to X)$ be a $(\cB,\cC)$-\hequiv\ and $\cN=(q_{\cN}\colon N \to Y)$ a $(\cC,\cD)$-\hequiv.
    \item We write $\cdot$ for the left and right actions on $\cM$, $X$, $\cN$, and $Y$. 
\end{itemize}
Out of $\cM$ and $\cN$, we will construct a new bundle, denoted $\cM\otimes_{\cC\z}\cN$, that we will show to be a \hequiv\ between $\cB$ and~$\cD$.

\subsection{Tensor product bundle}

If $x\in X,y\in Y$ with $s_{X}(x)=u=r_{Y}(y)$, we may let $\cM_x \odot_{C(u)} \cN_y$ denote the algebraic balanced tensor product as defined in \cite[Proposition 3.16]{RW:Morita}. On it, we define two inner products with values in a fibre of $\cB$ resp.\ of $\cD$; they are determined by
\begin{align}
\begin{split}\label{eq:ip-on-cKz}
    \lInner{\cB}{m\odot n}{m'\odot n'} &= \linner{\cB}{m}{m'\cdot \linner{\cC}{n'}{n}}, \text{ and}\\
    \rInner{\cD}{m\odot n}{m'\odot n'} &= \rinner{\cD}{ \rinner{\cC}{m'}{m}\cdot n}{n'},
\end{split}
\end{align}
and then extended to the algebraic tensor product by linearity.
Recall from \cite[Remark 3.17]{RW:Morita} that these are indeed `positive definite', meaning if $\inner{\xi}{\xi}=0$, then $\xi=0.$ Furthermore, $\norm{\linner{\cB}{\xi}{\xi}}=\norm{\rinner{\cD}{\xi}{\xi}},$ so that we can unambiguously denote its completion with respect to this norm by $\cM_x \otimes_{C(u)} \cN_y$. Note that, since $u$ is the source of $x$, it is clear from context, and so we will instead often write $\cM_x \otimes_{\cC\z} \cN_y$.

\begin{lemma}\label{lem:topology-on-cK}
    On the set $$K
       :=\bigsqcup_{(x,y)\in X \bfp{s}{r} Y} M(x) \otimes_{\cC\z} N(y),$$ consider all cross-sections of the form
    \[
       \sigma\otimes\tau\colon \quad Z=X\bfp{s}{r}Y \to K,\quad (x,y) \mapsto \sigma(x)\otimes \tau(y),
    \]
    for $\sigma\in \Gamma_{0} (X;\cM)$ and $\tau\in \Gamma_{0} (Y;\cN)$. Then there is a unique topology on 
    $K$ making it a \USCBb\ over $ Z $
    such that all of those cross-sections are continuous. We denote said bundle by $\cM\otimes_{\cC\z} \cN$ or, since $\cM$ and $\cN$ are fixed, by $\cK$; its bundle map will be denoted by $q_{\cK}$.
\end{lemma}
We note that we reserve the notation $\otimes_{\cC}$ for a different construction.

\begin{proof}
    By construction, the obvious map
    $$
    K=\bigsqcup_{(x,y)\in X \bfp{s}{r} Y} \cM_x \otimes_{\cC\z} \cN_y \overset{q_{\cK}}{\longrightarrow} X \bfp{s}{r} Y
    $$
    is a surjection and each fibre $q_{\cK}\inv(x,y)$ is a Banach space. We now let $\Gamma$ denote the $\mathbb{C}$-linear span of all elements $\sigma\otimes\tau$. If we can show that $\Gamma$ satisfies \ref{item:sections:usc} and \ref{item:sections:dense} in Remark~\ref{rmk:topology-in-USC-bundles},
    then there is such a unique topology.
    
    For \ref{item:sections:usc}, fix an arbitrary element of $\Gamma$, and write it as $f=\sum_{i=1}^k \sigma_{i}\otimes \tau_{i}.$ Let
    \[
       f'\colon X \bfp{s}{r} Y \to \cB,\quad
       f'(x,y):=\sum_{i,j=1}^k \linner{\cB}{\sigma_{i} (x)}{\sigma_{j} (x)\cdot \linner{\cC}{\tau_{j} (y)}{\tau_{i} (y)}}.
    \]
    Note that all summands in $f'(x,y)$ live in the same fibre, so that $f'$ is well defined.
    By continuity of $\sigma_{i}, \tau_{i}$, of the right $\cC$-action on $\cM$, of $\linner{\cB}{\cdot}{\cdot}$ and of $\linner{\cC}{\cdot}{\cdot}$, and by continuity of taking finite sums, we know that $f'$ is continuous. 
    By definition of the fibrewise norm, we have
    \begin{align*}
       \norm{f(x,y)}
       =
        \norm{\sum_{i=1}^k \sigma_{i} (x)\otimes \tau_{i} (y)}
        =
        \norm{f'(x,y)}^{\frac{1}{2}},
    \end{align*}
    so we have written $(x,y)\mapsto \norm{f(x,y)}$ as the composition of the continuous map $f'$ with the map $b\mapsto\norm{b}^{\frac{1}{2}}$, which is \usc\ by assumption on~$\cB$. As such, $(x,y)\mapsto \norm{f(x,y)}$ is  \usc\ also.

    \smallskip
    
    For \ref{item:sections:dense}, fix any $m\in \cM_x$ and $n\in \cN_y$. By assumption, there exist elements $\sigma\in \Gamma_{0} (X;\cM)$ and $\tau\in \Gamma_{0} (Y;\cN)$ for which $\sigma(x)=m$ and $\tau(y)=n$. 
    In particular, $\Gamma(x,y)$ contains all elementary tensors $m\otimes n$ and, by construction, their linear hull, making $\Gamma(x,y)$ dense.\qedhere
\end{proof}

\begin{remark}
    As pointed out in Remark~\ref{rmk:enough-sections}, the bundle $\cK$ automatically has enough continuous cross-sections  according to \cite[Proposition 3.4]{Hofmann1977} (see also \cite[II.13.19]{FellDoranVol1} and the discussion in \cite[Appendix A]{MW2008}). One needs to be cautious, however, that the linear span of sections $\sigma\otimes \tau$ is by itself {\em not} sufficient for that.
\end{remark}

\begin{lemma}\label{lem:topology-on-cMcN-w-nets}
    Suppose we have convergent nets $\{m_{\lambda}\}_{\lambda\in\Lambda}$ in $M$ and $\{n_{\lambda}\}_{\lambda\in\Lambda}$ in $N$, with limits $m$ resp.\ $n$. Suppose $s_{\cM}(m_{\lambda})=r_{\cN}(n_{\lambda})$ for all $\lambda\in\Lambda.$ Then $m_{\lambda}\otimes n_{\lambda} \to m\otimes n$ in $\cM\otimes_{\cC\z}\cN.$
\end{lemma}

\begin{proof}
We will use Lemma~\ref{lm.conv.equivalence.conditions}. Let $x_{\lambda}:=q_{\cM}(m_{\lambda})$, $y_{\lambda}:=q_{\cN}(n_{\lambda})$. Since $s_{X}(x_{\lambda})=s_{\cM}(m_{\lambda})=r_{\cN}(n_{\lambda})=s_{Y}(y_{\lambda})$ by assumption, continuity of the maps involved imply that $s_{X}(x)=r_{Y}(y)$ for $x:=q_{\cM}(m)$ and $y:=q_{\cN}(n)$ also.  Moreover, we automatically have $$q_{\cK}(m_{\lambda}\otimes n_{\lambda}) = (x_{\lambda},y_{\lambda}) \to (x,y) = q_{\cK}(m\otimes n).$$  If $m_{i}\in \cM_{x},n_{i}\in\cN_{y},$ then
\begin{align}\label{eq:tensor-minus-tensor}
    \norm{\sum_{i=0}^{k}m_{i}\otimes n_{i}}^2
    &=
    \norm{\sum_{i,j=0}^{k}\Inner{m_{i}\otimes n_{i}}{m_{j}\otimes n_{j}}}
    =
    \norm{
        \sum_{i,j=0}^{k}
         \linner{\cB}{m_{i}}{m_{j}\cdot \linner{\cC}{n_{j}}{n_{i}}}
    }.
\end{align}
Now take arbitrary $\sigma_{i}\in \Gamma_{0}(X;\cM)$ and $\tau_{i}\in\Gamma_{0} (Y;\cN)$. Their continuity, the continuity of the right $\cC$-action on $\cM$, of $\linner{\cB}{\cdot}{\cdot}$ and $\linner{\cC}{\cdot}{\cdot}$, and of taking finite sums, and the upper semi-continuity of the norm on $\cB$ together with Equation~\eqref{eq:tensor-minus-tensor} imply that
\begin{align*}
    \limsup_{\lambda}
    \norm{
    m_{\lambda} \otimes n_{\lambda}  -
    \sum_{i=1}^{k}\sigma_{i}(x_{\lambda})\otimes \tau_{i}(y_{\lambda})}^2
    \leq
    \norm{
    m\otimes n -
    \sum_{i=1}^{k}\sigma_{i}(x)\otimes \tau_{i}(y)}^2.
\end{align*}
Since the sections on $\cM$ and $\cN$ were arbitrary, it follows from Lemma~\ref{lm.conv.equivalence.conditions}, \ref{item:convergence2}$\implies$\ref{item:convergence1}, that $m_{\lambda}\otimes n_{\lambda}\to m\otimes n$ in $\cM\otimes_{\cC\z}\cN,$ as claimed.
\end{proof}

\begin{lemma}\label{lem:qMN-open}
    The map $ q_{\cK} \colon \cK=\cM\otimes_{\cC\z}\cN \to X \bfp{s}{r} Y=Z$ is open.
\end{lemma}
\begin{proof}
    Let $A  \subseteq K$ be any closed subset.
    To see that $q:= q_{\cK} $ is open, it suffices to prove that the set $B:=\{(x,y)\in Z  \,\mid\,q\inv (x,y)\subseteq A\}$ is closed.
    To this end, let $(x_{\lambda},y_{\lambda})\in B$ be a net that converges to $(x,y)$ in $Z$; we claim that $(x,y)\in B.$ In other words, if we fix any $\xi\in q\inv (x,y),$ we need to show that $\xi\in A$.
    
    For $\xi$, take any $\sigma\in \Gamma_{0} (Z;\cK)$ for which $\sigma(x,y)=\xi.$ Then, by continuity of $\sigma$, we have that $\xi_{\lambda} := \sigma(x_{\lambda},y_{\lambda})$  converges to $\xi$ in~$\cK$. Since $\sigma$ is a section, $\xi_{\lambda} \in q\inv (x_{\lambda},y_{\lambda})$. As $(x_{\lambda},y_{\lambda})\in B$, we have
    \(
       q\inv (x_{\lambda},y_{\lambda})\subseteq A,
    \)
    so that $\xi_{\lambda}\in A$ and thus $\xi$ is the limit of a net in $A$. As $A$ was closed by assumption, it follows that, indeed, $\xi\in A$ as claimed.
\end{proof}

\subsection{Additional structure on the product bundle}
In this subsection, we will study the \USCBb\ $\cK:=\cM\otimes_{\cC\z}\cN=(q_{\cK}\colon K \to X\bfp{s}{r}Y)$ in more detail. Recall that $ Z :=X\bfp{s}{r}Y$ naturally is a $(\cG,\mathcal{K})$-\pe\ by Lemma~\ref{lem:product-of-preequiv}, which is also where we defined $\Leoq[ Z ]{\cG}{\cdot}{\cdot}$ and $\Reoq[ Z ]{\cdot}{\cdot}{\cG}.$
As always, we will write $s_{\cK}:= s_{ Z }\circ q_{\cK}$ and $r_{\cK}:= r_{ Z }\circ q_{\cK}$.

\begin{proposition}\label{prop:cK:actions}
    The bundle $\cK$ naturally carries a left $\cB$- and a right $\cD$-action in the sense of Definition~\ref{def:USCBb-action} determined by
    \[ 
       b\cdot (m\otimes n) = (b\cdot m)\otimes n 
       \quad\text{and}\quad 
       (m\otimes n)\cdot d = m\otimes (n\cdot d)
    \]
    where $b\in B $, $m\otimes n\in K $, and $d\in D$ are such that $s_{\cB}(b)=r_{\cM}(m)$ and $s_{\cN}(n)=r_{\cD}(d).$ Moreover, these actions commute. 
\end{proposition}
    
\begin{proof}
    We will do the proof for $\cB$; the proof for $\cD$ will follow {\em mutatis mutandis}, and it is clear from the given description that the actions commute.
    Suppose $g\in \cG$ and $(x,y)\in  Z$ are such that $s_{\cG}(g)=r_{X}(x).$ For any fixed $b\in \cB_{g}$, define
    \[
       \dbtilde{\phi}^{(x,y)}_{b} = \dbtilde{\phi} \colon\quad \cM_{x}\times \cN_{y} \to \cM_{gx}\otimes_{\cC\z} \cN_{y},\quad
       (m, n) \mapsto (b\cdot m)\otimes n.
    \]
    By the properties of the left $\cB$-action on $\cM$, one quickly verifies that $\dbtilde{\phi}$ is bilinear, and since
    \[
    	(b\cdot (m\cdot c))\otimes n
    	=
    	((b\cdot m)\cdot c))\otimes n
    	=
 	    (b\cdot m)\otimes (c\cdot n)   	    	
    \]
    for any $c\in\cC_{s(x)}$, we conclude that the induced linear map on $\cM_{x}\odot \cN_{y}$ factors through the $\cC\z$-balancing. In other words, there exists a map
    \[
       \tilde{\phi}^{(x,y)}_{b} \colon\quad \cM_{x}\odot_{\cC\z} \cN_{y} \to \cM_{gx}\otimes_{\cC\z} \cN_{y}\quad \text{ determined by }\quad
       m \odot n \mapsto (b\cdot m)\otimes n.
    \]
    By Conditions~\ref{item:FE:ip:adjoint} and \ref{item:FE:ip:C*linear} on $\cM$ ($\cB$-*-linearity of the inner products), we have
    \begin{align*}
  		\lInner{\cB}{(b' \cdot m' )\otimes n' }{(b \cdot m )\otimes n }
  		&=
  		\linner{\cB}{b' \cdot m' }{(b \cdot m )\cdot \linner{\cC}{n }{n' }}
  		=
  		b' \,\linner{\cB}{m' }{m\cdot \linner{\cC}{n }{n' }}\,b^*
  		\\
  		&=
  		b' \,\lInner{\cB}{ m' \otimes n' }{m\otimes n }\,b^*,
     	\end{align*}
    so we conclude that for any $\xi\in \cM_{x}\odot_{\cC\z} \cN_{y}$,
 \begin{align*}
  		\tensor*[]{\norm{\tilde{\phi}^{(x,y)}_{b}(\xi)
  		}}{^2_{\cK}}
  		&=
  		\tensor*[]{\norm{
  			\lInner[\cK]{\cB}{\tilde{\phi}^{(x,y)}_{b}(\xi)
  			}{\tilde{\phi}^{(x,y)}_{b}(\xi)
  			}
  		}}{_{\cB}}
  		=
  		\tensor*[]{\norm{
  			b\,\lInner[\cK]{\cB}{ \xi
  			}{ \xi
  			}\,b^*
  		}}{_{\cB}}
  		\leq
  		\norm{b}^2\,
  		\tensor*[]{\norm{\lInner[\cK]{\cB}{ \xi
  		}{ \xi
  		}
  		}}{_{\cB}}
  		=
  		\norm{b}^2\,
  		\tensor*[]{\norm{ \xi
  		}}{^2_{\cK_{(x,y)}}}
  		,
  	\end{align*}
  	where we added the subscripts for clarity.
    It is now clear that $\tilde\phi^{(x,y)}_{b}$ extends to a continuous linear map 
  	\[
       {\phi}^{(x,y)}_{b} \colon\quad \cM_{x}\otimes_{\cC\z} \cN_{y} \to \cM_{gx}\otimes_{\cC\z} \cN_{y}\quad \text{ determined by }\quad
       m \otimes n \mapsto (b\cdot m)\otimes n,
    \]
    which satisfies
  	\begin{equation}\label{ineq:phi-norm-decreasing}
		\norm{\phi^{(x,y)}_{b}(\xi)}
  		\leq
  		\norm{b}\,\norm{\xi}
  		\quad\text{ for all }
  		\xi\in \cK_{(x,y)}=\cM_{x}\otimes_{\cC\z} \cN_{y}
  		.
  	\end{equation}
    Other properties of the $\cB$-action on $\cM$ imply that 
    \begin{equation}\label{eq:phi-lin-and-comp}
    	\phi^{(x,y)}_{b + \lambda b'}
    	=
    	\phi^{(x,y)}_{b}
    	+
    	\lambda
    	\phi^{(x,y)}_{b'}
    	\quad\text{and}\quad
       	\phi^{(g\cdot x,y)}_{b_{1}}
       	\circ 
       	\phi^{(x,y)}_{b_{2}}
       	=
       	\phi^{(x,y)}_{b_{1}b_{2}}
       	\quad\text{ where } g=p_{\cB}(b_{2})
    \end{equation}
    for all $b,b'\in B $ with $p_{\cB}(b)=p_{\cB}(b')$ and all $(b_{1},b_{2})\in\cB^{(2)}.$  We define 
   	\[
   		\cB \bfp{s}{r} \cK \to \cK,
   		\quad
   		(b,\xi) \mapsto 
   		b\cdot \xi := \phi^{(x,y)}_{b} (\xi)
   		\quad\text{ where }
       (x,y)=q_{\cK}(\xi)
   	    .
   	\]
   	By construction, we have $q_{\cK}(b\cdot \xi)=p_{\cB}(b)\cdot q_{\cK}(\xi)$. The first equality in~\eqref{eq:phi-lin-and-comp} implies that the action is linear and the second that
   	$a\cdot (b\cdot \xi)=(ab)\cdot \xi$ for appropriate $a,b\in \cB$.
   	Inequality~\eqref{ineq:phi-norm-decreasing} implies that $\norm{b\cdot \xi}\leq\norm{b}\,\norm{\xi}.$
   	
   	\smallskip
   	
   It only remains to show that $(b,\xi)\mapsto b\cdot \xi$ is continuous,    so suppose that we have a convergent net in $\cB\bfp{s}{r}\cK$, say $(b_{\lambda},\xi_{\lambda})\to (b,\xi)$. In particular,  $(g_{\lambda},x_{\lambda},y_{\lambda}):=(p_{\cB}(b_{\lambda}), q_{\cK}(\xi_{\lambda})) \to (p_{\cB}(b), q_{\cK}(\xi))=:(g,x,y)$.
  Fix $\epsilon >0$. Because of Lemma~\ref{lm.conv.equivalence.conditions}, \ref{item:convergence3.gamma}$\implies$\ref{item:convergence1},
  it suffices to find a section $\varphi$ of $\cK$ and some $\lambda_0$ such that for all $\lambda\geq\lambda_0,$ we have
  \[
        \norm{
            b_{\lambda}\cdot \xi_{\lambda}
          -
          \varphi(g_{\lambda}x_{\lambda},y_{\lambda})
        }
        < \epsilon
        \quad\text{and}\quad
        \norm{
            b\cdot \xi
          -
          \varphi(gx,y)
        }
        <\epsilon.
  \]
  Pick finitely many $\tau_{j}\in\Gamma_0 (X;\cM),\kappa_{j}\in\Gamma_0 (Y;\cN)$ such that
   \begin{equation}\label{eq:choice of tau and kappa}
       \norm{\xi-\sum_{j}\tau_{j}(x)\otimes \kappa_{j}(y)}
       <
        \frac{\epsilon}{2(\norm{b}+1)}.
   \end{equation}
    Because of Inequality~\eqref{ineq:phi-norm-decreasing}, this implies
   \begin{align}
      \label{eq:xi-tau0}
      \norm{b\cdot \xi- b\cdot\sum_{j}\tau_{j}(x)\otimes \kappa_{j}(y)}
      \leq
      \norm{b}\norm{\xi-\sum_{j}\tau_{j}(x)\otimes \kappa_{j}(y)}
      &< \frac\epsilon2.
    \end{align}
    Since $(b_{\lambda},\xi_{\lambda})\to(b,\xi),$ we know from Lemma~\ref{lm.conv.equivalence.conditions}, \ref{item:convergence1}$\implies$\ref{item:convergence2}, that
    \[
       \limsup_{\lambda}\norm{b_{\lambda}}\leq\norm{b}
       \quad\text{and}\quad
       \limsup_{\lambda}\norm{\xi_{\lambda}-\sum_{j}\tau_{j}(x_{\lambda})\otimes \kappa_{j}(y_{\lambda})}
       \leq
       \norm{\xi-\sum_{j}\tau_{j}(x)\otimes \kappa_{j}(y)}.
    \]
    Together with Inequalities~\eqref{ineq:phi-norm-decreasing} and~\eqref{eq:choice of tau and kappa}, this implies that 
    \begin{align*}
      &\limsup_{\lambda}\norm{b_{\lambda}\cdot \xi_{\lambda}- b_{\lambda}\cdot\sum_{j}\tau_{j}(x_{\lambda})\otimes \kappa_{j}(y_{\lambda})}
      \leq
      \limsup_{\lambda} \left(\norm{b_{\lambda}}\norm{\xi_{\lambda}-\sum_{j}\tau_{j}(x_{\lambda})\otimes\kappa_{j}(y_{\lambda})}\right)
      \notag
      \\
      &\qquad\leq
      \left(
      \limsup_{\lambda}\norm{b_{\lambda}}\right)
      \left(
      \limsup_{\lambda}
      \norm{\xi_{\lambda}-\sum_{j}\tau_{j}(x_{\lambda})\otimes\kappa_{j}(y_{\lambda})}
      \right)
      \leq
      \norm{b}
       \norm{\xi-
       \sum_{j}
       \tau_{j}(x)\otimes \kappa_{j}(y)}
      < \frac\epsilon2.
   \end{align*}
   So there exists $\lambda_{1}$ such that for $\lambda\geq\lambda_1$
    \begin{equation}\label{eq:cty-of-action:lambda1}
      \norm{b_{\lambda}\cdot \xi_{\lambda}- b_{\lambda}\cdot\sum_{j}\tau_{j}(x_{\lambda})\otimes \kappa_{j}(y_{\lambda})}
      < \frac\epsilon2.
    \end{equation}
   Since the $\cB$-action on $\cM$ is continuous, we have that
   \(
       \lim_{\lambda} b_{\lambda}\cdot\tau_{j} (x_{\lambda}) = b\cdot \tau_{j}(x).
   \)
   By Lemma~\ref{lem:topology-on-cMcN-w-nets}, it follows that  
   \(
       \lim_{\lambda} [b_{\lambda}\cdot \tau_{j} (x_{\lambda})]\otimes \kappa_{j} (y_{\lambda}) = [b\cdot \tau_{j}(x)]\otimes \kappa_{j} (y)
   \)
   in~$\cK$. By Lemma~\ref{lm.conv.equivalence.conditions}, \ref{item:convergence1}$\implies$\ref{item:convergence3.gamma},
    there exists a section $\varphi$ of $\cK$ and some $\lambda_{2}$ such that for all $\lambda\geq \lambda_2$
   \begin{equation}\label{eq:lambda2}
    \norm{
          \sum_{j}
           [b_{\lambda}\cdot \tau_{j} (x_{\lambda})] \otimes \kappa_{j} (y_{\lambda})
           -
           \varphi(g_{\lambda}x_{\lambda},y_{\lambda})
       }
       < \frac\epsilon2
   \end{equation}
   and
   \begin{equation}\label{eq:2}
    \norm{ 
          \sum_{j}
           [b\cdot \tau_{j} (x)] \otimes \kappa_{j} (y)
            -
           \varphi(gx,y)
       }
       < \frac\epsilon2.
   \end{equation}
   Fix any $\lambda_0\geq \lambda_1,\lambda_2$.
  Combining Equations~\eqref{eq:cty-of-action:lambda1} and~\eqref{eq:lambda2} respectively Equations~\eqref{eq:xi-tau0} and~\eqref{eq:2} with the triangle inequality, we see that for all $\lambda\geq \lambda_0$
  \begin{align*}
        \norm{
            b_{\lambda}\cdot \xi_{\lambda}
          -
          \varphi(g_{\lambda}x_{\lambda},y_{\lambda})
        }
        &
        < \epsilon
        \quad\text{and}\quad
        \norm{
          b\cdot \xi
          -
          \varphi(gx,y)
        }
        < \epsilon.
  \end{align*}
  This concludes our proof of continuity.  
\end{proof}

\begin{lemma}\label{lem:cK:ip-on-elementary-tensors}
    Suppose we are given elements 
    $m,m'\in\cM$ and $n,n'\in\cN$ such that
    $(x,y):=(q_{\cM}(m),q_{\cN}(n)),(x',y'):=(q_{\cM}(m'),q_{\cN}(n'))$ are elements of $Z=X\bfp{s}{r}Y$.
    \begin{enumerate}[label=\textup{(\arabic*)}]
       \item\label{item:cK:lip-on-elementary-tensors} We have $s_{\cK}(m\otimes n)=s_{\cK}(m'\otimes n'\bigr)$ if and only if $s_{\cN}(n)=s_{\cN}(n').$ In that case, the element \(
           \linner{\cB}{m}{m'\cdot \linner{\cC}{n'}{n} }\)
       is well defined and lives in the fibre over
       $\Leoq[ Z ]{\cG}{(x,y)}{(x',y')}
       $
       of~$\cB$. 
       \item\label{item:cK:rip-on-elementary-tensors}
        We have $r_{\cK}(m\otimes n)=r_{\cK}(m'\otimes n'\bigr)$ if and only if $r_{\cM}(m)=r_{\cM}(m').$ In that case,  the element \(
           \rinner{\cD}{\rinner{\cC}{m'}{m}\cdot n}{n'}\)
       is well defined and lives in the fibre over $\Reoq[ Z ]{(x,y)}{(x',y')}{\mathcal{K}}
    $ of~$\cD$.
    \end{enumerate}
\end{lemma}

\begin{proof}
    We will only deal with \ref{item:cK:lip-on-elementary-tensors}; {\em mutatis mutandis,} one proves  \ref{item:cK:rip-on-elementary-tensors}.
    We compute
    \begin{align*}
       s_{\cK}(m\otimes n)=s_{\cK}(m'\otimes n')
       \iff&s_{ Z }(q_{\cK}(m\otimes n))=s_{ Z }(q_{\cK}(m'\otimes n'))
       \\
       \iff& s_{ Z }(x,y) = s_{ Z }(x',y')
       \\
       \iff& s_{Y}(y) = s_{Y}(y')
       \iff s_{\cN}(n) = s_{\cN}(n'),
    \end{align*}
    which shows that $\linner{\cC}{n'}{n}$ makes sense. 
    Further, as $m\otimes n, m'\otimes n' \in \cK$, we have $s_{\cM}(m)=r_{\cN}(n)$ and $s_{\cM}(m')=r_{\cN}(n')$ in $\cH\z$, so 
    \begin{align*}
       r_{\cC} \left(\linner{\cC}{n'}{n} \right)
       &=
       (r_{\cH}\circ p_{\cC}) \left(\linner{\cC}{n'}{n} \right)
       =
       r_{\cH}\left(\Leoq[Y]{\cH}{q_{\cN}(n')}{q_{\cN}(n)} \right)
       =
       r_{\cH}\left(\Leoq[Y]{\cH}{y'}{y} \right)
       ,
    \end{align*}
    where the second equality follows by Condition~\ref{item:FHE:ip-fibre}, a property of the $\cC$-valued inner product on $\cN$. By Items~\mbox{\ref{item:def-preequiv:s-and-r}} and \mbox{\ref{item:def-preequiv:r-and-s}} in Definition~\ref{def:preequiv}, the range of $\Leoq[Y]{\cH}{y'}{y}$ in $\cH\z$ is $r_{Y}(y')$ and its source is $r_{Y}(y)$. Thus,
    \begin{align*}
       r_{\cC} \left(\linner{\cC}{n'}{n} \right)
       =
       r_{\cH}\left( \Leoq[Y]{\cH}{y'}{y} \right)
       =
       r_{Y}(y')
       =
       r_{\cN}(n')
       =s_{\cM}(m'),
    \end{align*}
    so that $m'\cdot \linner{\cC}{n'}{n}$ makes sense, and 
    \begin{align*}
       s_{\cM} \left( m'\cdot \linner{\cC}{n'}{n} \right) 
       &=
       s_{\cC} \left(\linner{\cC}{n'}{n} \right)
       =
       (s_{\cH}\circ p_{\cC}) \left(\linner{\cC}{n'}{n} \right)
       \\
       &
       =
       s_{\cH}\left( \Leoq[Y]{\cH}{y'}{y} \right)
       =
       r_{Y}(y)
       =
       r_{\cN}(n)
       =
       s_{\cM} (m),
    \end{align*}
    so that \(
       \linner{\cB}{m}{m'\cdot \linner{\cC}{n'}{n} }
    \) makes sense. 
    To see over which fibre that element lives, we compute
    \begin{alignat*}{2}
       p_{\cB}
       \left(
           \linner{\cB}{m}{m'\cdot \linner{\cC}{n'}{n} }
       \right)
       &=
       \Leoq[X]{\cG}{q_{\cM}(m)}{q_{\cM}(m'\cdot \linner{\cC}{n'}{n})}
       && \text{by \ref{item:FHE:ip-fibre}  for $\cM$}
       \\&=
       \Leoq[X]{\cG}{x}{q_{\cM}(m')\cdot p_{\cC}(\linner{\cC}{n'}{n})}
       && \text{by \ref{item:FA:fibre}  for $\cM$}
       \\&=
       \Leoq[X]{\cG}{x}{x'\cdot \Leoq[Y]{\cH}{q_{\cN}(n')}{q_{\cN}(n)}}
       && \text{by \ref{item:FHE:ip-fibre}  for $\cN$}
       \\&=
       \Leoq[X]{\cG}{x}{x'\cdot \Leoq[Y]{\cH}{y'}{y}}
       =
       \Leoq[ Z ]{\cG}{(x,y)}{(x',y')}.
    \end{alignat*}
    This concludes our proof.
\end{proof}

\begin{lemma}
\label{lem:cK:ip-on-odot+adjoint}
   Assume $(x,y),(x',y')\in Z= X\bfp{s}{r}Y$.
   \begin{enumerate}[label=\textup{(\arabic*)}]
       \item\label{item:cK:ip-on-elementary-tensors-cB} If $s_{Y}(y)=s_{Y}(y')$ and $g:= \Leoq[ Z ]{\cG}{(x,y)}{(x',y')}
       $, then there exists a sesquilinear map
       \begin{align*}
       \lInner{\cB}{\cdot}{\cdot}\colon\quad
       \left(\cM_{x}\odot \cN_{y}\right)
       &\times \left(\cM_{x'}\odot \cN_{y'}\right)
       \to 
       \cB_{g}
    \intertext{ determined by }
       \lInner{\cB}{m\odot n}{m'\odot n'}
       :=\ &
       \linner{\cB}{m}{m'\cdot \linner{\cC}{n'}{n} }
    \intertext{which satisfies}
        \lInner{\cB}{m\odot n}{m'\odot n'}^*
        =\ &
        \lInner{\cB}{m'\odot n'}{m\odot n}.
       \end{align*}
       \item
       If $r_{X}(x)=r_{X}(x')$ and 
       $k:= \Reoq[Z]{(x,y)}{(x',y')}{\mathcal{K}}$, then there exists a sesquilinear map
       \begin{align*}
       \rInner{\cD}{\cdot}{\cdot}\colon\quad
       \left(\cM_{x}\odot \cN_{y}\right)
       &\times \left(\cM_{x'}\odot \cN_{y'}\right)
       \to 
       \cD_{k}
   \intertext{
       determined by
       }
       \rInner{\cD}{m\odot n}{m'\odot n'}
       :=\ &
       \rinner{\cD}{\rinner{\cC}{m'}{m}\cdot n}{n'}
   \intertext{which satisfies}
       \rInner{\cD}{m\odot n}{m'\odot n'}^*
             =\ &
              \rInner{\cD}{m'\odot n'}{m\odot n}.
       \end{align*}
    \end{enumerate}
\end{lemma}

We point out that, while the formulas of the above 
forms look identical to those for the inner products on a fibre $\cK_{(x,y)}$ as defined in \eqref{eq:ip-on-cKz}, there is one major difference: These new forms can be evaluated at two elements that live in {\em different} fibres; this is hinted at in the codomains of the map, since the elements $g$ and $k$ are not necessarily units in $\cG$ resp.\ $\mathcal{K}$. We will soon see that the forms can be induced from the algebraic tensor product to the balanced completions, $\cK_{(x,y)}$ and $\cK_{(x',y')}$. 

\begin{proof}
    We will prove \ref{item:cK:ip-on-elementary-tensors-cB}; the other claim is done analogously. We have seen in Lemma~\ref{lem:cK:ip-on-elementary-tensors} that the formula makes sense for elementary tensors. 
    We now follow the ideas in \cite[Proof of Prop.\ 3.16]{RW:Morita}.  Fix $m'\in \cM_{x'},n'\in\cN_{y'}$, and consider the map
    \[
       \cM_{x}\times \cN_{y}\ni (m,n)
       \mapsto
       \linner{\cB}{m}{m'\cdot \linner{\cC}{n'}{n} }
       \in \cB_{g}
       .
    \]
    As $\linner{\cB}{\cdot}{\cdot}$ and $\linner{\cC}{\cdot}{\cdot}$ are both linear in the {\em first} and  conjugate linear in the {\em second} component, this map is bilinear and hence induces a {\em linear} map
      \[
           F_{m',n'}\colon \cM_{x}\odot \cN_{y} \to \cB_{g}
           \quad\text{ determined by }\quad
           m\odot n
           \mapsto
           \linner{\cB}{m}{m'\cdot \linner{\cC}{n'}{n} }
           .
       \]
    Again, sesquilinearity of $\linner{\cB}{\cdot}{\cdot}$ and $\linner{\cC}{\cdot}{\cdot}$ imply that the map
       \[
           (m',n')
           \mapsto
           F^*(m',n') :=
           \bigl(
              \xi\mapsto F_{m',n'}(\xi)^*
           \bigr)
       \]
       is a bilinear map from $\cM_{x'}\times \cN_{y'}$ into the vector space $\mathrm{CL}$ of conjugate linear maps $\cM_{x}\odot \cN_{y} \to \cB_{g\inv}$. Thus, there exists a {\em linear} map $F^*\colon \cM_{x'}\odot \cN_{y'}\to \mathrm{CL}$ determined by $m'\odot n'
       \mapsto
       F^*(m',n')$. We can thus define
       \[
           \lInner{\cB}{\cdot}{\cdot}\colon\quad
       \left(\cM_{x}\odot \cN_{y}\right)
       \times \left(\cM_{x'}\odot \cN_{y'}\right)
       \to 
       \cB_{g},\quad
           \lInner{\cB}{\xi}{\xi'} := [F^*(\xi')(\xi)]^*.
       \]
       The  map is sesquilinear: it is linear in the first component since $F^*(\xi')$ is conjugate linear, and it is conjugate linear in the second component by linearity of $F^*$.

\smallskip

 For the claim about the adjoint, we compute
       \begin{alignat*}{2}
           \lInner{\cB}{m\odot n}{m'\odot n'}^*
           &=
           \linner{\cB}{m}{m'\cdot \linner{\cC}{n'}{n} }^*
           \quad &&
           \\
           &=
           \linner{\cB}{m'\cdot \linner{\cC}{n'}{n} }{m}
           &&
           \text{{by \ref{item:FE:ip:adjoint} for~$\cM$}}
           \\
           &=
           \linner{\cB}{m'}{m \cdot \linner{\cC}{n'}{n}^*}
           &&
           \text{{by \ref{item:FHE:ip:adjointables} for~$\cM$}}
           \\
           &=
           \linner{\cB}{m'}{m \cdot \linner{\cC}{n}{n'}}
           &&
           \text{{by \ref{item:FE:ip:adjoint} for~$\cN$}}
           \\
           &=
           \lInner{\cB}{m'\odot n'}{m\odot n}.
           &&
       \end{alignat*}
       This concludes our proof.
\end{proof}

We now verify that the products satisfy Condition~\ref{item:FHE:2e}.

\begin{lemma}\label{lem:FHE:2e for cK} 
Suppose we are given elements $\xi_{i}\in \cM_{x_{i}}\odot \cN_{y_{i}}\subseteq \cK_{(x_{i},y_{i})}$ for $i=1,2,3,4$. If $(x_2,y_2)=(x_3,y_3)$, then we have
\begin{align*}
    \lInner{\cB}{\xi_{1}}{\xi_{2}} \lInner{\cB}{\xi_3}{\xi_4} &=  \lInner{\cB}{\xi_{1} \cdot \rInner{\cD}{\xi_{2}}{\xi_3} }{\xi_4},\text{ and}\\
    \rInner{\cD}{\xi_{1}}{\xi_2} \rInner{\cD}{\xi_3}{\xi_4} &=  \rInner{\cD}{\xi_{1}}{\lInner{\cB}{\xi_2}{\xi_3} \cdot \xi_4},
\end{align*}
wherever the inner products and products on the left-hand side make sense.
\end{lemma}

\begin{proof}
It suffices to prove the claim for elementary tensors $\xi_{i}=m_{i}\otimes n_{i}$.
Applying the definition of $\lInner{\cB}{\cdot}{\cdot}$ and $\rInner{\cD}{\cdot}{\cdot}$, we have:
\begin{alignat*}{2}
    &\lInner{\cB}{m_1\odot n_1}{m_2\odot n_2} \lInner{\cB}{m_3\odot n_3}{m_4\odot n_4} 
    \qquad&&\\ 
    =& \linner{\cB}{m_1}{m_2\cdot \linner{\cC}{n_2}{n_1}} \linner{\cB}{m_3}{m_4\cdot \linner{\cC}{n_4}{n_3}} \\
    =& \linner{\cB}{m_1\cdot \rinner{\cC}{m_2\cdot \linner{\cC}{n_2}{n_1}}{m_3}}{m_4\cdot \linner{\cC}{n_4}{n_3}} &&\text{by Condition~\ref{item:FHE:2e}}\\
    =& \linner{\cB}{m_1\cdot (\linner{\cC}{n_1}{n_2} \rinner{\cC}{m_2}{m_3})}{m_4\cdot \linner{\cC}{n_4}{n_3}} &&\text{by Conditions~\ref{item:FE:ip:C*linear} and~\ref{item:FE:ip:adjoint}.}
\end{alignat*}
On the other hand,
\begin{alignat*}{2}
    &\lInner{\cB}{(m_1\odot n_1) \cdot \rInner{\cD}{m_2\odot n_2}{m_3\odot n_3} }{m_4\odot n_4}\qquad&& \\
    =& \lInner{\cB}{m_1\odot (n_1\cdot (\rinner{\cD}{\rinner{\cC}{m_3}{m_2} \cdot n_2}{n_3}))}{m_4\odot n_4} \\
    =& \linner{\cB}{m_1}{m_4 \cdot \linner{\cC}{n_4}{n_1\cdot (\rinner{\cD}{\rinner{\cC}{m_3}{m_2} \cdot n_2}{n_3})}} \\
    =& \linner{\cB}{m_1}{m_4 \cdot \linner{\cC}{n_4 \cdot (\rinner{\cD}{n_3}{\rinner{\cC}{m_3}{m_2} \cdot n_2})}{n_1}} &&\text{by Conditions~\ref{item:FHE:ip:adjointables} and \ref{item:FE:ip:adjoint}}\\
    =& \linner{\cB}{m_1}{m_4\cdot (\linner{\cC}{n_4}{n_3}
    \linner{\cC}{\rinner{\cC}{m_3}{m_2}\cdot n_2}{n_1})} &&\text{by Condition~\ref{item:FHE:2e}}\\
    =& \linner{\cB}{m_1}{m_4\cdot (\linner{\cC}{n_4}{n_3} \rinner{\cC}{m_3}{m_2} \linner{\cC}{n_2}{n_1})} &&\text{by Condition~\ref{item:FE:ip:C*linear}}\\
    =& \linner{\cB}{m_1\cdot (\rinner{\cC}{m_3}{m_2}\linner{\cC}{n_2}{n_1})^*}{m_4\cdot \linner{\cC}{n_4}{n_3}} &&\text{by Condition~\ref{item:FHE:ip:adjointables}} \\
    =& \linner{\cB}{m_1\cdot (\linner{\cC}{n_1}{n_2} \rinner{\cC}{m_2}{m_3})}{m_4\cdot \linner{\cC}{n_4}{n_3}} &&\text{by Condition~\ref{item:FE:ip:C*linear}}\\
    =& \lInner{\cB}{m_1\odot n_1}{m_2\odot n_2} \lInner{\cB}{m_3\odot n_3}{m_4\odot n_4}&&\text{by our computation above}.
\end{alignat*}
The other equality follows similarly. 
\end{proof}

\begin{lemma}
    \label{prop:Inner-ineq}
       Suppose we are given elements $\xi\in \cM_{x}\odot \cN_{y}\subseteq \cK_{(x,y)}, \xi'\in \cM_{x'}\odot \cN_{y'}\subseteq \cK_{(x',y')}$.
    \begin{enumerate}[label=\textup{(\arabic*)}]
       \item
       If $s_{Y}(y)=s_{Y}(y')$, then the following
       inequality of positive elements in the $\textrm{C}^*$-algebra $\cB(r(x))$ holds:
        \begin{align*}
            \lInner{\cB}{\xi}{\xi'}
            \,
            \lInner{\cB}{\xi}{\xi'}^*
            \leq
            \norm{\xi'}^2
            \,
            \lInner{\cB}{\xi}{\xi}.
        \end{align*}
       \item 
       If $r_{X}(x)=r_{X}(x')$, then the following
       inequality of positive elements in the $\textrm{C}^*$-algebra $\cD(s(y))$ holds:
        \begin{align*}
            \rInner{\cD}{\xi}{\xi'}^*
            \,
            \rInner{\cD}{\xi}{\xi'}
            \leq
            \norm{\xi}^2
            \,
            \rInner{\cD}{\xi'}{\xi'}.
        \end{align*}
    \end{enumerate}
    In both cases, the norm
    on the right-hand side is on the bi-Hilbert module $\cM_{x'}\otimes_{\cC\z}\cN_{y'}$ resp.\  $\cM_{x}\otimes_{\cC\z}\cN_{y}$.
\end{lemma}

\begin{proof}
    We will prove the claim for the $\cB$-valued inner product.
    By Lemma~\ref{lem:FHE:2e for cK}, we can rewrite the left-hand side  of our inequality as
    \[
        \lInner{\cB}{\xi}{\xi'} \lInner{\cB}{\xi}{\xi'}^*
        =
      \lInner{\cB}{\xi}{\xi'} \lInner{\cB}{\xi'}{\xi} =  \lInner{\cB}{\xi \cdot \rInner{\cD}{\xi'}{\xi'} }{\xi},
    \]
    so our claim becomes
    \[
        \lInner{\cB}{\xi \cdot \rInner{\cD}{\xi'}{\xi'} }{\xi}
            \overset{!}{\leq}
            \norm{\xi'}^2
            \,
            \lInner{\cB}{\xi}{\xi}.
    \]
    Let $z:=(x,y)$ and $z:=(x',y').$
    Notice that 
        by Lemma~\ref{lem:cK:ip-on-elementary-tensors}\ref{item:cK:rip-on-elementary-tensors} and Condition~\ref{item:def-preequiv:units},
    $\rInner{\cD}{\xi'}{\xi'}$ is an element of the fibre of $\cD$ over 
    $\Reoq[Z]{z'}{z'}{\mathcal{K}}=s_{Z}(z')$.
    Moreover, since 
    \[
        q_{\cK} \left( \norm{\xi'}^2\xi  -  \xi\cdot \rInner{\cD}{\xi'}{\xi'} \right)
        =
        q_{\cK}(\xi),
    \]
    we similarly have that $\lInner{\cB}{ \norm{\xi'}^2\xi  -  \xi\cdot \rInner{\cD}{\xi'}{\xi'} }{\xi}$ is an element of the fibre of $\cB$ over
    \(
    \Leoq[Z]{\cG}{z}{z} = r_{Z}(z).
    \)
    Thus, we have identified that the  inner  products in question are simply the inner products as in Equation~\eqref{eq:ip-on-cKz} on the spaces $K(z)=\cM_{x}\otimes_{C(s_{X}(x))}\cN_{y}$ and $K(z')$, which are \ib s by Assumption~\ref{item:FE:SMEs}  on $\cM$ and $\cN$:
    \[
        \lInner[K(z)]{B(r(z))}{\xi \cdot \rInner[K(z')]{D(s(z'))}{\xi'}{\xi'} }{\xi}
            \overset{!}{\leq}
            \norm{\xi'}^2
            \,
            \lInner[K(z)]{B(r(z))}{\xi}{\xi}.
    \]
    The claim now follows from \cite[Cor.\ 2.22]{RW:Morita} and \cite[Prop.\ 3.16]{RW:Morita}.
\end{proof}


\begin{corollary}
    \label{cor:Inner-ineq-norm}
    If $\xi,\xi'\in\cK$, then 
      \(
        \norm{
        \lInner{\cB}{\xi}{\xi'}
        }_{\cB}
          \leq
          \norm{\xi}_{\cK}
          \,
          \norm{\xi'}_{\cK}
      \)
    resp.\
    \(
        \norm{
            \rInner{\cD}{\xi}{\xi'}
        }_{\cD}
          \leq
          \norm{\xi}_{\cK}
          \,
          \norm{\xi'}_{\cK},
    \)
    whenever the left-hand side of the inequality makes sense.
\end{corollary}

\begin{proposition}\label{prop:cK:ip}
    There exist continuous sesquilinear maps $\lInner{\cB}{\cdot}{\cdot}$ on $\cK\bfp{s}{s}\cK$ and $\rInner{\cD}{\cdot}{\cdot}$ on $\cK\bfp{r}{r}\cK$ which 
    are determined by
    \begin{align*}
       \lInner{\cB}{m\otimes n}{m'\otimes n'}
       :=\ &
       \linner{\cB}{m}{m'\cdot \linner{\cC}{n'}{n} }
       \\
       \text{and} \quad
       \rInner{\cD}{m\otimes n}{m'\otimes n'}
       :=\ &
       \rinner{\cD}{\rinner{\cC}{m'}{m}\cdot n}{n'}.
    \end{align*}
    These maps further satisfy for all $\xi_{i}\in K$, $b\in B ,$ $d\in D$ wherever it makes sense:
    \begin{enumerate}[style=multiline, labelwidth=.5cm,
    leftmargin=1cm, label=\textup{(\alph*)}]
       \item\label{item:cK:ip-fibre} 
       \(
           p_{\cB}
           \left(
              \lInner{\cB}{\xi_{1}}{\xi_{2}}
           \right)
           =
           \Leoq[ Z ]{\cG}{q_{\cK}(\xi_{1})}{q_{\cK}(\xi_{2})}
       \) and \(
           p_{\cD}
           \left(
              \rInner{\cD}{\xi_{1}}{\xi_{2}}
           \right)
           =
           \Reoq[ Z ]{q_{\cK}(\xi_{1})}{q_{\cK}(\xi_{2})}{\mathcal{K}},
       \)
       \item\label{item:cK:ip-adj} $\lInner{\cB}{\xi_{1}}{\xi_{2}}^* = \lInner{\cB}{\xi_{2}}{\xi_{1}}$ and $\rInner{\cD}{\xi_{1}}{\xi_{2}}^* = \rInner{\cD}{\xi_{2}}{\xi_{1}}$, and
       \item\label{item:cK:ip-C*linear} $\lInner{\cB}{b\cdot \xi_{1}}{\xi_{2}}=b\lInner{\cB}{\xi_{1}}{\xi_{2}}$ and $\rInner{\cD}{\xi_{1}}{\xi_{2}\cdot d}=\lInner{\cD}{\xi_{1}}{\xi_{2}}d$.
       \item\label{item:cK:ip-adjointable} $\lInner{\cB}{\xi_{1}\cdot d}{\xi_{2}}=\lInner{\cB}{\xi_{1}}{\xi_{2}\cdot d^*}$ and $\rInner{\cD}{b\cdot \xi_{1}}{\xi_{2}}=\lInner{\cD}{\xi_{1}}{b^*\cdot \xi_{2}}$.
       \item\label{item:cK:ip-FHE2e} 
       If $q_\cK(\xi_2)=q_\cK(\xi_3)$,
       \(
    \lInner{\cB}{\xi_{1}}{\xi_{2}} \lInner{\cB}{\xi_3}{\xi_4} =  \lInner{\cB}{\xi_{1} \cdot \rInner{\cD}{\xi_{2}}{\xi_3} }{\xi_4}\) and \(
    \rInner{\cD}{\xi_{1}}{\xi_2} \rInner{\cD}{\xi_3}{\xi_4} =  \rInner{\cD}{\xi_{1}}{\lInner{\cB}{\xi_2}{\xi_3} \cdot \xi_4}.
    \)
    \end{enumerate}
\end{proposition}

\begin{proof}
    As always, we will only deal with the left (pre-)inner product $\lInner{\cB}{\cdot}{\cdot}$. 
    We have already seen in Lemma~\ref{lem:cK:ip-on-odot+adjoint} for $(x,y),(x',y')\in Z$ with $s_{Y}(y)=s_{Y}(y')$ that there exists a map
    \[
       \lInner{\cB}{\cdot}{\cdot}\colon\quad
       \left(\cM_{x}\odot \cN_{y}\right)
       \times \left(\cM_{x'}\odot \cN_{y'}\right)
       \to 
       \cB_{g}
    \]
    with the defining property. Note that the maps $F_{m',n'}$ that  we used in its definition factor through the $\cC\z$-balancing because of Condition~\ref{item:FHE:ip:adjointables} for $\cM$ and Condition~\ref{item:FE:ip:C*linear} for $\cN$:
       \begin{align*}
           F_{m',n'}((m\cdot c)\odot n)
           &=
           \linner{\cB}{m\cdot c}{m'\cdot \linner{\cC}{n'}{n} }
           =
           \linner{\cB}{m}{(m'\cdot \linner{\cC}{n'}{n})\cdot c^* }
           \\&
           =
           \linner{\cB}{m}{m'\cdot (\linner{\cC}{n'}{n} c^*) }
           =
           \linner{\cB}{m}{m'\cdot \linner{\cC}{n'}{c\cdot n} }
           =
           F_{m',n'}(m\odot (c\cdot n)).
       \end{align*}
       Similarly, the map $F^*$ factors through the $\cC\z$-balancing because of 
       Condition~\ref{item:FE:ip:C*linear} for $\cN$:
       \begin{align*}
           F^*\left((m'\cdot c)\odot n'\right)\left(m\odot n\right)
           &=
           \linner{\cB}{m}{(m'\cdot c)\cdot \linner{\cC}{n'}{n}}^*
           \\&=
           \linner{\cB}{m}{m'\cdot \linner{\cC}{c\cdot n'}{n}}^*
           =
           F^*\left(m'\odot (c\cdot n')\right)\left(m\odot n\right)
       \end{align*}
        Consequently, we may define the pre-inner product on the {\em balanced} spaces,
    \[
       \lInner{\cB}{\cdot}{\cdot}\colon\quad
       \left(\cM_{x}\odot_{\cC\z} \cN_{y}\right)
       \times \left(\cM_{x'}\odot_{\cC\z} \cN_{y'}\right)
       \to 
       \cB_{g}.
    \]
    By Corollary~\ref{cor:Inner-ineq-norm}, we have
       \(
         \norm{\lInner{\cB}{\xi}{\xi'}}
        \leq 
        \norm{
            \xi
        }_{\cK}
        \,
         \norm{
            \xi'
        }_{\cK}
        ,\)
       which shows that that map then extends to a sesquilinear map
    \[
       \lInner{\cB}{\cdot}{\cdot}\colon\quad
       \cK_{(x,y)}
       \times \cK_{(x',y')}
       \to 
       \cB_{g}
    \]
    which also satisfies $\norm{\lInner{\cB}{\xi}{\xi'}}\leq\norm{\xi}\,\norm{\xi'}$.
    Since $(x,y),(x',y')$ were arbitrary with $s_{Y}(y)=s_{Y}(y')$, we get the map
       \[
           \lInner{\cB}{\cdot}{\cdot}\colon\quad
           \cK \bfp{s}{s} \cK \to \cB
           .
       \]
       
       \medskip 
       
       To see that this map is continuous, assume that $(\xi_{\lambda},\xi'_{\lambda})$ is a net in $\cK \bfp{s}{s} \cK$ converging to $(\xi,\xi')$. Let $q_{\cK}(\xi_{\lambda},\xi'_{\lambda})=(x_{\lambda},y_{\lambda},x'_{\lambda},y'_{\lambda})$ and $q_{\cK}(\xi,\xi')=(x,y,x',y')$, and write
       \(
        g_{\lambda}:=\Leoq{\cG}{x_{\lambda}}{x'_{\lambda} \cdot \Leoq{\cH}{y'_{\lambda}}{y_{\lambda}}}
       \)
       and
       \(
        g:=\Leoq{\cG}{x}{x' \cdot \Leoq{\cH}{y'}{y}}.
       \)
       By continuity of $q_{\cK}$, of $\Leoq[Y]{\cH}{\cdot}{\cdot}$, of the $\cH$-action on $X$, and of $\Leoq[X]{\cG}{\cdot}{\cdot}$, it follows that $g_{\lambda}\to g$. 
       Fix $\epsilon>0.$ By Lemma~\ref{lm.conv.equivalence.conditions}, \ref{item:convergence3.gamma}$\implies$\ref{item:convergence1}, it suffices to fix a section $\sigma$ of $\cK$ for which
       \begin{equation}\label{eq:choice-of-sigma}
        \norm{
            \lInner{\cB}{\xi}{\xi'}
            -
            \sigma (g)
          }
          <\frac\epsilon3
       \end{equation}
       and to show that there exists $\lambda_0$ such that for all $\lambda\geq\lambda_0$, we have
       \(
         \norm{
            \lInner{\cB}{\xi_{\lambda}}{\xi'_{\lambda}}
            -
            \sigma (g_{\lambda})
            }
          < \epsilon.
       \)
        The proof will follow the same ideas as that of continuity of the left $\cB$-action on $\cK$ (see Proposition~\ref{prop:cK:actions}). Fix sections $\tau_{i},\tau'_{i}$  of $\cM$ and $\kappa_{i},\kappa'_{i}$ of $\cN$ such that for the sections $\alpha:=\sum_{i} \tau_{i} \otimes\kappa_{i}$ and $
            \alpha' :=\sum_{i} \tau'_{i} \otimes\kappa'_{i}$ of $\cK$,
        we have
       \begin{equation}\label{eq:choice of alpha and alpha'}
            \norm{\xi - \alpha (x,y)} 
            <
            \frac{\epsilon}{6(1+\norm{\xi'})}
            \quad\text{and}\quad
            \norm{\xi' - \alpha' (x',y')} 
            <
            \frac{\epsilon}{6(1+\norm{\alpha (x,y)})}.
       \end{equation}

        We write
       \begin{align*}
            \lInner{\cB}{\alpha (x,y)}{\alpha' (x,'y')}
            =\ &
            \lInner{\cB}{\alpha (x,y)}{\alpha' (x,'y')-\xi' }
            +
            \lInner{\cB}{
                \alpha (x,y) - \xi
                }{
                \xi' 
                }
            +
            \lInner{\cB}{\xi }{\xi' }
            .
       \end{align*}
       Using the triangle inequality and that $\norm{\lInner{\cB}{\eta}{\eta'}}\leq\norm{\eta}\,\norm{\eta'}$, the choice of $\sigma$ in Equation~\eqref{eq:choice-of-sigma} and of $\alpha$ and $\alpha'$ in Equation~\eqref{eq:choice of alpha and alpha'} then implies
       \begin{align}
            \norm{\lInner{\cB}{\alpha (x,y)}{\alpha' (x,'y')}
            -
            \sigma(g)}
            \leq
            \
            &
            \norm{\alpha (x,y)}
            \,
            \norm{\alpha' (x,'y')-\xi'}
            +
            \norm{
                \alpha (x,y)-\xi 
                }
                \,
            \norm{
                \xi' 
                }
             +
            \norm{\lInner{\cB}{\xi }{\xi' }
            -
            \sigma(g)}\notag
            \\
            <\ &
            \frac{2\epsilon}{3}
            \label{ineq:alpha-close-to-sigma}
            .
       \end{align}

       As $\xi_{\lambda}\to\xi,\xi_{\lambda}'\to\xi'$ in $\cK$ and as $\alpha,\alpha'$ are continuous,
       we have by Lemma~\ref{lm.conv.equivalence.conditions}, \ref{item:convergence1}$\implies$\ref{item:convergence2}, that there exists $\lambda_1$ such that for all $\lambda\geq\lambda_1$, we have
       \begin{gather}
       \label{eq:cty-of-ip:1}
            \norm{
                \xi_{\lambda}
            }
            \leq
            \norm{
                \xi
            }
            \quad\text{and}\quad
            \norm{
                \xi_{\lambda}'
            }
            \leq
            \norm{
                \xi'
            }
            \\
            \label{eq:cty-of-ip:2}
            \norm{
                \alpha (x_{\lambda},y_{\lambda})
            }
            \leq
            \norm{
                \alpha (x,y)
            }
            \quad\text{and}\quad
            \norm{
                \alpha (x_{\lambda}',y_{\lambda}')
            }
            \leq
            \norm{
                \alpha (x',y')
            }
            \\
            \label{eq:cty-of-ip:3}
            \norm{
                \xi_{\lambda}
                -
                \alpha(x_{\lambda},y_{\lambda})
            }
            \leq
            \norm{
                \xi
                -
                 \alpha (x,y)
            }
            \quad\text{and}\quad
            \norm{
                \xi_{\lambda}'
                -
                \alpha(x_{\lambda}',y_{\lambda}')
            }
            \leq
            \norm{
                \xi'
                -
                 \alpha (x',y')
            }
            .
        \end{gather}
       This time, we write 
       \begin{align*}
            \lInner{\cB}{\xi_{\lambda}}{\xi'_{\lambda}}
            =\ &
            \lInner{\cB}{
                \xi_{\lambda}-\alpha (x_{\lambda},y_{\lambda})
                }{
                \xi'_{\lambda}
                }
          +
            \lInner{\cB}{\alpha (x_{\lambda},y_{\lambda})}{\xi'_{\lambda}-\alpha(x_{\lambda}',y_{\lambda}')}
          +
            \lInner{\cB}{\alpha (x_{\lambda},y_{\lambda})}{\alpha (x_{\lambda}',y_{\lambda}')},
       \end{align*}
       so that similarly to before in Inequality~\eqref{ineq:alpha-close-to-sigma}, we get
       \begin{align*}
        \norm{
            \lInner{\cB}{\xi_{\lambda}}{\xi'_{\lambda}}
            -
            \sigma (g_{\lambda})
          }
          \leq\ &
            \norm{
                \xi_{\lambda}-\alpha (x_{\lambda},y_{\lambda})
                }
            \,
            \norm{
                \xi'_{\lambda}
                }
          +
            \norm{\alpha (x_{\lambda},y_{\lambda})}
            \,
            \norm{\xi'_{\lambda}-\alpha (x_{\lambda}',y_{\lambda}')}
          \\&+
            \norm{
                \linner{\cB}{\alpha (x_{\lambda},y_{\lambda})}{\alpha (x_{\lambda}',y_{\lambda}')}
                -
            \sigma (g_{\lambda})
            }.
       \end{align*}
       By the second inequality in~\eqref{eq:cty-of-ip:1}, the first in~\eqref{eq:cty-of-ip:3}, and by choice of $\alpha$ in~\eqref{eq:choice of alpha and alpha'}, the first summand for $\lambda\geq \lambda_{1}$ is bounded by
       \begin{align*}
           \norm{
                \xi_{\lambda}-\alpha (x_{\lambda},y_{\lambda})
                }
            \,
            \norm{
                \xi'_{\lambda}
                }
            &\leq 
            \norm{
                \xi-\alpha (x,y)
                }
            \,
            \norm{
                \xi'
                }
            <\frac\epsilon6
            .
       \end{align*}
       Similarly, by the first inequality in~\eqref{eq:cty-of-ip:2}, the second in~\eqref{eq:cty-of-ip:3}, and by choice of $\alpha'$ in~\eqref{eq:choice of alpha and alpha'}, the second summand for $\lambda\geq \lambda_{2}$ is bounded by
       \begin{align*}
           \norm{\alpha (x_{\lambda},y_{\lambda})}
            \,
            \norm{\xi'_{\lambda}-\alpha (x_{\lambda}',y_{\lambda}')}
            &\leq 
            \norm{\alpha (x,y)}
            \,
            \norm{\xi'-\alpha(x',y')}
            <\frac\epsilon6.
       \end{align*}
        Using the last two inequalities in the preceding one, we have for $\lambda\geq \lambda_{1}$
       \begin{align}\label{ineq:net:xi-close-to-sigma}
        \norm{
            \lInner{\cB}{\xi_{\lambda}}{\xi'_{\lambda}}
            -
            \sigma (g_{\lambda})
          }
          \leq\ &
          \frac{\epsilon}{3}
          +
            \norm{
                \lInner{\cB}{ \alpha (x_{\lambda},y_{\lambda})}{\alpha (x_{\lambda}',y_{\lambda}')}
                -
                \sigma (g_{\lambda})
            }.
       \end{align}
        By definition of the inner product and of $\alpha$ and $\alpha'$, we have
        \begin{align*}
            \lInner{\cB}{\alpha(x_{\lambda},y_{\lambda})}{\alpha(x_{\lambda}',y_{\lambda}')}
            =\ &
            \sum_{i,j}
            \linner{\cB}{\tau_{i}(x_{\lambda})}{\tau'_{j} (x_{\lambda}')\cdot \linner{\cC}{\kappa'_{j} (y_{\lambda}')}{\kappa_{i} (y_{\lambda})} }
        \end{align*}
        By continuity of all maps involved in the right-hand side, we know that this converges to 
        \begin{align*}
            \lInner{\cB}{\alpha(x,y)}{\alpha(x',y')}
            =\ &
            \sum_{i,j}
            \linner{\cB}{\tau_{i}(x)}{\tau'_{j} (x')\cdot \linner{\cC}{\kappa'_{j} (y')}{\kappa_{i} (y)} }
        \end{align*}
        in $\cB$. Thus, there exists $\lambda_2$ such that for all $\lambda\geq \lambda_2$, we have
        \begin{align*}
            \norm{\lInner{\cB}{\alpha (x_{\lambda},y_{\lambda})}{\alpha (x_{\lambda}',y_{\lambda}')}
            -
            \sigma(g_{\lambda})}
            \notag
            \leq\ 
            \norm{\lInner{\cB}{\alpha (x,y)}{\alpha (x',y')}
            -
            \sigma(g)}            <\frac{2\epsilon}{3}
            .
      \end{align*}
      Combining the above with Inequality~\eqref{ineq:net:xi-close-to-sigma}, we arrive at $\norm{
            \lInner{\cB}{\xi_{\lambda}}{\xi'_{\lambda}}
            -
            \sigma (g_{\lambda})
          } <\epsilon$ for all $\lambda\geq \lambda_{1},\lambda_{2}$, as claimed. The map $\lInner{\cB}{\cdot}{\cdot}$ is hence continuous.

    \medskip
    
    \pagebreak[0]
    It remains to check Conditions~\ref{item:cK:ip-fibre}--  \ref{item:cK:ip-FHE2e}. Condition~\ref{item:cK:ip-fibre} follows directly from Lemma~\ref{lem:cK:ip-on-elementary-tensors}, Condition~\ref{item:cK:ip-adj} from Lemma~\ref{lem:cK:ip-on-odot+adjoint}, and Condition~\ref{item:cK:ip-FHE2e} from Lemma~\ref{lem:FHE:2e for cK}. For the other two, it suffices to restrict our attention to elementary tensors $\xi=m
       \otimes n$ and $\xi'= m'\otimes n'$: 
       For \ref{item:cK:ip-C*linear}, we compute
       \begin{alignat*}{2}
           \lInner{\cB}{ b\cdot \xi}{\xi'}
           &=
           \lInner{\cB}{(b\cdot m)\otimes n}{m'\otimes n'}
           \quad&&
           \text{{by def.\ of the action on $\cK$}}
           \\
           &=
           \linner{\cB}{b\cdot m}{m'\cdot \linner{\cC}{n'}{n} }
           &&
           \text{{by def.\ of the inner product on $\cK$}}
           \\
           &=
           b\linner{\cB}{m}{m'\cdot \linner{\cC}{n'}{n} }
           &&
           \text{{by \ref{item:FE:ip:C*linear} for~$\cM$}}
           \\
           &=
           b \lInner{\cB}{\xi}{\xi'}.
        \end{alignat*}
    Lastly, for \ref{item:cK:ip-adjointable}, we compute 
    \begin{alignat*}{2}
        \lInner{\cB}{\xi\cdot d}{\xi'}
        &=
         \lInner{\cB}{ m\otimes (n\cdot d)}{ m'\otimes n'}
        \quad&&
        \text{{by def.\ of the action on $\cK$}}
           \\
           &=
         \linner{\cB}{m}{m'\cdot \linner{\cC}{n'}{n\cdot d} }
        &&
           \text{{by def.\ of the inner product on $\cK$}}
           \\
           &=
         \linner{\cB}{m}{m'\cdot \linner{\cC}{n'\cdot d^*}{n} }
         &&
           \text{{by Cond.~~\ref{item:FHE:ip:adjointables} for~$\cN$}}
           \\
           &=
         \lInner{\cB}{ m\otimes n}{ m'\otimes (n'\cdot d^*)}
        &&
           \text{{by def.\ of the inner product on $\cK$}}
           \\
           &=\lInner{\cB}{\xi}{\xi'\cdot d^*}.
    \end{alignat*}
    This concludes our proof.
\end{proof}

We can now come to our first main theorem:
\pagebreak[3]\begin{theorem}\label{thm:cK-is-word}
    Assume we are given 
    \begin{itemize}
    \item three saturated Fell bundles $\cB, \cC, \cD$ over \LCH\ \etale\ groupoids $\cG,\cH,\mathcal{K}$, respectively;
    \item a $(\cG,\cH)$-\pe\ $X$ and an $(\cH,\mathcal{K})$-\pe\ $Y$;
    \item a $(\cB,\cC)$-\hequiv\ $\cM$ over $X$
	and a $(\cC,\cD)$-\hequiv\ $\cN$ over $Y$.
\end{itemize}
    Let $\cK$ be the bundle over  the $(\cG,\mathcal{K})$-\pe\ $X\bfp{s}{r}Y$  defined in Lemma~\ref{lem:topology-on-cK}. Then $\cK$ is a $(\cB,\cD)$-\hequiv\ when equipped with the actions defined in Proposition~\ref{prop:cK:actions} and the inner products defined in Proposition~\ref{prop:cK:ip}.
\end{theorem}
\begin{proof}
    In this proof, all Items are references to Definition~\ref{def:FBequivalence} unless otherwise specified.
    
    We have seen in  Lemma~\ref{lem:product-of-preequiv} that $ Z =X\bfp{s}{r}Y$ is a groupoid \pe, and in Lemma~\ref{lem:topology-on-cK} that $\cK$ is a \USCBb\ over $ Z $. By Proposition~\ref{prop:cK:actions}, $\cK$ carries commuting $\cB$- and $\cD$-actions, so Item~\ref{item:FE:actions} is satisfied. Conditions~\ref{item:cK:ip-adj} and \ref{item:cK:ip-C*linear} in Proposition~\ref{prop:cK:ip} prove Items~\ref{item:FE:ip:adjoint} and \ref{item:FE:ip:C*linear} respectively, while
    Conditions~\ref{item:cK:ip-fibre}, \ref{item:cK:ip-adjointable}, and~\ref{item:cK:ip-FHE2e} in the same proposition prove Items~\ref{item:FHE:ip-fibre}, \ref{item:FHE:ip:adjointables}, and~\ref{item:FHE:2e} of Definition~\ref{def:FBword}.
    
    Lastly, the fibre of $\cK$ over $(x,y)\in  Z $ is $\cM_x \otimes_{C(u)}\cN_{y}$ with $u=s_{X}(x)=r_{Y}(y)$; as $\cM_{x}$ is a $B(r_{X}(x))-C(u)$-\ib\ and $\cN_{y}$ is a $C(u)-D(s_{Y}(y))$-\ib\ by Item~\ref{item:FE:SMEs} for $\cM$ resp.\ $\cN$, we conclude that $\cK_{(x,y)}$ is a $B(r_{X}(x))-D(s_{Y}(y))$-\ib, as needed for $\cK$ for Item~\ref{item:FE:SMEs}.
\end{proof}

\begin{remark}
\label{rmk:why-cP-is-not-faulty}
    One may be tempted to say that, if $X=\X$ and $Y=\Y$ are {\em equivalences} of groupoids and $\cM$ and $\cN$ {\em equivalences} of Fell bundles, then $\cK$ is a $(\cB,\cD)$-{\em equivalence}. Unfortunately, that is not the case: Just like one has to take a quotient of $\X\bfp{s}{r}\Y$ to make the groupoid equivalence~$\X\bfp{}{\cH}\Y$, we will have to take a quotient of $\cK$ to get a Fell bundle equivalence. 
    
    As $ Z =\X\bfp{s}{r}\Y$ is not a $\cG-\mathcal{K}$ equivalence as pointed out in Remark~\ref{rm.product.is.not.eq}, it does not admit a map $\leoq[ Z ]{\cG}{\cdot}{\cdot}$;
    in particular, we cannot ask for $\cK$ to satisfy Item~\ref{item:FE:ip:fibre}. But we can remedy that to some extent: 
    Suppose we are given $m\otimes n\in\cK_{(x,y)}$ and  $m'\otimes n'\in\cK_{(x',y')}$ where $r_{X}(x)=r_{X}(x')$ and $s_{Y}(y)=s_{Y}(y')$.
    Let $h:=\leoq[\Y]{\cH}{y'}{y}$, the unique element of $\cH$ such that $y' = h\cdot y$, and let
    \(
       g
       :=
       \leoq[\X]{\cG}{x}{x'\cdot h},
    \)
    the unique element of $\cG$ such that 
    \(
       x = g \cdot \left(x'\cdot h\right).
    \)
    Then
    \begin{equation*}
       (q_{\cM}(m),q_{\cN}(n)) = (x,y) = (g \cdot \left( x'\cdot h\right), h\inv \cdot y').
    \end{equation*}
    If $[\cdot]_{\ipscriptstyle\cH}$ denotes the quotient map $ Z =\X\bfp{s}{r}\Y\to \X\ast_{\cH}\Y=:\mathcal{Z}$, then we at least have
    \begin{align*}
       \leoq[\mathcal{Z}]{\cG}{[q_{\cK}(m\otimes n)]_{\ipscriptstyle\cH}}{[q_{\cK}(m'\otimes n')]_{\ipscriptstyle\cH}}
       &=
       \leoq[\mathcal{Z}]{\cG}{[(x,y)]_{\ipscriptstyle\cH}}{[(x',y')]_{\ipscriptstyle\cH}}
        \\&
       =
       \leoq[\mathcal{Z}]{\cG}{[(g \cdot \left( x'\cdot h\right), h\inv \cdot y')]_{\ipscriptstyle\cH}}{[(x',y')]_{\ipscriptstyle\cH}}
       \\&=
       \leoq[\mathcal{Z}]{\cG}{[(g \cdot x', y')]_{\ipscriptstyle\cH}}{[(x',y')]_{\ipscriptstyle\cH}}
       \\&=
       \leoq[\mathcal{Z}]{\cG}{g \cdot [(x', y')]_{\ipscriptstyle\cH}}{[(x',y')]_{\ipscriptstyle\cH}}
       =
       g
        \\&
       =
       p_{\cB}
       \left(
           \lInner{\cB}{m\otimes n}{m'\otimes n'}
       \right).
    \end{align*}
    
    We will soon adjust $\cK$ to give us an actual equivalence between $\cB$ and~$\cD$.
\end{remark}

\begin{remark}
    Just as with balanced tensor products of Hilbert $\textrm{C}^*$-modules, one can show that the above construction of taking a balanced tensor product of Fell bundle \hequiv s is associative. To be precise, if $V,X,Y$ are groupoid \pe s such that $V\bfp{s}{r}X$ and $X\bfp{s}{r}Y$ make sense, and if $\mathscr{L}$, $\cM$, and $\cN$  are 
    \hequiv s over those spaces, then there is a natural isomorphism
    \[
        (\mathscr{L}\otimes_{\cB\z}\cM)\otimes_{\cC\z}\cN
        \cong 
        \mathscr{L}\otimes_{\cB\z}(\cM\otimes_{\cC\z}\cN)
    \]
    of \USCBb s.
\end{remark}

\subsection{Some isomorphic bundles}

As before, we will let $\cK$ denote $\cM\otimes_{\cC\z}\cN$, the Fell bundle \hequiv\ described in Theorem~\ref{thm:cK-is-word}.
The goal of this subsection is to construct a map that allows us to `collapse' certain fibres of $\cK$: We intend to mimic the balancing over $C(u)$ in the tensor product $\cM_{x'} \otimes \cN_{y'}$, where $s_{X}(x')=u=r_{Y}(y')$, and extend it to a balancing over $C(h)$ for any $h\in \cH$. In other words, we would like to identify the fibres $\cM_{x}\otimes_{C(u)}\cN_{hy}$ and $\cM_{xh}\otimes_{C(v)}\cN_{y}$, where $u=r_{\cH}(h)$ and $v=s_{\cH}(h)$. A na{\"i}ve approach would be to mod out by a space spanned by all elements of the form
\[
    (m\cdot c) \otimes n - m\otimes (c\cdot n),
\]
where $c\in C(h)$ for some $h$. However, this 
is doomed to fail:  the displayed difference of elementary tensors does not make sense since the two summands live in distinct vector spaces, $\cM_{x}\otimes_{C(u)}\cN_{hy}$ resp.\ $\cM_{xh}\otimes_{C(v)}\cN_{y}$. Therefore, we must construct a function on $\cK$ that allows us to identify these fibres; see Theorem~\ref{thm:Psi-cts}.

\begin{definition}\label{def:t}
    Let $t_{ Z }\colon Z=X \bfp{s}{r} Y \to \cH\z$  be defined by $t_{ Z }(x,y)=s_{X}(x)= r_{Y}(y)$, and let $t_{\cK}:=t_{ Z }\circ q_{\cK}$. Note that 
    \[
       t_{ Z }(g\cdot x,y) = t_{ Z }(x,y)
       =
       t_{ Z }(x,y\cdot k)
    \]
    for all $g\in\cG$ and $k\in\mathcal{K}$ where it makes sense.
\end{definition}
\begin{lemma}\label{lem:MCN-iso-HMN}
    There exists an isomorphism
    \begin{align*}
       (\cM\otimes_{\cC\z}\cC)\otimes_{\cC\z} \cN &\overset{\Omega}{\longrightarrow}
       \cH  \bfp{s}{t} \cK
    \intertext{of \USCBb s determined by}
       (m\otimes c) \otimes n &\longmapsto (p_{\cC}(c), (m\cdot c) \otimes n )
    \end{align*}
    covering the homeomorphism
    \begin{align*}
       \omega\colon
       \quad
       (X \bfp{s}{r} \cH) \bfp{s}{r}  Y
       &\to
       \cH \bfp{s}{t}  (X \bfp{s}{r} Y)
       ,
       \quad
       (x,h,y)\mapsto (h,x\cdot h,y).
    \end{align*}
    Here, $\cH  \bfp{s}{t} \cK$ is as defined in Definition~\ref{def:beefing-up-a-USC-bundle}.
\end{lemma}

In the above, we have tacitly used Theorem~\ref{thm:cK-is-word}: Since $\cM$ and $\cC$ are both \hequiv s of  Fell bundles, it follows that $\cM\otimes_{\cC\z}\cC$ is also a \hequiv, namely over the $(\cG,\cH)$-\pe\ $X\bfp{s}{r}\cH$; otherwise, we would not be able to consider $(\cM\otimes_{\cC\z}\cC)\otimes_{\cC\z} \cN$.

\begin{proof}
    For this proof, let \(\cL := (\cM\otimes_{\cC\z}\cC) \otimes_{\cC\z}\cN\).
    The fibre of $\cL$ at some element $(x,h,y)$ of $(X \bfp{s}{r} \cH) \bfp{s}{r}  Y$ is by definition (Lemma~\ref{lem:topology-on-cK}) exactly 
    \[
       \cL_{(x,h,y)}
       =
       (\cM \otimes_{\cC\z} \cC)_{(x,h)} \otimes_{\cC\z}\cN_{y}.
    \]
    By another application of the same lemma (this time with $\cN:=\cC$ and $Y:=\tensor*[_{\cH}]{\cH}{_{\cH}}$), we further have
    \[
       (\cM \otimes_{\cC\z} \cC)_{(x,h)}
       =
       \cM_x \otimes_{\cC\z} \cC_{h}.
    \]
    As a consequence of Remark~\ref{rmk:MW-6.2-for-words}, the fibrewise map determined by $(m\otimes c)\otimes n \mapsto (m\cdot c)\otimes n$ is not only well defined but induces an isomorphism
    \[
       \Omega_{(x,h, y)}\colon\quad
       \cL_{(x,h,y)}
       =
       (\cM_x \otimes_{\cC\z} \cC_{h}) \otimes_{\cC\z}\cN_{y}
       \cong
       \{h\}\times \cM_{x\cdot h} \otimes_{\cC\z}\cN_{y}
       =
       \bigl(\cH  \bfp{s}{t} \cK\bigr)_{\omega(x,h,y)}
    \]
    of \ib s. In particular, $\norm{\Omega_{(x,h, y)}(\xi)}=\norm{\xi}$ for all $\xi\in  \cL_{(x,h,y)}.$ These fibrewise maps can be `stitched together' to a global, bijective map $\Omega$. To see that $\Omega$ is an isomorphism of Banach bundles, we will show that $\Omega$ is continuous; from 
    Proposition~\ref{prop.13.17}
    and from the fact that $\Omega$ is isometric, it will follow that $\Omega\inv$ is likewise continuous, proving that $\Omega$ is the claimed isomorphism.
    
    \smallskip
    
    For continuity of $\Omega$,  
    Proposition~\ref{prop.13.16}
    asserts that it suffices to find a 
    family
	$\Gamma$ of continuous cross-sections of $\cL$ such that
    \begin{enumerate}[label=(\arabic*)]
       \item\label{item:dense span in cL} for each $(x,h,y)\in (X \bfp{s}{r} \cH) \bfp{s}{r}  Y$, 
       the linear span of 
       $\{\gamma(x,h,y)\,\mid\,\gamma\in\Gamma\}$ is dense in $\cL_{(x,h,y)}$, and
       \item\label{item:concatenation cts on cL} for all $\gamma\in\Gamma$, $\Omega\circ\gamma\circ \omega\inv$ is a continuous cross-section of $\cH  \bfp{s}{t} \cK
    $.
    \end{enumerate}
    We  let $\Gamma$ be the collection of elements of the form $\sigma:=(\sigma_{\cM}\otimes\sigma_{\cC})\otimes\sigma_{\cN}$ for $\sigma_{\cM}\in \Gamma_{0} (X;\cM)$ etc.; these cross-sections are, by construction of the topology on  $\cL$, indeed continuous and satisfy~\ref{item:dense span in cL}.
    
    For~\ref{item:concatenation cts on cL}, suppose that  $(h_{\lambda}, x_{\lambda},y_{\lambda})\to (h,x,y)$ is a convergent net in $(X \bfp{s}{r} \cH) \bfp{s}{r}  Y$.
    We have to show that, for any fixed $\sigma$ as above, the net $$s_{\lambda}:=(\Omega\circ \sigma\circ\omega\inv) (h_{\lambda},x_{\lambda},y_{\lambda}) \quad\text{converges to}\quad s:=(\Omega\circ \sigma\circ\omega\inv) (h,x,y).$$
    We compute:
    \begin{align*}
    	s &=
    	(\Omega\circ \sigma\circ\omega\inv) (h,x,y)
    	=
    	(\Omega\circ \sigma) (x\cdot h\inv,h,y)
    	=
    	\Omega \bigl( (\sigma_{\cM}(x\cdot h\inv)\otimes\sigma_{\cC}(h))\otimes\sigma_{\cN}(y)\bigr)
    	\\&=
    	\bigl(h, [\sigma_{\cM}(x\cdot h\inv)\cdot\sigma_{\cC}(h)]\otimes\sigma_{\cN}(y)\bigr)
    	\intertext{and similarly}
	    s_{\lambda} &=
	\bigl(h, [\sigma_{\cM}(x_{\lambda}\cdot h_{\lambda}\inv)\cdot\sigma_{\cC}(h_{\lambda})]\otimes\sigma_{\cN}(_{\lambda}y)\bigr).
    \end{align*}
    We know by continuity of $\sigma_{\cM}$, $\sigma_{\cC}$, and the right-action of $\cH$ on $X$ that 
    \begin{alignat*}{2}
    	m_{\lambda}:= \sigma_{\cM}(x_{\lambda}\cdot h_{\lambda}\inv)\cdot\sigma_{\cC}(h_{\lambda}) &\longrightarrow \sigma_{\cM}(x\cdot h\inv)\cdot\sigma_{\cC}(h)=: m\quad &&\text{ in }\cM, \text{ and}\\
    	n_{\lambda}:= \sigma_{\cN}(y_{\lambda}) &\longrightarrow \sigma_{\cN}(y)=:n&&\text{ in }\cN.
    \end{alignat*}
    By Lemma~\ref{lem:topology-on-cMcN-w-nets}, it follows that 
    \(
    	m_{\lambda}\otimes n_{\lambda} \to m\otimes n
    \)
    in $\cK$. By Lemma~\ref{lm.conv.equivalence.conditions},  \ref{item:convergence1}$\implies$\ref{item:convergence2}, we conclude that for any $\tau_{\cM}\in \Gamma_{0} (X;\cM)$ and $\tau_{\cN}\in \Gamma_{0} (X;\cN)$, 
    \begin{equation}\label{eq:elementary in cK}
	    \limsup\norm{m_{\lambda}\otimes n_{\lambda} -  (\tau_{\cM}\otimes\tau_{\cN}) (x_{\lambda},y_{\lambda})}
	    \leq
	    \norm{m\otimes n - (\tau_{\cM}\otimes\tau_{\cN}) (x,y)}.
    \end{equation}
    But note that
    \begin{align*}
    	\norm{m\otimes n - (\tau_{\cM}\otimes\tau_{\cN}) (x,y)}_{\cK}
    	&=
    	\norm{s - [\mathrm{id}(\tau_{\cM}\otimes\tau_{\cN})] (h,x,y)}_{\cH\bfp{}{}\cK},
\end{align*}
	and similarly for the version with subscript $\lambda$.
   Thus, Equation~\eqref{eq:elementary in cK} is exactly a proof of convergence $s_{\lambda}\to s$ in $\cH\bfp{s}{t}\cK$.
	This proves continuity of $\Omega.$ 
\end{proof}

Clearly, we could have proved a fully analogous statement for the other set of parentheses. To be precise, there exists another isomorphism
\begin{align*}
       \Lambda\colon \quad
       \cM\otimes_{\cC\z}(\cC\otimes_{\cC\z} \cN) &\longrightarrow
       \cK  \bfp{t}{r} \cH
    \\
       m\otimes (c \otimes n) &\longmapsto (m\otimes (c\cdot n ), p_{\cC}(c))
    \end{align*}
    of \USCBb s,   covering the homeomorphism
    \begin{align*}
       \lambda\colon
       \quad
       (X \bfp{s}{r} \cH) \bfp{s}{r}  Y
       &\to
       (X \bfp{s}{r} Y) \bfp{t}{r} \cH
       ,
       \quad
       (x,h,y)\mapsto (x,h\cdot y, h).
    \end{align*}
We arrive at the following commutative diagram:
    \begin{equation}\label{diag:huge}
 \begin{tikzcd}[column sep = tiny]
       &[-45pt]
           (\cM\otimes_{\cC\z} \cC)\otimes_{\cC\z} \cN
           \ar[ddd, "\cong", "\Omega"']
           \ar[ld]
           \ar[rrr,"\cong"]
       &[-60pt]&&[-60pt]
           \cM\otimes_{\cC\z} (\cC\otimes_{\cC\z} \cN)
           \ar[ddd,"\cong"', "\Lambda"]
           \ar[rd]
       &[-45pt]
       \\
           (X \bfp{s}{r} \cH) \bfp{s}{r} Y
           \ar[d, "\approx", "\omega"']
       &&
           (m\otimes c)\otimes n
           \ar[r, mapsto]
           \ar[d, mapsto]
       &
           m\otimes (c\otimes n)
           \ar[d, mapsto]
       &&
           X \bfp{s}{r} (\cH \bfp{s}{r} Y)
           \ar[d, "\approx"', "\lambda"]
       \\
           \cH  \bfp{s}{t}  (X \bfp{s}{r}  Y)
       &&
           (p(c), (m\cdot c) \otimes n)
       &
           (m\otimes (c\cdot n), p(c))
       &&
           (X \bfp{s}{r} Y) \bfp{t}{r} \cH
       \\
       &
           \cH  \bfp{s}{t} ( \cM \otimes_{\cC\z} \cN)
           \ar[lu]
           \ar[rrr, dashed, "\Psi"]
       &&&
           ( \cM \otimes_{\cC\z} \cN)  \bfp{t}{r} \cH
           \ar[ru]
       &
    \end{tikzcd}
    \end{equation}
    In particular, since the vertical maps and the horizontal map at the top are isomorphisms, it follows that there is a unique isomorphism on the bottom making the diagram commute.
To sum up: 
\pagebreak[3]\begin{theorem}\label{thm:Psi-cts}
     	Suppose $\cM$ and $\cN$ are {\hequiv s} and let $\cK=\cM\otimes_{\cC\z}\cN$ be the \hequiv\ constructed in Theorem~\ref{thm:cK-is-word}. Then the map
    \begin{align*}
       \Psi\colon\quad 
       \cH \bfp{s}{t}  \cK
       &\to
       \cK\bfp{t}{r}  \cH
       \\
       (h, \xi)&\mapsto (\Psi_{h} (\xi), h)
    \quad\text{
    determined by
    }\quad
    	\Psi_{h} ((m\cdot c)\otimes n) = m\otimes (c\cdot n)
    	\text{ for } c\in \cC_{h},
    \end{align*}
    is an isomorphism of \USCBb s covering the homeomorphism
    \begin{align*}
       \psi:=\lambda\circ\omega\inv\colon
       \quad
       \cH \bfp{s}{t}  (X \bfp{s}{r} Y)
       &\to
       (X \bfp{s}{r} Y) \bfp{t}{r} \cH
       \\
       (h,x,y)& \mapsto (xh\inv, h y, h).
    \end{align*}
    The map $\Psi$ has the following additional properties:
    \begin{enumerate}[label=\textup{($\Psi$\arabic*)}]
       \item\label{item:Psi:jointly-cts}
       If $(h_{\lambda}, \xi_{\lambda})\to (h,\xi)$ is a convergent net in $\cH \bfp{s}{t} \cK$, then $\Psi_{h_{\lambda}}(\xi_{\lambda})\to \Psi_{h}(\xi)$ in~$\cK$.
    \item\label{item:Psi:transitivity}
    	For appropriate $x\in X,y\in Y,$ and $(h',h)\in\cH^{(2)}$, the composition of
    	    \[
    	       \cM_{xh'h} \otimes_{\cC\z} \cN_{y}
    	       \overset{\Psi_{h}}{\longrightarrow}
    	       \cM_{xh'} \otimes_{\cC\z} \cN_{hy}
    	    \quad\text{
    	    with 
    	    }\quad
    	       \cM_{xh'} \otimes_{\cC\z} \cN_{hy}
    	       \overset{\Psi_{h'}}{\longrightarrow}
    	       \cM_{x} \otimes_{\cC\z} \cN_{h'hy}
    	    \]
    	    coincides with 
    	    \[
    	    	\cM_{xh'h} \otimes_{\cC\z} \cN_{y}
    	    	\overset{\Psi_{h'h}}{\longrightarrow}
    	    	\cM_{x} \otimes_{\cC\z} \cN_{h'hy}.
    	    \]
    \item\label{item:Psi:unit-identity}
       If $h=v\in\cH\z$ is a unit, then $\Psi_{h}$ is the identity on $\cM_{x} \otimes_{\cC\z} \cN_{y}$.
     \item\label{item:Psi:inverse}
     For appropriate $x,h,y$, the map
        \begin{align*}
            \Psi_{h}\colon \quad\cM_{x} \otimes_{\cC\z} \cN_{y}=\cM_{(xh\inv)h} \otimes_{\cC\z} \cN_{y}\longrightarrow \cM_{xh\inv} \otimes_{\cC\z} \cN_{hy} 
            \intertext{is inverse to}
            \Psi_{h\inv}\colon \quad
            \cM_{xh\inv} \otimes_{\cC\z} \cN_{hy} \longrightarrow \cM_{x} \otimes_{\cC\z} \cN_{h\inv(hy)}= \cM_{x} \otimes_{\cC\z} \cN_{y}.
        \end{align*}
    \item\label{item:Psi:biequivariant}
    The map $\Psi$ is bi-equivariant in the following sense: if $(h,\xi)\in\cH \bfp{s}{t}  \cK$ and if $b\in B$ and $d\in D$ are such that $s_{\cB}(b)=r_{\cK}(\xi)$ and $s_{\cK}(\xi)=r_{\cD}(d)$, then
    \[
        \Psi_{h}(b\cdot \xi)=
        b\cdot \Psi_{h}(\xi)
        \quad\text{and}\quad
        \Psi_{h}(\xi\cdot d)=
        \Psi_{h}(\xi)\cdot d.
    \]
    In other words, if $p_{\cB}(b)=g\in\cG$ and $p_{\cD}(d)=k\in\mathcal{K}$, then the following diagrams commute:
    \[
    \begin{tikzcd}
       \cM_{x\cdot h}\otimes \cN_{y} \ar[r, "\Psi_{h}"]\ar[d, "b\cdot"']& \cM_{x}\otimes \cN_{h \cdot y}\ar[d, "b\cdot"]
       \\
       \cM_{(g\cdot x)\cdot h}\otimes \cN_{y} \ar[r, "\Psi_{h}"']& \cM_{g\cdot x}\otimes \cN_{h \cdot y}
    \end{tikzcd}
    \quad
    \begin{tikzcd}
       \cM_{x\cdot h}\otimes \cN_{y} \ar[r, "\Psi_{h}"]\ar[d, "\cdot d"']& \cM_{x}\otimes \cN_{h \cdot y}\ar[d, "\cdot d"]
       \\
       \cM_{x\cdot h}\otimes \cN_{y\cdot k} \ar[r, "\Psi_{h}"']& \cM_{g\cdot x}\otimes \cN_{h \cdot (y\cdot k)}
    \end{tikzcd}
    \]
    \end{enumerate}
\end{theorem}

In Item~\ref{item:Psi:biequivariant}, we made use of the fact that our computations in Definition~\ref{def:t} imply
\begin{align*}
    t_{\cK}(b\cdot \xi)
    =
    t_{ Z }\bigl( q_{\cK}(b\cdot \xi)\bigr)
    =
    t_{ Z }\bigl( p_{\cB}(b)\cdot q_{\cK}(\xi)\bigr)
    =
    t_{ Z }\bigl( q_{\cK}(\xi)\bigr)
    =
    t_{\cK}(\xi),
\end{align*}
so that $\Psi_{h}(b\cdot \xi)$ makes sense if and only if $b\cdot \Psi_{h}(\xi)$ makes sense; similarly for $\xi\cdot d.$ 

\begin{proof}
The existence of the isomorphism follows from our prior discussion, so we only have to check the additional properties.
\begin{description}
    \item[Re \ref{item:Psi:jointly-cts}] $\Psi$ is continuous by construction as the concatenation of continuous functions. The claim now follows directly from the definition of the topology of the bundles involved.
    \item[Re \ref{item:Psi:transitivity}]
    For $ m \in \cM_{x}, c_{1}\in \cC_{h'}, c_{2}\in \cC_{h}, n \in \cN_{y}$, we have
       \[
           \Psi_{h} (( m c_{1})c_{2}\otimes  n ) =  m c_{1} \otimes c_{2} n .
       \]
       As this is an element of $(\cM_{x}\cdot \cC_{h'}) \odot (\cC_{h}\cdot \cN_{y})$, we know how $\Psi_{h'}$ acts:
       \[
           \Psi_{h'} ( m c_{1} \otimes c_{2} n ) =  m  \otimes c_{1}c_{2} n .
       \]
       On the other hand,
       \[
           \Psi_{h'h} ( m (c_{1}c_{2})\otimes  n ) =  m  \otimes (c_{1}c_{2}) n  = \Psi_{h'}\circ\Psi_{h} (( m c_{1})c_{2}\otimes  n ).
       \]
       Using linearity, we conclude that $\Psi_{h'h}$ and $\Psi_{h'}\circ\Psi_{h}$ coincide on the subspace $(\cM_{x}\cdot \cC_{h'}\cdot \cC_{h})\odot_{C(v)} \cN_{y}$ of $\cM_{xh'h} \otimes_{\cC\z} \cN_{y}$, which is dense by Remark~\ref{rmk:MW-6.2-for-words}. They thus agree everywhere.
       
    \item[Re \ref{item:Psi:unit-identity}]
    Follows since the module is balanced over $\cC_{v}=C(v)$.
    
     \item[Re \ref{item:Psi:inverse}]
     If $v= h\inv h$ and $u=hh\inv,$ then $\Psi_{h\inv} \circ \Psi_{h} = \Psi_{v}$ and $\Psi_{h}\circ \Psi_{h\inv} = \Psi_{u}$ by \ref{item:Psi:transitivity}, both of which are the identity map according to \ref{item:Psi:unit-identity}.
     \item[Re \ref{item:Psi:biequivariant}]
    Since the actions on $\cM$ and on $\cN$ commute, we have
    \begin{align*}
       b\cdot \Psi_{h} ((m\cdot c)\otimes n) 
       &=
       b\cdot \left[ m \otimes (c\cdot n) \right]
       =
       [b\cdot m] \otimes (c\cdot n)
       =
       \Psi_{h} (([b\cdot m]\cdot c)\otimes n) 
       \\&=
       \Psi_{h} ([b\cdot (m\cdot c)]\otimes n) =
       \Psi_{h} (b\cdot [( m\cdot c)\otimes n]) ,
    \end{align*}
    and similarly for $d\in D$. \qedhere
    \end{description}
\end{proof}

\section{Transitivity of Fell bundle equivalence}\label{sec:Transitivity}
In this section, we prove that Fell bundle equivalence is transitive and hence an equivalence relation.
Throughout this section, we make slightly stronger assumptions than in Section~\ref{sec:Product}:
\begin{itemize}
    \item We still fix three saturated Fell bundles $\cB=(p_{\cB}\colon B \to \cG),\cC=(p_{\cC}\colon C \to \cH), \cD=(p_{\cD}\colon D \to \mathcal{K})$ over \LCH\ \etale\ groupoids
    		$\cG,\cH,\mathcal{K}$.
 \item This time, let $\X$ be a $(\cG,\cH)$- and  $\Y$ be an $(\cH,\mathcal{K})$-{\em equivalence}; let $ Z := X\bfp{s}{r}Y$, which is just a $(\cG,\mathcal{K})$-\pe.
\item Similarly, we let $\cM=(q_{\cM}\colon M \to  \X)$ be a $(\cB,\cC)$- and $\cN=(q_{\cN}\colon N \to  \Y)$ be a $(\cC,\cD)$-{\em equivalence}.
    \item As before, we write $\cdot$ for the left and right actions on $\cM$, $\X$, $\cN$, and $\Y$. 
\end{itemize}

Using the techniques developed in the last section, we will build a \USCBb\ $\cP$ over the known $(\cG,\mathcal{K})$-groupoid equivalence
\[
    \Z:=\X*_{\cH}\Y=\{[x,y]_{\ipscriptstyle\cH}: s_{\X}(x)=r_{\Y}(y)\}
\]
where $[\cdot]_{\ipscriptstyle\cH}\colon Z\to \Z$ is the quotient map, defined by setting $[xh, h\inv y]_{\ipscriptstyle\cH} = [x,y]_{\ipscriptstyle\cH}$ for all $h\in\cH^{s(x)}$.
We will show that $\cP$ is the alleged Fell bundle equivalence between $\cB$ and~$\cD$.

\subsection{The quotient bundle}

\begin{lemma}\label{lem:sim-on-MastN}
    Let $\xi, \xi'\in \cK= \cM\otimes_{\cC\z}\cN$ with $ q_{\cK} (\xi)=(x,y)$ and $ q_{\cK} (\xi')=(x',y')$. We define the relation $\FBRel$ by 
    \[
       \xi \FBRel  \xi'
       :\iff \exists h\in\cH^{s(x')}_{s(x)} \text{ such that }
       x=x'h,
       \quad 
       y=h\inv y',
       \quad
       \Psi_{h} (\xi) = \xi'.
    \]
    This defines a closed equivalence relation on $\cK$.
\end{lemma}

\begin{proof}
    For reflexivity, fix any $\xi$ with $q_{\cK}(\xi)=(x,y)$. Take $h= s_{\X}(x)\in \cH\z$. As units act trivially on both $\X$ and $\Y$, we automatically have $x=xh$ and $y=h\inv y$. Furthermore, by Theorem~\ref{thm:Psi-cts}, \ref{item:Psi:unit-identity}, we have $\Psi_{h} (\xi) = \xi$ and thus $\xi \FBRel \xi$.
    
    \smallskip

    For symmetry, suppose we have some $h$ such that $x=x'h$, $y=h\inv y'$, and $\Psi_{h} (\xi) = \xi'.$ Then in particular, $xh\inv =x'$ and $h\inv y=y'$. By  Theorem~\ref{thm:Psi-cts}, \ref{item:Psi:inverse}, we further have that $\Psi_{h} (\xi) = \xi'$ implies $\xi = \Psi_{h\inv}(\xi')$. This shows that $\xi\FBRel \xi'$ implies $\xi' \FBRel  \xi .$
    
    \smallskip

    For transitivity, suppose that $\xi  \FBRel  \xi'$ and $\xi' \FBRel  \xi''$; let $h$ implement the first and $h'$ implement the second. Since $s_{\X}(x)=s(h)=r_{\Y}(y)$, $s_{\X}(x')=r(h)=r_{\Y}(y')$, and similarly $s_{\X}(x')=s(h')=r_{\Y}(y')$, $s_{\X}(x'')=r(h')=r_{\Y}(y''),$ we see that $r(h)=s(h')$, so $(h',h) \in \cH^{(2)}$; let $k=h'h$. Then
    \begin{align*}
       x&=x'h =(x''h') h =x''k, \text{ and}
       \\
       y&=h\inv y'  = h\inv ((h')\inv y'') = k\inv y''.
    \end{align*}
    Since $\Psi_{h'}\circ \Psi_{h}=\Psi_k$ by   Theorem~\ref{thm:Psi-cts}, \ref{item:Psi:transitivity}, we lastly conclude that 
    \[
       \Psi_k (\xi) = \Psi_{h'}\circ \Psi_{h} (\xi) = \Psi_{h'}(\xi') = \xi'',
    \]
    as required.

    \smallskip
    
    To see that $\FBRel$ is closed, suppose $\xi_{\lambda},\xi_{\lambda}'\in \cM\otimes_{\cC\z}\cN$ are such that $\xi_{\lambda} \FBRel \xi_{\lambda}'$ and $(\xi_{\lambda},\xi_{\lambda}')\to (\xi,\xi');$ let $q_\cK (\xi_{\lambda})=(x_{\lambda},y_{\lambda}),$ $q_\cK(\xi)=(x,y)$, and analogously for the primed versions.  For each $\lambda,$ there exists $h_{\lambda}\in \cH^{s(x_{\lambda}')}_{s(x_{\lambda})}$ such that $x_{\lambda} = x_{\lambda}' \cdot h_{\lambda}$ and $\Psi_{h_{\lambda}}(\xi_{\lambda})=\xi_{\lambda}'$. Being an equivalence,  $\X$ is a proper $\cH$-space, and so the fact that $(x_{\lambda}' \cdot h_{\lambda}, x_{\lambda}')=(x_{\lambda},x_{\lambda}')$ converges to $(x,x')$ implies that $h_{\lambda}$ converges to some $h\in \cH$ which satisfies $x=x'\cdot h$. By uniqueness, this $h$ also satisfies $y=h\inv \cdot y'.$ It follows from \ref{item:Psi:jointly-cts} that 
    \(
        \xi_{\lambda}' = \Psi_{h_{\lambda}}(\xi_{\lambda})\to \Psi_{h}(\xi).
    \)
    Since $\cM\otimes_{\cC\z}\cN$ is Hausdorff, limits are unique, so that $\xi_{\lambda}'\to\xi'$ implies that $\Psi_{h}(\xi)=\xi'$. In other words, $\xi \FBRel \xi'.$
    \qedhere
\end{proof}

\begin{definition}\label{def:cP}
We let $\cP$ be the quotient of $\cM\otimes_{\cC\z}\cN$ by the above equivalence relation, equipped with the quotient topology, and let $Q\colon \cM\otimes_{\cC\z}\cN \to (\cM\otimes_{\cC\z}\cN)/\FBRel$ denote the quotient map. The equivalence class $Q(\xi)$ of $\xi$ will be written as $[\xi]$.
\end{definition}

\begin{remark}\label{rmk:balancing-in-cP}
    By construction, if $\xi\in \cM\otimes_{\cC\z}\cN$ with  $ q_{\cK} (\xi)=(x,y)$, then for any $h\in\cH$ such that $s_{\X}(x)=r_{\cH}(h)=r_{\Y}(y)$, we have $[\xi]=[\Psi_{h} (\xi)]$ in $\cP$. In particular, for any $m \in \cM, c\in \cC, n \in \cN$ with $s_{\cM}(m)=r_{\cC}(c)$ and $s_{\cC}(c)=r_{\cN}(n)$, we have $[(m\cdot c)\otimes n]=[m\otimes (c\cdot n)]$ in $\cP$.
\end{remark}

\begin{lemma}\label{lem:q_cP open, cts, surjective}
    The map $q_{\cP}\colon \cP \to \Z:= \X\ast_{\cH} \Y$ given by 
    $q_{\cP}[\xi] := [ q_{\cK} (\xi)]_{\ipscriptstyle\cH}$ is well defined, surjective, continuous, and open. 
\end{lemma}

\begin{proof}
    If $\xi  \FBRel  \xi'$ with $ q_{\cK} (\xi)=(x,y)$ and $ q_{\cK} (\xi')=(x', y')$, then there exists $h\in\cH$ such that 
    \[
       x=x'h, 
       \quad\text{and}\quad
       y=h\inv y',
    \]
    which exactly means that
    \[
       [ q_{\cK} (\xi)]_{\ipscriptstyle\cH}
       =
       [x, y]_{\ipscriptstyle\cH}
       =
       [x'h, h\inv y']_{\ipscriptstyle\cH}
       =
       [x', y']_{\ipscriptstyle\cH}
       =
       [ q_{\cK} (\xi')]_{\ipscriptstyle\cH}
    \]
    in $\Z$, which shows that $q_{\cP}$ is well defined. Surjectivity is clear. For continuity, assume $U$ is an open set in $\Z$. In order to show that $q_{\cP}\inv (U)$ is open in $\cP$, we have to show that $Q\inv(q_{\cP}\inv (U))$ is open in $\cM\otimes_{\cC\z}\cN$. 
    Since $q_{\cP}(Q(\xi))= [ q_{\cK} (\xi)]_{\ipscriptstyle\cH} $, we have
   \(
       Q\inv(q_{\cP}\inv (U))
       =
        q_{\cK} \inv\bigl( [\cdot ]_{\ipscriptstyle\cH}\inv (U) \bigr).
    \)   
    As $[\cdot ]_{\ipscriptstyle\cH}$ is continuous by definition of the topology on $\Z$, and as $ q_{\cK} \colon \cM\otimes_{\cC\z}\cN \to \X \bfp{s}{r} \Y$ is continuous by construction of the topology on $\cM\otimes_{\cC\z}\cN$, we see that $Q\inv(q_{\cP}\inv (U))$ is indeed open.
    
 \smallskip
    
    Lastly, to see that $q_{\cP}$ is an open map, take any $V\subseteq \cP$ open; we have to show that $q_{\cP}(V)$ is open in $\Z$, i.e., that $[\cdot]_{\ipscriptstyle\cH}\inv (q_{\cP}(V))$ is open in $\X \bfp{s}{r} \Y$. We claim that
    \begin{equation}\label{eq:inverse-images-equal}
       [\cdot]_{\ipscriptstyle\cH}\inv (q_{\cP}(V))
       \overset{!}{=}
        q_{\cK}  (Q\inv (V));
    \end{equation}
    this is sufficient for our claim because $Q$ is continuous and $ q_{\cK} $ is open (see Lemma~\ref{lem:qMN-open}).

    For $\supseteq$, take $(x',y')\in  q_{\cK}  (Q\inv (V))$, i.e., there exists some $\xi'\in Q\inv (V)$ such that $(x',y')=  q_{\cK} (\xi')$. As $p':=Q(\xi')\in V$, it follows from 
    \[
    [x',y']_{\ipscriptstyle\cH}
    = [ q_{\cK} (\xi')]_{\ipscriptstyle\cH}
    = q_{\cP}(p')
    \in
    q_{\cP}(V)
    \]
    that   $(x',y') \in [\cdot]_{\ipscriptstyle\cH}\inv (q_{\cP}(V))$.
    
    For $\subseteq$, take $(x,y)\in [\cdot]_{\ipscriptstyle\cH}\inv (q_{\cP}(V))$, i.e., there exists $p\in V$ such that $q_{\cP}(p)=[x,y]_{\ipscriptstyle\cH}$. By definition of $q_{\cP},$ any $\xi\in Q\inv (\{p\}) \subseteq Q\inv (V)$ then has the property $[ q_{\cK} (\xi)]_{\ipscriptstyle\cH}=[x,y]_{\ipscriptstyle\cH}$. If $(x',y'):=  q_{\cK} (\xi)$, then that implies that there exists $h\in\cH$ with $(xh,h\inv y)=(x', y'),$ so $\xi\in \cM_{xh}\otimes_{\cC\z} \cN_{h\inv y}$ and we may let $\eta:= \Psi_{h}(\xi)\in \cM_{x}\otimes_{\cC\z} \cN_{y}$. By construction we have $\xi \FBRel  \eta$, so that $Q(\eta)=Q(\xi)\in V$, i.e., $\eta\in Q\inv (V)$ also. We have shown that
    \[
       (x,y)
       =
        q_{\cK}  (\eta)
       \in
        q_{\cK}  (Q\inv (V)),
    \]
    proving Assertion~\eqref{eq:inverse-images-equal} and hence the claim that $q_{\cP}$ is open.
\end{proof}

\begin{proposition}\label{prop:linear-fibres}
    Any fibre $\cP_{z}:=q_{\cP}\inv(z)$ of $\cP$ is naturally a vector space. In fact, if $[\xi_{i}]\in \cP_{z}$ for $i=1,2$, then there exists a unique $h\in\cH$ such that $ q_{\cK} (\xi_{1}) =  q_{\cK} (\Psi_{h}(\xi_{2}))$ and we can define
    \[ 
       [\xi_{1}]+   [\xi_{2}]
       :=
       [\xi_{1} +   \Psi_{h}(\xi_{2})]
       \quad\text{and}\quad
       \lambda[\xi_{1}] := [\lambda\xi_{1}]
    \]
    where  $\lambda \in\mathbb{C}$.
\end{proposition}

\begin{proof}
    Let  $\xi_{1},\xi_{2}$ be any representatives in $\cM\otimes_{\cC\z} \cN$ of the given elements of $\cP$. Let us first note that $\lambda[\xi_{1}] = [\lambda\xi_{1}]$ is well defined as $\Psi_{h}$ is $\mathbb{C}$-linear.

    As $[\xi_{i}]\in\cP_{z}$, we have $[x_{1}, y_{1}]_{\ipscriptstyle\cH} = z=
       [x_{2}, y_{2}]_{\ipscriptstyle\cH}$, where  $ q_{\cK} (\xi_{i})=(x_{i}, y_{i})$. Thus, there exists $h\in \cH^{s(x_{1})}_{s(x_{2})}$ such that 
       $
       x_{2}=x_{1}h, 
       $ and $y_{2}=h\inv y_{1}$.
    By Remark~\ref{rmk:balancing-in-cP}, we have
    $[\xi_{2}]=[\Psi_{h}(\xi_{2})].$ As $\xi_{1}$ and $\Psi_{h}(\xi_{2})$ are both elements of $\cM_{x_{1}}\otimes_{\cC\z}\cN_{y_{1}}
    =
    \cM_{x_{2}h\inv }\otimes_{\cC\z}\cN_{h y_{2}}
    $, it makes sense to consider the element $\xi_{1} + \Psi_{h}(\xi_{2})$. We now want to define 
    \[ 
       [\xi_{1}]+   [\xi_{2}]
       =
       [\xi_{1}]+   [\Psi_{h}(\xi_{2})]
       :=
       [\xi_{1} +    \Psi_{h}(\xi_{2})],
    \]
    and to this end, we need to check that it does not depend on the choice of representatives. So suppose that $\xi_{3},\xi_{4}$ are such that $[\xi_{1}]=[\xi_{3}]$ and $[\xi_{2}]=[\xi_{4}]$. With $x_{i}\in \X$ and $y_{i}\in \Y$ the elements such that $ q_{\cK} (\xi_{i})=(x_{i}, y_{i})$ for $i=3,4$, we then know
    \begin{equation}\label{eq:xi1xi3}
       \exists h_{o}\in\cH^{s(x_{1})}_{s(x_{3})} \text{ such that }
       \Psi_{h_{o}} (\xi_{3}) = \xi_{1}
    \end{equation}
    and
    \[
       \exists h_{e}\in\cH^{s(x_{2})}_{s(x_{4})} \text{ such that }
       \Psi_{h_{e}} (\xi_{4}) = \xi_{2}.
    \]
    We point out that, since $h\in \cH^{s(x_{1})}_{s(x_{2})}$, we can consider
    \[
       h':=
       h_{o}\inv h h_{e}
       \in \cH.
    \]
    We claim that $h'$ is the unique element such that $\xi_{3}$ and $\Psi_{h'}(\xi_{4})$ are in the same fibre. Indeed, by    Theorem~\ref{thm:Psi-cts}, \ref{item:Psi:transitivity}, we have
    \begin{align}\label{eq:xi4xi2}
       \Psi_{h'} (\xi_{4})
       =
       \Psi_{h_{o}\inv h h_{e}} (\xi_{4})
       =
       \Psi_{h_{o}\inv}
       \left(
        \Psi_{h} (\xi_{2})
       \right).
    \end{align}
    Now $\Psi_{h} (\xi_{2})$ is in the same fibre as $\xi_{1}=\Psi_{h_{o}} (\xi_{3})$, so that by  Theorem~\ref{thm:Psi-cts}, \ref{item:Psi:inverse}, we indeed have that $\Psi_{h'} (\xi_{4})$ is in the same fibre as $\xi_{3}.$ Thus, our definition requires
    \[
       [\xi_{3}]+   [\xi_{4}] 
       =
       [\xi_{3} +    \Psi_{h'} (\xi_{4})], 
    \]
    and we therefore only need to show that
    \[
       [\xi_{1} +    \Psi_{h}(\xi_{2})]
       \overset{!}{=}
       [\xi_{3} +    \Psi_{h'} (\xi_{4})].
    \]
    We have
    \begin{alignat*}{2}
       [\xi_{1} +    \Psi_{h}(\xi_{2})]
       &=
       [
       \Psi_{h_{o}} (\xi_{3})
       +
          \Psi_{h_{o}} (\Psi_{h'} (\xi_{4}))
       ]
      \qquad&&
       \text{by Equations~\eqref{eq:xi1xi3} and~\eqref{eq:xi4xi2}}
       \\
       &=
       [
       \Psi_{h_{o}} (\xi_{3}+   \Psi_{h'} (\xi_{4}))
       ]
       &&
       \text{by linearity}\\
       &=
       [
           \xi_{3}+   \Psi_{h'} (\xi_{4})
       ]
       &&
       \text{by Remark~\ref{rmk:balancing-in-cP}.}
    \end{alignat*}
    To see that this defines the structure of a vector space, we need to check some axioms.
    First, addition is associative: for appropriate $h, h'$, we have
    \begin{align*}
       ([\xi_{1}]+   [\xi_{2}]) + [\xi_{3}]
       &
       =  [\xi_{1} +   \Psi_{h}(\xi_{2})] + [\xi_{3}]
       = [(\xi_{1} +   \Psi_{h}(\xi_{2})) + \Psi_{h'}(\xi_{3})]
       \\&
       = [\xi_{1} +   \Psi_{h}(\xi_{2} + \Psi_{h\inv h'}(\xi_{3}))]
       \\&
       = [\xi_{1}] +   [\xi_{2} + \Psi_{h\inv h'}(\xi_{3})]
       \overset{(*)}{=} [\xi_{1}] +   \left( [\xi_{2}] + [\xi_{3}]\right).
    \end{align*}
    As both $h$ and $h'$ are unique, we do not need do dwell on the fact that $h\inv h'$ is the unique element that makes $(*)$ true. \smallskip
    
    Second, addition is also commutative: using that $[\xi]=[\Psi_{h}(\xi)]$ for appropriate $h$, we get with linearity of $\Psi_{h}$,
    \begin{align*}
       [\xi_{1}]+   [\xi_{2}]
       &
       =[\xi_{1} +   \Psi_{h}(\xi_{2})]
       =[\Psi_{h}(\xi_{2}) + \xi_{1}]
       \\&
       =[\xi_{2} + \Psi_{h\inv}(\xi_{1})]
       =[\xi_{2}]+   [\xi_{1}] \end{align*}
    The other axioms are even easier:
    \begin{itemize}
        \item the identity element of addition is $[0]=[\Psi_h(0)]$;
       \item the additive inverse of $[\xi]$ is $[- \xi]$;
       \item for $\lambda,\nu\in\mathbb{C}$, we have $(\lambda\nu) [\xi] = \lambda(\nu[ \xi])$, $\lambda([\xi_{1}]+   [\xi_{2}])=  [\lambda\xi_{1}]+[\lambda\xi_{2}]$, and $(\lambda+\nu)[\xi]= \lambda[\xi]+\nu[\xi]$; and lastly
       \item $1 [\xi] = [ \xi]$.
    \end{itemize} 
    All of these follow from the corresponding properties of the vector space $\cM_{x}\otimes_{\cC\z}\cN_{y}.$
\end{proof}

Fix $z\in \Z$ and let $(x,y)\in \X  \bfp{s}{r}  \Y$ with $v:=s_{\X}(x)=r_{\Y}(y)$ be any representative, i.e., $z=[x,y]_{\ipscriptstyle\cH}$. If we restrict the quotient map $Q\colon
       \cM \otimes_{\cC\z} \cN
       \to 
       \cP$ to the fibre $\cM_x \otimes_{\cC\z} \cN_y=
       \cM_x \otimes_{C(v)} \cN_y$, it is $\cP_z$-valued and  linear by construction of the linear structure on~$\cP_z$. In fact, more is true:

\begin{lemma}\label{lem:Pz=MxNy}
    The restricted quotient map $Q|\colon
       \cM_{x} \otimes_{\cC\z} \cN_{y}
       \to 
       \cP_{z}$ is a linear homeomorphism onto~$\cP_z$. In particular, the Banach space norm that $\cP_z$ inherits from $\cM_{x} \otimes_{\cC\z} \cN_{y}$ induces the subspace topology that $\cP_z$ inherits from~$\cP.$ 
\end{lemma}

\begin{proof}
    First note that the map is indeed $\cP_{z}$-valued: if we take $m\in \cM_{x}$ and $n\in \cN_{y}$, we have
       \[
           q_{\cP} ([m\otimes n]) 
           =
           [ q_{\cK} (m\otimes n)]_{\ipscriptstyle\cH}
           =
           [q_{\cM}(m), q_{\cN}(n)]_{\ipscriptstyle\cH}
           =
           [x,y]_{\ipscriptstyle\cH}
           =
           z.
       \]
    As $Q|$ maps elementary tensors to $\cP_z$, it does so for all other elements of $\cM_{x} \otimes_{\cC\z} \cN_{y}$ also, since $\cP_{z}$ is closed in $\cP$.
    Linearity is now obvious in light of the definition of $\cP$ and its fibre $\cP_z$.   For continuity, just note that $Q$ is continuous and $\cM_x\otimes_{\cC\z}\cN_y$ is closed in~$\cM\otimes_{\cC\z}\cN$.

    Note that $[\xi]=0$ in $\cP$ means $\xi=\Psi_{h}(0)=0$ for some $h\in\cH$, i.e., only the zero vector in any of the fibres of $\cM\otimes_{\cC\z}\cN$ gives rise to the zero vector in the corresponding fibre of $\cP$. In particular, $Q|$ can only send $0$ to $0,$ i.e., the linear map $Q|$ is injective. 
    
    To see that $Q|$ is surjective, take an arbitrary $p\in\cP_{z}$ and let $\xi'\in\cM\otimes_{\cC\z}\cN$ be any representative, i.e., $p=[\xi']$. As
    \[
       [x,y]_{\ipscriptstyle\cH}=z=q_{\cP}(p)=[ q_{\cK} (\xi')]_{\ipscriptstyle\cH},
    \]
    we know that, if $ q_{\cK} (\xi')=(x',y')$, there exists $h\in\cH$ such that
    $ xh=x'$ and $h\inv y=y'$. Let $\xi:= \Psi_{h}(\xi')$, an element of $\cM_{x} \otimes_{\cC\z} \cN_{y}$. Then
    \[
       Q|(\xi)=[\xi] = [\Psi_{h}(\xi')] = [\xi'] = p.
    \]
    
To see that $Q|$ is a homeomorphism, it now suffices to check that it is closed, so suppose that $A\subseteq \cM_{x}\otimes_{\cC\z}\cN_{y}$ is some closed set. We have to show that $Q|(A)$ is closed in $\cP_{z}$. Since $\cP_{z}$ is closed in $\cP$, this is equivalent to showing that $Q|(A)$ is closed in $\cP$, which, by the quotient topology, means that $Q\inv (Q|(A))$ is closed in~$\cM\otimes_{\cC\z}\cN$. So assume that $\xi_{\lambda}$ is a net in $Q\inv (Q|(A))$ which converges to some $\xi\in \cM\otimes_{\cC\z}\cN$; we have to show that $\xi\in Q\inv (Q|(A))$.

As $\xi_{\lambda}\in Q\inv (Q|(A))$, there exist $a_{\lambda}\in A$ such that $\xi_{\lambda} \FBRel  a_{\lambda}$. Since $A\subseteq \cM_{x}\otimes_{\cC\z}\cN_{y}$, this implies that there exist $h_{\lambda}\in \cH^{s(x)}$ such that \( q_{\cK} (\xi_{\lambda})=(xh_{\lambda} , h_{\lambda}\inv y)\) and \(a_{\lambda} = \Psi_{h_{\lambda}}(\xi_{\lambda})\). Since $\xi_{\lambda}\to \xi$, we  have
\[
    (xh_{\lambda} , h_{\lambda}\inv y) =  q_{\cK} (\xi_{\lambda})\to  q_{\cK} (\xi) =: (x_{0}, y_{0}).
\]
In particular, $r_{\X}(x_{0})=r_{\X}(x)$ and $s_{\Y}(y_{0})=s_{\Y}(y).$ Since $\X$ and $\Y$ are equivalences of groupoids, it follows from the assumption $\X / \cH \cong \cG\z$ and $\cH \backslash \Y \cong \cK\z$ that there exist $h\in\cH^{s(x)}, h'\in \cH_{r(y)}$ such that $x_{0}=xh$ and $y_{0}=h' y$. Since
\[
    xh_{\lambda} 
    \to 
    xh
    \quad\text{and}\quad
    h_{\lambda}\inv y
    \to 
    h' y,
\]
it follows that $h_{\lambda} \to h$ and $h_{\lambda}\inv\to h'$, so that $h'=h\inv$ because $\cH$ is Hausdorff.
    We may let $a:= \Psi_{h}(\xi)$, and it follows from Theorem~\ref{thm:Psi-cts}, \ref{item:Psi:jointly-cts}, that $a_{\lambda} \to a$ since $h_{\lambda}\to h$ and $\xi_{\lambda}\to\xi.$
Since $A$ is closed, this implies that $a\in A$.
As $\xi \FBRel  a$ by construction of $a$, it follows that $\xi\in Q\inv (Q|(A))$ as claimed. 
\end{proof}

\begin{proposition}
\label{prop:Q-open}
    The quotient map $Q\colon \cM\otimes_{\cC\z}\cN \to \cP$ is open. Consequently, $\cP$ is \LCH.
\end{proposition}

To prove Proposition~\ref{prop:Q-open}, we need two lemmas; see Definition~\ref{def:t} for the map $t$ and Definition~\ref{def:beefing-up-a-USC-bundle} for that of pull-back bundles.
\begin{lemma}\label{lem:Psi-induces-sections}
Any continuous section $\kappa$ of $\cM\otimes_{\cC\z}\cN$ induces a continuous section $\Psi(\kappa)$ of $(\cM\otimes_{\cC\z} \cN) \bfp{t}{r}  \cH$ defined by
        \begin{align*}
            \Psi(\kappa)\colon \quad
            (\X  \bfp{s}{r}  \Y) \bfp{t}{r} \cH
            &\to
            (\cM\otimes_{\cC\z} \cN) \bfp{t}{r}  \cH
            \\
            (x,y,h)
            &\mapsto
            \Psi(h,\kappa(xh,h\inv y))    
        \end{align*}
\end{lemma}
\begin{proof}
     As $(x,y,h)\in (\X  \bfp{s}{r}  \Y) \bfp{t}{r} \cH$, we have $r_{\cH}(h)= t_{Z}(x,y),$  so $xh$ and $h\inv y$ make sense. Moreover, since $\kappa$ is a section of $\cK=\cM\otimes_{\cC\z}\cN$, we have
     $$t(q_{\cK}(\kappa(xh ,h\inv y))) =  t(xh ,h\inv y) = s(h) ,$$ so $(h,\kappa(xh,h\inv y))$ is in the domain $\cH \bfp{s}{t}  (\cM\otimes_{\cC\z} \cN)$ of $\Psi$. In other words, the definition of $\Psi(\kappa)$
     makes sense. To see that $\Psi(\kappa)$ is a section, we note that $\kappa(xh,h\inv y)\in \cM_{xh}\otimes_{\cC\z}\cN_{h\inv y}$, which is being mapped to an element of $\cM_{x}\otimes_{\cC\z}\cN_{ y}$ by $\Psi_h$, so that
    \begin{align*}
        \Psi(\kappa)\, (x,y,h)
        &=
        \Psi(h,\kappa(xh,h\inv y))
        =
        (\Psi_h(\kappa(xh,h\inv y)), h)
        \\
        &\in 
        \cM_{x}\otimes_{\cC\z}\cN_{ y} \times \{h\}
        =
        \left[
            (\cM\otimes_{\cC\z} \cN) \bfp{t}{r}  \cH
        \right]_{(x,y,h)}
        .
    \end{align*}
    To see that $\Psi(\kappa)$ is continuous, just notice that it is a concatenation of continuous functions: inversion in $\cH$, the $\cH$-action on $\X$ resp.\ $\Y$, and the continuous map $\Psi$.
\end{proof}

\begin{lemma}
\label{lem:proj-map-open}
    The projection map
    \(
        \mathrm{pr}_{\cK}\colon
        \cK \bfp{t}{r}  \cH
        \to
        \cK,
        (\xi,h)\mapsto \xi,
    \)
    is open.
\end{lemma}

\begin{proof}
    As explained in Lemma~\ref{lem:beefing-up-a-USC-bundle}, subsets of the form $U_{1} \bfp{t}{r}  U_{2}$ for $U_{1}\subseteq \cK=\cM\otimes_{\cC\z} \cN$ and $U_{2}\subseteq \cH$ basic open sets form a basis of the topology of $\cK \bfp{t}{r}  \cH$, so it suffices to check that $\mathrm{pr}_{\cK}(U_{1}\bfp{t}{r} U_{2})$ is open. We compute:
    \begin{align*}
        \mathrm{pr}_{\cK}(U_{1}\bfp{t}{r} U_{2})
        &
        =
        \{
            \xi\in U_{1}\,\mid\,t(q_{\cK}(\xi))\in r_{\cH}(U_{2})
        \}
        \\&
        =
        U_{1}\cap q_{\cK}\inv \bigl( t\inv (r_{\cH}(U_{2})\bigr).
    \end{align*}
    Note that, since $\cH$ is \etale, $r_{\cH}(U_{2})$ is open in $\cH\z$. Since $t$ and $q_{\cK}$ are continuous, we therefore have shown that $\mathrm{pr}_{\cK}(U_{1}\bfp{t}{r} U_{2})$ is indeed open. 
\end{proof}

\begin{proof}[Proof of Proposition~\ref{prop:Q-open}]
    It suffices to show that $Q$ maps basic open sets to open sets. Recall \cite[Proof of Theorem C.25]{Wil2007}
    that a basic open set of $\cK=\cM\otimes_{\cC\z}\cN$ is of the form
    \[
        W:= 
        W(
    \kappa
    , V, \epsilon)
        :=
        \left\{
            \xi\in \cK
            \,\mid\,
            q_{\cK}(\xi)\in V, \norm{\xi - \kappa(q_{\cK}(\xi))}< \epsilon
        \right\}
    \]
    for  $\kappa$ a continuous section of $\cK$, $V\subseteq Z=\X  \bfp{s}{r}  \Y$ open, and $\epsilon >0.$ We have to show that $Q(W)$ is open in $\cP$, meaning that $Q\inv (Q(W))$ is open in~$\cK$. We claim that
    \begin{equation}\label{eq:QinvQW}
        Q\inv (Q(W))
        \overset{!}{=}
        \mathrm{pr}_{\cK}\bigl( W(\Psi(
    \kappa
    ), \psi(\cH \bfp{s}{t} V), \epsilon)\bigr),
    \end{equation}
    where $\mathrm{pr}_{\cK}$ is the open map from Lemma~\ref{lem:proj-map-open} and $\psi$ the homeomorphism from Theorem~\ref{thm:Psi-cts}.
    
    \medskip
    
    For ``$\subseteq$'', let $\eta$ be such that there exists $\xi\in W$ with $\eta \FBRel \xi.$ 
    In other words, there exists $h\in\cH$ with $s_{\cH}(h)=t_{\cK}(\xi)$
    such that 
    $\eta=\Psi_{h}(\xi).$ We claim that $(\eta,h)\in W(\Psi(
    \kappa
    ), \psi(\cH \bfp{s}{t} V), \epsilon)$. If $\pi$ denotes the projection map $(\cM\otimes_{\cC\z}\cN)\bfp{t}{r}\cH\to (\X\bfp{s}{r} \Y)\bfp{t}{r}\cH$ of the pull-back bundle,
    then we have
    \begin{align*}
        &\pi\left( \eta,h\right)
        \in \psi(\cH \bfp{s}{t} V)
        \\\iff& 
        \exists k\in \cH, (x,y)\in V\text{ such that } s_{\cH}(k)=t(x,y) \text{ and } (q_{\cK}(\eta),h) = (xk\inv, ky,k)
        \\\iff& 
        \exists (x,y)\in V\text{ such that } s_{\cH}(h)=t(x,y) \text{ and } q_{\cK}(\eta) = (xh\inv, hy).
    \end{align*}
    Such a tuple indeed exists: by choice of the element $\xi$ in~$W$, we have that $(x,y):=q_{\cK}(\xi)$ is in~$V$, and the equality $\eta=\Psi_{h}(\xi)$ implies $q_{\cK}(\eta) = (xh\inv, hy)$.
    Furthermore, 
    \begin{align*}
        \norm{
            (\eta,h) - \Psi(
    \kappa
    )(\pi\left( \eta,h\right))
        }
        &=
         \norm{
            (\eta,h) - \Psi(
    \kappa
    )\left( xh\inv, hy,h\right)
        }
        \\&=
        \norm{
            (\eta,h) -
            \Psi(h,
    \kappa
    (x,y))
        }
        \\
        &=
        \norm{
            \Psi_h (\xi) -
            \Psi_h (
    \kappa
    (x,y))
        }
        \\
        &=
        \norm{
            \xi -
    \kappa
    (q_{\cK}(\xi))
        }
        < \epsilon,
        \text{ since } \xi\in W.
    \end{align*}
    This proves that $(\eta,h)\in W(\Psi(
    \kappa
    ), \psi(\cH \bfp{s}{t} V), \epsilon)$, so that $\eta$ is in the right-hand side of the alleged equation.
    
    \smallskip
    
    For ``$\supset$'', assume that $\eta$ is in the right-hand side, so there exists $h\in \cH$ such that $t(q_{\cK}(\eta))=r_{\cH}(h)$, $\pi(\eta,h)\in \psi(\cH \bfp{s}{t} V)$, and
    \begin{align*}
        \norm{
            (\eta,h) - \Psi( \kappa )(\pi\left( \eta,h\right))
        }
        < \epsilon
        .
    \end{align*}
    As argued above, $\pi(\eta,h)\in \psi(\cH \bfp{s}{t} V)$ implies that there exists $(x,y)\in V$ such that $s_{\cH}(h)=t(x,y)$ and $q_{\cK}(\eta) = (xh\inv, hy)$. We claim that $\xi:=\Psi_{h\inv}(\eta)$ is an element of $W$, so that,
    since $\eta \FBRel \xi$, we can conclude $\eta\in Q\inv (Q(W)).$
    We have $q_{\cK}(\xi)=(x,y)\in V$ by construction, and the same computation as above yields
    \begin{align*}
        \norm{
            \xi -
            \kappa
            (q_{\cK}(\xi))
        }
        &=
        \norm{
            (\eta,h) - \Psi(\kappa)(\pi\left( \eta,h\right))
        },
    \end{align*}
    and the right-hand side is smaller than $\epsilon$ by choice of $\eta$.
    This proves that $\xi\in W$, as claimed.
    All in all, we have shown Equation~\eqref{eq:QinvQW}.
    
    \medskip
    
    By Lemma~\ref{lem:Psi-induces-sections}, $\Psi(\kappa)$ is a continuous section of $\cK\bfp{t}{r}\cH$. Since $\psi$ is a homeomorphism and $V$ is open in~$Z$, $\psi(\cH\bfp{s}{t}V)$ is open in~$Z\bfp{t}{r}\cH$. Thus, $W(\Psi(
    \kappa
    ), \psi(\cH \bfp{s}{t} V), \epsilon)$ is a basic open set in~$\cK\bfp{t}{s}\cH$. Since $\mathrm{pr}_{\cK}$ is open by Lemma~\ref{lem:proj-map-open}, we conclude that $Q\inv(Q(W))$ is open, as claimed.

\end{proof}

As $\cP$ is a bundle over a groupoid \pe, we will follow our convention in Notation~\ref{notation:s-and-r}: $r_{\cP}\colon \cP\to \cG\z$ and $s_{\cP}\colon \cP\to \mathcal{K}\z$ are defined by
\begin{align*}
    r_{\cP}([\xi]):=&\ (r_{\Z}\circ q_{\cP})([\xi])) =r_{\Z}([q_{\cK}(\xi)]_{\ipscriptstyle\cH}) = r_{Z}(q_{\cK}(\xi))
\end{align*}
respectively
    $s_{\cP}([\xi])
    =
    s_{Z}(q_{\cK}(\xi))
    .$
Note that $r_{\cP}$ and $s_{\cP}$ are continuous with respect to the quotient topology on $\cP$: 
for any $U\subseteq \cG\z$, we have that
\[
    r_{\cP}\inv (U)
    =
    Q\left(
        (r_{Z} \circ q_{\cK}
        )\inv
    (U)
    \right).
\]

Thus, if $U$ is open, then $r_{\cP}\inv (U)$ is open since $r_{Z}$ and $q_{\cK}$ are continuous and since $Q$ is open by Proposition~\ref{prop:Q-open}. 


\begin{remark}\label{rmk:id-ast-Q-open} 
    It now follows from  Proposition~\ref{prop:Q-open} and Lemma~\ref{lem:bfp of open is open} that  the maps
    \begin{align*}
        \mathrm{id}\ast Q\colon& \quad
        \cB\bfp{s}{r}(\cM\otimes_{\cC\z}\cN) \to 
        \cB\bfp{s}{r}\cP,
        \quad
        (b,\xi)\mapsto (b,[\xi]), \quad\text{and}
    \\
         Q\ast \mathrm{id}\colon& \quad
        (\cM\otimes_{\cC\z}\cN)\bfp{s}{r}\cD \to 
        \cP\bfp{s}{r}\cD,
        \quad
        (\xi,d)\mapsto ([\xi],d),
    \end{align*}
    are open.
\end{remark}

\subsection{The quotient bundle is USC}

\begin{proposition}\label{prop:cd}
    Let $\cB=(q_{\cB}\colon B\to X)$ and $\cC=(q_{\cC}\colon C\to Y)$ be bundles of Banach spaces over \LCH\ spaces. Assume we have maps
    \[
    \begin{tikzcd}[ampersand replacement=\&]
        B
        \ar[r, "\Phi"]\ar[d, "q_{\cB}"']
        \&
        C
        \ar[d, "q_{\cC}"]
        \\
        X
        \ar[r, "\varphi"]
        \&
        Y
    \end{tikzcd}
    \]
    and topologies on $B$ and $C$ with the following properties:
    \begin{enumerate}[label=\textup{(\arabic*)}]
        \item\label{item:general quot:quotient maps} $\Phi$ and $\varphi$ are open quotient maps; in particular, they are surjective and continuous.
        \item\label{item:general quot:Phi lin isometry}  $\Phi$ is a linear isometry when restricted to the map $B_{x}\to C_{\varphi(x)}$.
        \item\label{item:general quot:preimages} For every $c\in C$ and $x\in X$ with $q_{\cC}(c)=\varphi(x)$, there exists $b\in B(x)$ such that $\Phi(b)=c$. 
    \end{enumerate}
    If $\cB$ is an \USCBb, then 
    $\cC$ is \usc\ as well.
\end{proposition}


\begin{proof}
    If $U\subseteq Y$ is open, then $q_{\cC}\inv (U)$ is open in $C$ exactly when $\Phi\inv (q_{\cC}\inv (U))$ is open in $B$, but
    \[
        \Phi\inv (q_{\cC}\inv (U)) = q_{\cB}\inv (\varphi\inv(U))
    \]
    by \assumption{commutativity of the diagram}. As \assumption{$q_{\cB}$ and $\varphi$ are continuous}, $q_{\cC}$ is continuous. 
    Likewise,  \assumption{surjectivity of $q_{\cB}$ and $\varphi$} implies surjectivity of $q_{\cC}$: if $y\in Y$, pick any $x\in \varphi\inv(y)\subseteq X$ and then $b\in q_{\cB}\inv (x)\subseteq B$, so that $\Phi(b)\in q_{\cC}(y)$.
    
    \smallskip
    
    We now need to check Conditions~\ref{cond.B-USC}--
    \ref{cond.B-nets}, all of which follow from  $\cB$ satisfying these conditions and from $\Phi$ being open. To be more precise: \smallskip
    
    For \ref{cond.B-USC}, 
    we need to show that the map $C\to \mathbb{R}_{\geq 0}$, $c\mapsto\norm{c}$, is \usc, or in other words, that
    \(
       \{
           c\in C  \,\mid\,\norm{c} < r
       \}
    \)
    is open for any fixed $r\in \mathbb{R}_{\geq 0}.$
    \assumption{By Assumption~\ref{item:general quot:Phi lin isometry}, we have $\norm{ b }_{\cB}=\norm{\Phi( b )}_{\cC}$,} so that
    $$\{c\in C \,\mid\,\norm{c} < r\} =  \Phi(\{ b \in B \,\mid\,\norm{ b } < r\}).$$ Since \assumption{$\cB$ is a \USCBb\ }and \assumption{$\Phi$ is an open map}, the claim follows.

   \smallskip
   
   For \ref{cond.B-plus}, let $U\subseteq C$ be an open neighborhood. Since \assumption{$\Phi$ is continuous}, $\Phi\inv (U)$ is an open subset of $B$. Since \assumption{$\cB$ is a \USCBb}, $+_{B}$ is continuous on~$\cB$. Consequently, $V:=+_{B}\inv(\Phi\inv(U))$ is an open subset of $B\bfp{q}{q} B$. \assumption{As $\Phi$ is an open map}, so is the map $\Phi\bfp{q}{q} \Phi$ by Lemma~\ref{lem:bfp of open is open}, and hence $\Phi\bfp{q}{q} \Phi(V)$ is an open set in $C\bfp{q}{q} C$. We compute
    \begin{align*}
       \Phi\bfp{q}{q} \Phi(V)
       &=
       \left\{
           \big(\Phi( b ),\Phi( b')\big)
            \,\mid\,
           ( b , b')\in +_{B}\inv(\Phi\inv(U))
       \right\}
       \\
       &=
       \left\{
           \big(\Phi( b ),\Phi( b')\big)
            \,\mid\,
            b + b' \in \Phi\inv(U)
       \right\}
       =
       \left\{
           \big(\Phi( b ),\Phi( b')\big)
            \,\mid\,
           \Phi( b + b') \in U
       \right\}
       =
       +_{C}\inv (U),
    \end{align*}
    so $+_{C}\inv (U)$ is open in $C\bfp{q}{q} C$, meaning that $+_{C}$ is indeed continuous on $\cC$.
    
    \smallskip
    
    For \ref{cond.B-times}, take an arbitrary but fixed $\lambda\in\mathbb{C}$; we have to show that the map $f\colon C\to C,$ $c\mapsto \lambda c$, is continuous. Let $U\subseteq C$ be open, so $\Phi\inv (U)$ is open in $B$. Since \assumption{$\cB$ is a \USCBb},
    \[
       V:=
       \{ b \in B \,\mid\,
       \lambda b \in \Phi\inv (U)\}
    \]
    is open in $B$. Since \assumption{$\Phi$ is open}, $\Phi(V)$ is open in $C$. As
    \begin{align*}
       \Phi(V)
       &=
       \{\Phi( b )  \,\mid\,
       \lambda b \in \Phi\inv (U)\}
       =
       \{\Phi( b )  \,\mid\,
       \lambda \Phi( b )=\Phi(\lambda b )\in U\}
       =
       f\inv (U),
    \end{align*}
    using that \assumption{$\Phi$ is fibrewise linear by \ref{item:general quot:Phi lin isometry}},  it follows that $f$ is continuous.    
    
    \medskip
    
    For \ref{cond.B-nets}, suppose $\{c_{ i }\}_{ i \in I}$ is a net in $C$ such that $q_{\cC}(c_{ i })$ converges to some $ y\in Y$ and $\norm{c_{ i }}\to 0$. Assume for a contradiction that $\{c_i \}_{ i \in I}$ does not converge to $0_{ y}\in\cC_{ y}$ in $C$. In other words, there exists an open neighborhood $U$ of $0_{ y}$ in $C$ such that 
    \begin{equation}\label{eq:choice of U}
        \text{for all } i \in I\text{, there exists }j\geq  i \text{ such that }c_{j}\notin U.
    \end{equation} 
    
    The outline of the proof is as follows. We will first construct a subnet $\{c_{h(\beta)}\}_{\beta\in J }$ of $\{c_i \}_{ i \in I}$ which lies completely outside of~$U$. For a subnet $\{c_{f(\gamma)}\}_{\gamma\in \Gamma}$ of $\{c_{h(\beta)}\}_{\beta\in  J }$ and any fixed representative $ x\in \varphi\inv ( y)\subseteq X$, we will then find lifts $b_{\gamma}\in \Phi\inv (c_{f(\gamma)})\subseteq B$ which satisfy both $\lim_{\gamma} q_{\cB} (b_\gamma)=  x$ and $\lim_{\gamma} \norm{b_{\gamma}}= 0$. Since $\cB$ is \usc, this implies $b_{\gamma}\to 0_{ x }$ in $B$.
    Since $\Phi\inv (U)$ is an open neighborhood of $0_{ x }$ in $B$ by \assumption{continuity of $\Phi$}, this contradicts that $\Phi(b_{\gamma})= c_{f(\gamma)}\notin U$ for all $\gamma$.
    
    \smallskip
    
    Let
    \[
        J := \{ ( i , j) \in I\times I \,\mid\, i \leq j\text{ and } c_{j}\notin U
       \},
    \]
    which is non-empty by Assumption~\eqref{eq:choice of U} on our net.
    Define a preorder on $ J $ by
    \[
       ( i ,j)\leq ( i', j') :\iff
        i \leq i'\text{ and } j\leq j'.
    \]
    To see that $( J ,\leq)$ is directed, take two elements $(i_{1},j_{1}), (i_{2}, j_{2})\in  J $. As $I$ is directed, there exist $i_{3},j_{3}'\in I$ such that  $i_{i}\leq i_{3}$ and   $j_{i}\leq j_{3}'$ for $i=1,2$ and also $i_{i}\leq j_{3}'$  for $i=1,2,3$. By Assumption~\eqref{eq:choice of U}, there exists $j_{3}\geq j_{3}'$ such that $c_{j_{3}}\notin U. $  Note that $j_{3}$ is greater than or equal to  all other elements we have considered, so that  $(i_{3},j_{3})\in  J $ and $(i_{3},j_{3})$ dominates both $(i_{1},j_{1})$ and $(i_{2}, j_{2})$. 
    
    The map
    \(
       h\colon  J \to I, \,h( i ,j)=j,
    \)
    is monotone and final.
    Thus, $(c_{h(\beta)})_{\beta\in  J }$ is a subnet of $(c_{ i })_{ i \in I}$ that lives completely outside of $U$ by construction. The assumptions that $\lim_i  \norm{c_{ i }} = 0$ and $\lim_i  q_{\cC}(c_{ i })=  y$ imply that  $\lim_{\beta}\norm{c_{h(\beta)}}= 0$ and that $ y_{\beta} := q_{\cC}(c_{h(\beta)})$ converges to $ y$ also.
    
    \smallskip
    
    Now fix any $ x\in \varphi\inv ( y)\subseteq X$. Since \assumption{$\varphi$ is open and surjective},
    \cite[II.13.2]{FellDoranVol1} asserts that we can find a subnet set $\left\{ y_{f(\gamma)}\right\}_{\gamma\in \Gamma}$ of $\{ y_{\beta}\}_{\beta\in  J }$ and a net $\{ x_{\gamma}\}_{\gamma\in \Gamma}$ in $X$, indexed by the same set, such that  $ x_\gamma \to x $ and $\varphi( x_\gamma)= y_{f(\gamma)}$ for all $\gamma\in \Gamma$.

    For each $\gamma\in \Gamma$, let $b_{\gamma}$ be an element of $\Phi\inv(c_{f(\gamma)})\cap q_{\cB}\inv(x_{\gamma})$, which exists \assumption{because of \ref{item:general quot:preimages}}; in particular,
    \[
       \lim_{\gamma} q_{\cB} (b_{\gamma}) = \lim_{\gamma}  x_{\gamma} =  x 
       \quad\text{ and }\quad
       \lim_{\gamma}\norm{b_{\gamma}} = \lim_{\gamma}\norm{ c_{f(\gamma)} } = \lim_{i}\norm{ c_{i} } = 0,
    \]
    since, by construction, the net $\{c_{f(\gamma)}\}_{\gamma\in \Gamma}$ is a subnet of $\{c_{i}\}_{i\in I }$.
    Since \assumption{$\cB$ is a \USCBb}, it follows from Condition~\ref{cond.B-nets} of $\cB$ that $\lim_{\gamma}b_{\gamma}= 0_{ x }$ in $B$, implying that $\lim_{\gamma}\Phi(b_{\gamma})= 0_{ y}$ in $C$ by \assumption{fibrewise linearity of $\Phi$}. But by construction, $\Phi(b_{\gamma}) = c_{f(\gamma)} \notin U$ for all $\gamma\in \Gamma$, which is a contradiction as $U$ is a neighborhood of $0_{ y}$ in $C$.
\end{proof}

\begin{corollary}\label{cor:cP:USCBb}
    With the described structure on the bundle $\cP=(q_{\cP}\colon P \to \Z)$, it is a \USCBb.
\end{corollary}

\begin{proof}
    We will apply Proposition~\ref{prop:cd} to the following diagram which is commutative by construction; here, $K$ is the total space of the bundle $\cK=\cM\otimes_{\cC\z}\cN$.
    \[
    \begin{tikzcd}[ampersand replacement=\&]
        K
        \ar[r, "Q"]\ar[d, "q_{\cK}"']
        \&
        P
        \ar[d, "q_{\cP}"]
        \\
        \X\bfp{s}{r}\Y
        \ar[r, "{[\cdot]_{\ipscriptstyle\cH}}"]
        \&
        \Z
    \end{tikzcd}
    \]
    
    First recall that each fibre    $\cP_{z}$ has the structure of a complex Banach space by Lemma~\ref{lem:Pz=MxNy} and that $\cK$ is a \USCBb\ by  Lemma~\ref{lem:topology-on-cK}.
    By Proposition~\ref{prop:Q-open} and Corollary~\ref{cor:bfp-of-gpd}, the quotient maps $Q$ and  $[\cdot]_{\ipscriptstyle\cH}$ are open, so we have Assumption~\ref{item:general quot:quotient maps}. By definition of the Banach space
    structure on the fibres of $\cC$ (see Lemma~\ref{lem:Pz=MxNy}), we have both Assumption~\ref{item:general quot:Phi lin isometry} and~\ref{item:general quot:preimages}.
\end{proof}

\begin{lemma}\label{lem:cP:actions} 
    The left $\cB$-action on $\cM$ induces a left $\cB$-action on $\cP$. To be precise, the action is given by $b\cdot [\xi] := [b\cdot\xi]$ for $b\in B $ and $\xi\in \cK=\cM\otimes_{\cC\z}\cN$ with $s_{\cB}(b) =  r_{\cK} (\xi).$
    
    Similarly, the right $\cD$-action on $\cN$ induces a right $\cD$-action on $\cP$ given by $[\xi]\cdot d := [\xi\cdot d]$ for $d\in\cD$ and $\xi\in \cM\otimes_{\cC\z}\cN$ with $ s_{\cK} (\xi) = r_{\cD}(d).$
\end{lemma}

\begin{proof}
    We will only deal with the left action; the same proof, replacing left- with right-arguments and vice versa, shows the claim about the right action.
    
    We have seen in Proposition~\ref{prop:cK:actions} that $\cK$ has a left $\cB$-action determined by $b\cdot (m\otimes n)=(b\cdot m)\otimes n$, which explains what `$b\cdot \xi$' means. Let us first note that our map is well defined, so suppose $[\xi]=[\eta].$ Then there exists $h\in \cH$ such that $\xi=\Psi_{h}(\eta).$ We have seen in  Theorem~\ref{thm:Psi-cts}, \ref{item:Psi:biequivariant}, that $b\cdot \xi=b\cdot \Psi_{h}(\eta) = \Psi_{h}(b\cdot \eta).$ Thus,  $[b\cdot \xi]=[\Psi_{h}(b\cdot \eta)]=[b\cdot \eta],$ so the definition of $b\cdot [\xi]$ does not depend on the chosen representative $\xi$.
    
    To see that the map
    \[
       f\colon \quad \cB\bfp{s}{r}\cP \to \cP, \quad (b, p) \mapsto b\cdot p,
    \]
    is continuous, let $U\subseteq \cP$ be an open set, i.e., $Q\inv (U)$ is open in~$\cK$. As the left $\cB$-action on $\cK$ is continuous (see Proposition~\ref{prop:cK:actions}), we know that
    \[
       V
       :=
       \{
           (b,\xi) \in \cB\bfp{s}{r}\cK  \,\mid\,
           b\cdot \xi\in Q\inv (U)
       \}
    \]
    is open in $\cB\bfp{s}{r}\cK$. By Remark~\ref{rmk:id-ast-Q-open}, the set $(\mathrm{id}\ast Q)(V)$ is thus open in $\cB\bfp{s}{r}\cP$. Since
    \begin{align*}
       (\mathrm{id}\ast Q)(V)
       &=
       \left\{
           (b,[\xi]) \in \cB\bfp{s}{r}\cP  \,\mid\,
           b\cdot \xi\in Q\inv (U)
       \right\}
       \\
       &=
       \left\{
           (b,[\xi]) \in \cB\bfp{s}{r}\cP  \,\mid\,
           b\cdot [\xi]=[b\cdot \xi]\in U
       \right\}
       \\
       &=
       \left\{
           (b,p) \in \cB\bfp{s}{r}\cP  \,\mid\,
           b\cdot p\in U
       \right\}
       =
       f\inv (U),
    \end{align*}
    it follows that $f$ is continuous.
    It remains to check the numbered conditions of an action in Definition \ref{def:USCBb-action}.
    
    \smallskip
    
       For \ref{item:FA:fibre}, we compute:
       \begin{align*}
           q_{\cP}(b\cdot [\xi])
           &=
           q_{\cP}([b\cdot\xi])
           =
           [q_{\cK}(b\cdot\xi)]_{\ipscriptstyle\cH}
           \\
           &=
           [p_{\cB}(b)\cdot q_{\cK}(\xi)]_{\ipscriptstyle\cH},
           \quad \text{ by the same property for } \cK.
       \end{align*}
       By definition of the left $\cG$-action on $\Z$, we have for $g\in \cG$ and compatible $[x,y]_{\ipscriptstyle\cH}$ in $\Z$ that $[g\cdot x, y]_{\ipscriptstyle\cH} = g\cdot [x,y]_{\ipscriptstyle\cH}$. Both combined yield:
       \begin{align*}
           q_{\cP}(b\cdot [\xi])
           =
           p_{\cB}(b)\cdot [q_{\cK}(\xi)]_{\ipscriptstyle\cH}
           =
           p_{\cB}(b)\cdot q_{\cP}([\xi]),
           \quad\text{ as claimed.}
       \end{align*}

       Property \ref{item:FA:assoc}, i.e., that $b'\cdot (b\cdot [\xi])= (b'b)\cdot [\xi]$ for all compatible $b',b\in \cB$ and $[\xi]\in\cP$, follows from the same property of~$\cK$.
       
       \smallskip
       
       For \ref{item:FA:norm}, we compute
       \begin{alignat*}{2}
           \norm{b\cdot [\xi]}
           &=
           \norm{[b\cdot \xi]}
           =
           \norm{b\cdot \xi}
           \quad &&\text{by Lemma~\ref{lem:Pz=MxNy}}
           \\
           &\leq
           \norm{b}\,\norm{\xi}
           &&\text{by \ref{item:FA:norm} for~$\cK$}
           \\
           &=
           \norm{b}\,\norm{[\xi]}&&\text{by Lemma~\ref{lem:Pz=MxNy}}.
       \end{alignat*}
       This proves our claim.\qedhere
\end{proof}

Clearly, the left- and right-actions on $\cP$ commute, i.e., 
Condition~\ref{item:FE:actions} in Definition~\ref{def:FBequivalence} is satisfied.
In the following, we will use superscripts to provide more clarity.
\begin{proposition}
\label{prop:cP:ip}
    There exist sesquilinear, continuous maps $\linner[\cP]{\cB}{\cdot}{\cdot}$ on $\cP\bfp{s}{s}\cP$ and $\rinner[\cP]{\cD}{\cdot}{\cdot}$ on $\cP\bfp{r}{r}\cP$ defined for $\bigl([\xi], [\xi']\bigr)$ in the appropriate set by
    \begin{align*}
       \linner[\cP]{\cB}{[\xi]}{[\xi']}
       :=
       \lInner[\cK]{\cB}{\xi}{\xi'}
       \quad\text{resp.}\quad
       \rinner[\cP]{\cD}{[\xi]}{[\xi']}
       :=\ &
       \rInner[\cK]{\cD}{\xi}{\xi'}.
    \end{align*}
    These maps further satisfy the conditions in \ref{item:FE:ip} of Definition~\ref{def:FBequivalence}.
\end{proposition}

\begin{proof}
    As always, we will only deal with the left (pre-)inner product $\linner[\cP]{\cB}{\cdot}{\cdot}$.
    Let us first check that it is well defined for elementary tensors: suppose $\xi=(m\cdot c)\otimes n$ and $\xi'=(m'\cdot c')\otimes n'$ with $c\in\cC_{h}, c'\in\cC_{h'}$, so that $\xi$ is equivalent to $m \otimes (c\cdot n)$ and $\xi'$ to $m'\otimes (c'\cdot n')$. We have
    \begin{alignat*}{2}
       \lInner[\cK]{\cB}{\xi}{\xi'}
       &=
       \lInner[\cK]{\cB}{(m\cdot c)\otimes n}{(m'\cdot c')\otimes n'}
        \\       &
       =
       \linner{\cB}{m\cdot c}{(m'\cdot c')\cdot \linner{\cC}{n'}{n} }
       \\
       &=
       \linner{\cB}{m}{
       m'\cdot \bigl(c' \linner{\cC}{n'}{n}c^*\bigr) }
       \qquad&&
       \text{by Corollary~\ref{cor:adjointable}}.
    \end{alignat*}
    Note that
    \(
       c' \linner{\cC}{n'}{n}c^*
       =
       \linner{\cC}{c' \cdot n'}{c\cdot n},
    \) 
    so that all in all:
     \begin{align*}
       \lInner[\cK]{\cB}{\xi}{\xi'}
       &=
       \linner{\cB}{m}{m'\cdot \linner{\cC}{c' \cdot n'}{c\cdot n} }
       =
       \lInner[\cK]{\cB}{ m\otimes (c\cdot n)}{m' \otimes (c'\cdot n')}
       \\
       &=
       \lInner[\cK]{\cB}{\Psi_{h}(\xi)}{\Psi_{h'}(\xi')}
       .
    \end{align*}
    By linearity, we conclude that
    the same equation holds for all $\xi \in \cM_{x} \odot \cN_{y}$ and $\xi'\in \cM_{x'} \odot \cN_{y'}$ for which $s_{\Y}(y)=s_{\Y}(y')$ and for all compatible $h,h'$. By continuity of the inner product on $\cK$ and continuity of $\Psi$ (see Proposition~\ref{prop:cK:ip} resp.\ Theorem~\ref{thm:Psi-cts}), uniqueness of limits implies that the equality holds for all $(\xi,\xi')\in\cK\bfp{s}{s}\cK$ and compatible $h,h'.$ Consequently, the (pre-)inner product on $\cP$ is well defined.
    
    \smallskip
    
    To see that $f:=\linner[\cP]{\cB}{\cdot}{\cdot}$ is continuous, let $U\subseteq \cB$ be open. By continuity of $\tilde{f}:=\lInner[\cK]{\cB}{\cdot}{\cdot}$, we know that  $\tilde{f}\inv (U)$ is open in $\cK\bfp{s}{s}\cK.$ As the quotient map $Q$ is open by Proposition~\ref{prop:Q-open}, $(Q\bfp{s}{s}Q)(\tilde{f}\inv (U))$ is open in $\cP\bfp{s}{s}\cP.$ As
    \begin{align*}
       (Q\bfp{s}{s}Q)(\tilde{f}\inv (U))
       &
       =
       \left\{
           ([\xi],[\xi']) \,\mid\,
           \linner[\cP]{\cB}{[\xi]}{[\xi']}=
           \lInner[\cK]{\cB}{\xi}{\xi'}
           \in U
       \right\}
       =
       f\inv(U),
    \end{align*}
    $f$ is continuous.

    For the remaining conditions of Definition~\ref{def:FBequivalence}, note that sesquilinearity follows from the sesquilinearity of the inner product on $\cK$; Condition~\ref{item:FE:ip:fibre} follows from our computations in Remark~\ref{rmk:why-cP-is-not-faulty}; Conditions~\ref{item:FE:ip:adjoint} and \ref{item:FE:ip:C*linear} follow from Proposition~\ref{prop:cK:ip}~\ref{item:cK:ip-adj} and \ref{item:cK:ip-C*linear}.
    
    Lastly, for Condition~\ref{item:FE:ip:compatibility}, it suffices to prove $\lInner[\cP]{\cB}{[\xi_{1}]}{[\xi_{2}]}\cdot [\xi_{3}]  = [\xi_{1}]\cdot \rInner[\cP]{\cD}{[\xi_{2}]}{[\xi_{3}]}$   for elementary tensors $\xi_{i}=m_{i} \otimes n_{i}$. We compute:
       \begin{align*}
           \lInner[\cK]{\cB}{ \xi_{1}}{\xi_{2}} \cdot \xi_{3}
           &
           =
           \linner{\cB}{m_{1}}{m_{2}\cdot \linner{\cC}{n_{2}}{n_{1}} } \cdot \xi_{3}
           =
           \left(\linner{\cB}{m_{1}}{m_{2}\cdot \linner{\cC}{n_{2}}{n_{1}} } \cdot m_{3}\right)\otimes n_{3}
           \\
           &=
           \left( m_{1}\cdot \rinner{\cC}{m_{2}\cdot \linner{\cC}{n_{2}}{n_{1}} }{m_{3}}\right) \otimes n_{3}
           =\left( m_{1}\cdot c\right) \otimes n_{3}
           ,
       \end{align*}
       where
       \begin{align*}
           c
          :=\ &
          \rinner{\cC}{m_{2}\cdot \linner{\cC}{n_{2}}{n_{1}} }{m_{3}}
           =
           \rinner{\cC}{m_{2} }{m_{3}} \linner{\cC}{n_{1}}{n_{2}}
           =
           \linner{\cC}{n_{1}}{\rinner{\cC}{m_{3}}{m_{2}}\cdot n_{2}}.
       \end{align*}
       
       On the other hand,
       \begin{align*}
           \xi_{1}\cdot \rInner[\cK]{\cD}{\xi_{2}}{\xi_{3}}
           &=
           \xi_{1}\cdot \rinner{\cD}{\rinner{\cC}{m_{3}}{m_{2}}\cdot n_{2}}{n_{3}}
           =
           m_{1}\otimes \left(n_{1}\cdot \rinner{\cD}{\rinner{\cC}{m_{3}}{m_{2}}\cdot n_{2}}{n_{3}}\right)
           \\
           &=
           m_{1}\otimes \left(\linner{\cC}{n_{1}}{\rinner{\cC}{m_{3}}{m_{2}}\cdot n_{2}}\cdot n_{3}\right)
           =
           m_{1}\otimes \left(c\cdot n_{3}\right).
       \end{align*}
       Since $[(m_{1}\cdot c)\otimes n_{3}]=[m_{1}\otimes(c\cdot n_{3})]$ in $\cP$ (see Remark~\ref{rmk:balancing-in-cP}), we have proved the claim.
\end{proof}

We can now prove our second main theorem:
\pagebreak[3]\begin{theorem}\label{thm:cP-is-equiv}
Assume we are given 
\begin{itemize}
	\item three saturated Fell bundles $\cB, \cC, \cD$ over \LCH\ \etale\ groupoids
		$\cG,\cH,\mathcal{K}$, respectively;
	\item a $(\cG,\cH)$-equivalence $\X$ and an $(\cH,\mathcal{K})$-equivalence $\Y$;
	\item a $(\cB,\cC)$-equivalence $\cM$ over $\X$
	and a $(\cC,\cD)$-equivalence $\cN$	over $\Y$.
\end{itemize}
    Then the \USCBb\ $\cP=(q_{\cP}\colon P\to \X\bfp{}{\cH}\Y)$, defined as the quotient of $\cM\otimes_{\cC\z}\cN$ in Definition~\ref{def:cP}, is a $(\cB,\cD)$-Fell bundle equivalence when equipped with the actions defined in Lemma~\ref{lem:cP:actions} and the inner products defined in Proposition~\ref{prop:cP:ip}. 
\end{theorem}

    The following diagram gives an overview of our construction, where dotted arrows are the maps we have constructed in the current section.
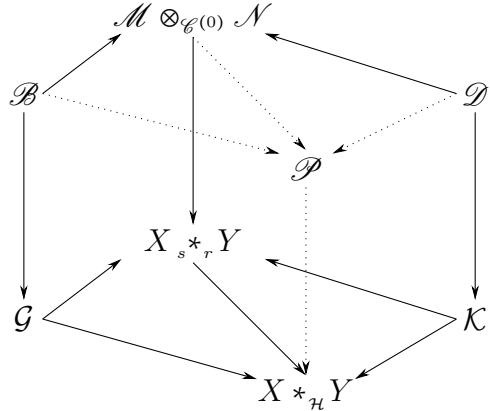
\begin{figure}[h]
    \centering
    \def\vs{3}
    \begin{tikzpicture}[scale=1, every node/.style={scale=1}]
    
    \node at (0,0) {$\cB$};
    \node at (2.25,1) {$\cM\otimes_{\cC\z}\cN$};
    \node at (3.75,-1) {$\cP$};
    \node at (6,0) {$\cD$};
    \draw  [-{Stealth[length=2mm, width=1mm]}] plot [smooth] coordinates {(0.25, 0) (1.3, 0.8)};
    \draw  [dotted, -{Stealth[length=2mm, width=1mm]}] plot [smooth] coordinates {(0.25, 0) (3.4, -0.8)};
    \draw  [-{Stealth[length=2mm, width=1mm]}] plot [smooth] coordinates {(5.75, 0) (3.2, 0.8)};
    \draw  [dotted, -{Stealth[length=2mm, width=1mm]}] plot [smooth] coordinates {(5.75, 0) (4.1, -0.8)};
    
    \node at (0,-\vs) {$\mathcal{G}$};
    \node at (2.25,1-\vs) {$X \bfp{s}{r} Y$};
    \node at (3.75,-1-\vs) {$X\bfp{}{\cH}Y$};
    \node at (6,-\vs) {$\mathcal{K}$};
    
    \draw  [-{Stealth[length=2mm, width=1mm]}] plot [smooth] coordinates {(0.25, -\vs) (1.3, 0.8-\vs)};
    \draw  [-{Stealth[length=2mm, width=1mm]}] plot [smooth] coordinates {(0.25, -\vs) (3.1, -0.8-\vs)};
    \draw  [-{Stealth[length=2mm, width=1mm]}] plot [smooth] coordinates {(5.75, -\vs) (3.2, 0.8-\vs)};
    \draw  [-{Stealth[length=2mm, width=1mm]}] plot [smooth] coordinates {(5.75, -\vs) (4.4, -0.8-\vs)};
    
    \draw  [-{Stealth[length=2mm, width=1mm]}] plot [smooth] coordinates {(2.25, 0.75) (2.25, 1.25-\vs)};
    \draw  [-{Stealth[length=2mm, width=1mm]}] plot [smooth] coordinates {(0, -0.25) (0, 0.25-\vs)};
    \draw  [-{Stealth[length=2mm, width=1mm]}] plot [smooth] coordinates {(6, -0.25) (6, 0.25-\vs)};
    \draw  [dotted, -{Stealth[length=2mm, width=1mm]}] plot [smooth] coordinates {(3.75, -1.25) (3.75, -0.75-\vs)};
    
    \draw  [dotted, -{Stealth[length=2mm, width=1mm]}] plot [smooth] coordinates {(2.25, 0.75) (3.75, -0.75)};
    \draw  [-{Stealth[length=2mm, width=1mm]}] plot [smooth] coordinates {(2.25, 0.75-\vs) (3.75, -0.75-\vs)};
    \end{tikzpicture}
\end{figure}

\begin{proof}
    In this proof, all Items refer to Definition~\ref{def:FBequivalence}.
    We have seen in Corollary~\ref{cor:cP:USCBb} that~$\cP$ is a \USCBb\ over the $(\cG,\mathcal{K})$-equivalence~$\Z=\X\bfp{}{\cH}\Y$. By Lemma~\ref{lem:cP:actions}, $\cP$ carries commuting $\cB$- and $\cD$-actions, so Item~\ref{item:FE:actions} is satisfied. By Proposition~\ref{prop:cP:ip}, the inner products on $\cP$ satisfy exactly the conditions in  Item~\ref{item:FE:ip}.
    
    For Item~\ref{item:FE:SMEs}, we recall that each fibre $\cP_{[x,y]}$ is isomorphic as bimodule to $\cM_{x}\otimes_{\cC\z}\cN_{y}$ by Lemma~\ref{lem:Pz=MxNy}. As argued in the proof of Theorem~\ref{thm:cK-is-word}, this is a $B(r_{\X}(x))-D(s_{\Y}(y))$-\ib,
    as needed for $\cP$ for Item~\ref{item:FE:SMEs}. This concludes our proof.
\end{proof}

\begin{corollary}
    Fell bundle equivalence is transitive and hence an equivalence relation.
\end{corollary}

\section{Acknowledgements}
The authors would like to thank Paul Muhly for patiently answering all of their queries regarding~\cite{MW2008}. The first-named author was supported by a RITA Investigator grant \mbox{(IV017)}. The second-named author was partially supported by Prof. Dilian Yang from University of Windsor. 

\clearpage
\appendix

\section{Some lemmas on topology and USC Banach bundles}


\begin{lemma}\label{lemma:proj-open-from-bfp}
    Suppose $ X ,  Y ,  Z $ are topological spaces, $p\colon  X \to  Z $ is continuous, and $q\colon  Y \to  Z $ is open. Then the map $\mathrm{pr}_{ X }\colon  X \bfp{p}{q} Y  \to  X , (x,y)\mapsto x$, is an open map.
\end{lemma}

\begin{proof}
    We have to show that, for any open $U\subseteq  X \bfp{p}{q} Y $ and $(x,y)\in U$, there exists an open neighborhood $V$ of $x$ in $ X $ such that $V\subseteq \mathrm{pr}_{ X }(U).$
    It suffices to consider a basic open set $U$, so assume $U=V_{1}\bfp{}{} V_{2}$ where $V_{1}$ is an open neighborhood of $x$ and $V_{2}$ of $y$.
    Let $V:= V_{1} \cap p\inv (q(W_{2}))$. This set is open, 
    since $p$ is continuous and $q$ is open. It is further a neighborhood of $\mathrm{pr}_{ X }(x,y)=x$, because $(x,y)\in V_{1}\times V_{2}$ and $p(x)=q(y)$. 
    Now take an arbitrary $a\in V$, so that $a\in V_{1}$ and there exists $b\in V_{2}$ such that $p(a)=q(b)$. In particular, $(a,b)\in V_{1}\bfp{}{}V_{2}\subseteq U$ and $\mathrm{pr}_{ X }(a,b)= a$, proving that $a\in \mathrm{pr}_{ X }(U)$.
\end{proof}


\begin{lemma}\label{lem:bfp of open is open}
    Suppose we are given bundles $\mathscr{A}, \mathscr{A}', \mathscr{B}, \mathscr{B}'$ over the same topological space $Z$ and maps $f, g$ as follows:
    \begin{equation*}
    \begin{tikzcd}[ampersand replacement=\&]
         {A}
        \ar[dd, "f"']
        \ar[rd, "q_{ \mathscr{A}}"]
        \& \&
         {B}
        \ar[dd, "g"]
        \ar[ld, "q_{ \mathscr{B}}"']
        \\
        \& Z \&
        \\
         {A}'
        \ar[ru, "q_{ \mathscr{A}'}"']
        \&\&
        \ar[lu, "q_{ \mathscr{B}'}"]
        B'
    \end{tikzcd}
    \end{equation*}
    If $f, g$ are open bundle maps, then the map 
    \[
    f\ast g\colon \mathscr{A}\bfp{q}{q}\mathscr{B}\to \mathscr{A}'\bfp{q}{q}\mathscr{B}'
    ,\quad (a,b)\mapsto (f(a),g(b)),
    \]is also an open bundle map.
\end{lemma}

\begin{proof}
    The map is a bundle map since
    \[(q\circ [f\ast g])(a,b)= q(f(a),g(b)) \overset{(\dagger)}{=} \bigl(q(f(a)),q(g(b))\bigr)\overset{(\ddagger)}{=}(q(a),q(b))\overset{(\dagger)}{=} q(a,b),
    \]
    where we dropped the subscripts on the bundle maps and where both equations labeled $(\dagger)$ follow from the definition of the bundle map on a fibre product and $(\ddagger)$ follows since since $f$ and $g$ are bundle maps
    . Next fix open sets $V\subseteq  {A}$ and $U\subseteq  {B}$, and consider $V\ast U := (V\times U)\cap ( \mathscr{A}\bfp{q}{q} \mathscr{B}).$
    We have $(f\ast g) (V\ast U)= f(V)\ast g(U)$, since $f$ and $g$ are bundle maps
    . Since $f$ and $g$ are open, $f(V)\ast g(U)$ is an open set in $\mathscr{A}'\bfp{q}{q}\mathscr{B}'$. This proves that $f\ast g$ maps basic open sets to (basic) open sets and is hence open.
\end{proof}

The following lemmas 
give an \usc\ version of \cite[II.13.16 and 17]{FellDoranVol1} and should be compared to \cite[Proposition 2.4]{BE2012}. 
We first prove a strengthened condition on the convergence in an \USCBb. 
\begin{lemma}[{\cite[Proposition 2.4]{BE2012}}]\label{lm.conv.equivalence.conditions}
Let $\cM=(q_\cM\colon M\to X)$ be an \USCBb\ over some Hausdorff space $X$. Let $\Gamma \subseteq \Gamma(X;\cM) $ be a vector space of continuous sections such that $\{f(x): f\in\Gamma \}$ is dense in~$M_x$ for each~$x\in X$.
Suppose $(m_{i})_{i}$ is a net in $M$ and $m\in M$. Then the following are equivalent:
\begin{enumerate}[label=\textup{(\arabic*)}]
    \item\label{item:convergence1} $m_{i}\to m$ in $M$.
    \item\label{item:convergence3.gamma} We have $q_\cM(m_{i})\to q_\cM(m)$ and for every $\epsilon>0$, there exists $g\in\Gamma$ such that $\norm{m-g(q_\cM(m))}<\epsilon$ and  $\norm{m_{i}-g(q_\cM(m_{i}))} < \epsilon$ for $i$ large enough. 
    \item\label{item:convergence2} We have $q_\cM(m_{i})\to q_\cM(m)$ and for all 
    $f\in\Gamma $, \[\limsup \|m_{i}-f(q_\cM(m_{i}))\| \leq \|m-f(q_\cM(m))\|.\]
\end{enumerate}
\end{lemma}

\begin{proof}
Let $x_{i}:= q_{\cM}(m_{i})$ and $x:= q_{\cM}(m).$
For \ref{item:convergence1}$\implies$\ref{item:convergence2}, continuity of $q_\cM$ implies $x_{i} \to x $. Since any $f\in\Gamma$ is continuous, we thus have $f(x_{i} )\to f(x )$. By Condition~\ref{cond.B-plus} and \ref{cond.B-times}, it follows that $m_{i}-f(x_{i} ) \to m-f(x)$. By Condition~\ref{cond.B-nets}, we conclude 
\[\limsup \|m_{i}-f(x_{i} )\| \leq \|m-f(x)\|.\]

For \ref{item:convergence2}$\implies$\ref{item:convergence3.gamma}, note that for every $\epsilon>0$, there exists $g\in \Gamma$ such that $\norm{m-g(x )}<\epsilon$, since $\{f(x ): f\in\Gamma \}$ is dense in $M_{x }$. By \ref{item:convergence2}, 
\[\limsup\norm{m_{i}-g(x_{i} )} \leq \norm{m-g(x )}
.\]
Therefore, there exists $i_{0}$ such that for all $i\geq i_{0}$, $\norm{m_{i}-g(x_{i} )}\leq \norm{m-g(x )}<\epsilon$. 

Finally, \ref{item:convergence3.gamma}$\implies$\ref{item:convergence1} follows from \cite[Proposition C.20]{Wil2007}. 
\end{proof}

\begin{remark} For continuous Banach bundles, the $\limsup$ in Condition~\ref{item:convergence2} can be replaced by $\lim$ (see \cite[II 13.12]{FellDoranVol1}). While \cite[Appendix C]{Wil2007} focuses on \usc\ $\textrm{C}^*$-bundles, several of its proofs, including the proof of \cite[Proposition C.20]{Wil2007}, do not use the $\textrm{C}^*$-identity and thus work for general \USCBb. 

If we pick a continuous section $f\in\Gamma(X;\cM)$ such that $f(q_\cM(m))=m$,
then $m_{i}\to m$ if and only if $q_\cM(m_{i})\to q_\cM(m)$ and by \ref{item:convergence2}, 
\[0\leq \limsup \|m_{i}-f(q_\cM(m_{i}))\| \leq \|m-f(q_\cM(m_{i}))\|=0.\]
In this case, the $\limsup$ can likewise be replaced by $\lim$ (see \cite[Lemma A.3.]{MW2008}).
\end{remark}

\begin{lemma}[{cf.\ {\cite[Lemma C.18]{Wil2007}}}]
\label{lm.convergence.at.zero} Let $\cM=(q_\cM\colon M\to X)$ be a \USCBb\ over some Hausdorff space $X$. Then $m_{i}\to 0_x$ if and only if $q_\cM(m_{i})\to x$ and $\|m_{i}\|\to 0$. 
\end{lemma}

\begin{proof}
	It follows from the upper semi-continuity of $m\mapsto \|m\|$ and the continuity of $q_{\cM}$ that $m_{i}\to 0_x$ implies $q_\cM(m_{i})\to x$ and $\limsup \|m_{i}\|\leq 0$; in particular, $\|m_{i}\|\to 0$. 
The converse follows from Condition~\ref{cond.B-nets}.
\end{proof}

\begin{remark}
	For continuous Banach bundles, Lemma~\ref{lm.convergence.at.zero} is a direct consequence of continuity of the norm. As our proof shows, however,
	upper semi-continuity is in fact sufficient. 
\end{remark}

The above results imply that a USC analogue of \cite[II.13.16 and II.13.17]{FellDoranVol1} holds: their proofs go through verbatim, except 
where continuity of the norm 
needs to be replaced by $\limsup$-approximation of the norm. To be more precise:

\begin{proposition}[cf.\ {\cite[II.13.16]{FellDoranVol1}}]\label{prop.13.16}
Let $\cM=(q_\cM\colon M\to X)$ and $\cM'=(q_{\cM'}\colon M'\to X')$ be two {\USCBb s} 
and let $\omega\colon X\to X'$ be a homeomorphism. 
Let $\Gamma $ be a vector space of continuous sections of $\cM$ such that the $\mathbb{C}$-linear span of
 $\{f(x): f\in\Gamma \}$ is dense in each~$M_x$.
Let $\Phi\colon M\to M'$ be a map which satisfies the following:
\begin{enumerate}[label=\textup{(\alph*)}]
    \item\label{item:ptws linear} For each $x\in X$, $\Phi(M_x)\subseteq M'_{\omega(x)}$ and $\Phi|_{M_x}$ is linear.
    \item\label{item:constant} There exists a constant $K>0$ such that $\|\Phi(m)\|\leq K\|m\|$ for all $m\in M$.
    \item\label{item:Phi circ f} For each $f\in\Gamma $,
    $\Phi\circ f \circ\omega\inv$ is a continuous section of $\cM'$.
\end{enumerate}
Then $\Phi$ is continuous. 
\end{proposition}

\begin{proof}

Suppose we have a convergent net in $\cM$, say $m_{i}\to m$. 
	Let $x_{i}=q_{\cM}(m_i)$, $x=q_{\cM}(m),$ and $x'_{i}=\omega(q_{\cM}(m_i))$, $x'=\omega(q_{\cM}(m))$; by continuity of $q_{\cM}$ and $\omega$, we have $x_{i}\to x$ in $X$ and $x'_{i}\to x'$ in~$X'.$ By \ref{item:ptws linear}, we have $\Phi(m_{i})=x'_{i}$ and $\Phi(m)=x'$.

Now fix $\epsilon>0$. By density of $\Gamma$ in $M_{x}$, we can pick some $g_1,\ldots, g_{n}\in\Gamma$ such that we have $\|m-g(x)\|<\frac{\epsilon}{K}$ for $g:=\sum_{j=1}^{n}g_{j}$ and $K$ the constant in \ref{item:constant}. 
Since $g$ is continuous, Lemma~\ref{lm.conv.equivalence.conditions},  \ref{item:convergence1}$\implies$\ref{item:convergence2}, asserts that
    there exists $i_0$ such that for all $i\geq i_0,$
    $$\|m_{i}-g(x_{i})\|\leq \|m-g(x)\|<\frac{\epsilon}{K}.$$
 Using Assumptions~\ref{item:ptws linear} and~\ref{item:constant}, we see that not only
$$\|\Phi(m)-\Phi\circ g(x)\|
\leq K\|m - g(x)\|
<\epsilon,$$
but also for
    $i\geq i_0$,
$$
 \|\Phi(m_{i})-\Phi\circ g(x_{i}))\|
\leq K \|m_{i} -  g(x_{i})\|
<\epsilon.$$
If we let 
$f:=\Phi\circ g\circ\omega\inv$, then the above inequalities can be rephrased to
\[
\|\Phi(m)- f(x')\|
<\epsilon
\quad\text{and}\quad
\|\Phi(m_{i})- f(x'_{i})\|
<\epsilon
\]
for all $i\geq i_{0}$. 
By Assumption~\ref{item:Phi circ f}, each $\Phi\circ g_{j}\circ \omega\inv$ is a continuous section of $\cM'$, and hence so is $f=\sum_{j=1}^{n} \Phi\circ g_{j}\circ \omega\inv$. Since $\epsilon>0$ was arbitrary, it follows from \cite[Proposition C.20]{Wil2007} that $\Phi(m_{i})\to \Phi(m)$, proving the continuity of~$\Phi$. 
\end{proof}

\begin{proposition}[cf.\ {\cite[II.13.17]{FellDoranVol1}}]\label{prop.13.17}
Let $\cM=(q_\cM\colon M\to X)$ and $\cM'=(q_{\cM'}\colon M'\to X')$ be two {\USCBb s} 
and let $\omega\colon X\to X'$ be a homeomorphism. 
Let $\Phi\colon M\to M'$ be a map such that:
\begin{enumerate}[label=\textup{(\arabic*)}]
    \item For each $x\in X$, $\Phi(M_x)\subseteq M'_{\omega(x)}$ and $\Phi|_{M_x}$ is linear.
    \item $\Phi$ is continuous.
    \item There exists a constant $k>0$ such that $\|\Phi(m)\| \geq  k\|m\|$ for all $m\in M$.
\end{enumerate}
Then $\Phi\inv \colon \Phi(M)\to M$ is continuous.
\end{proposition}

Note that $\Phi\inv$ exists because $\Phi$ is injective as a fibrewise linear map that is bounded below in norm.

\begin{proof} 
The proof of {\cite[II.13.17]{FellDoranVol1}} can be adapted almost verbatim: They invoke continuity only once to argue that $m_{i}\to 0_x$ if and only if $q_\cM(m_{i})\to x$ and $\|m_{i}\|\to 0$. But we have seen in 
Lemma~\ref{lm.convergence.at.zero} that that condition is also true in \USCBb s. Moreover, adapting from $X=X'$ to the more general case involving $\omega$ is straight forward as done in Proposition~\ref{prop.13.16}. 
We will omit the details. 
\end{proof}


\begin{thebibliography}{10}

\bibitem{AF:EquivFb}
F.~Abadie and D.~Ferraro.
\newblock Equivalence of {F}ell bundles over groups.
\newblock {\em J. Operator Theory}, 81(2):273--319, 2019.

\bibitem{BE2012}
A.~Buss and R.~Exel.
\newblock Fell bundles over inverse semigroups and twisted \'{e}tale groupoids.
\newblock {\em J. Operator Theory}, 67(1):153--205, 2012.

\bibitem{FellDoranVol1}
J.~M.~G. Fell and R.~S. Doran.
\newblock {\em Representations of {$^*$}-algebras, locally compact groups, and
  {B}anach {$^*$}-algebraic bundles. {V}ol. 1}, volume 125 of {\em Pure and
  Applied Mathematics}.
\newblock Academic Press, Inc., Boston, MA, 1988.
\newblock Basic representation theory of groups and algebras.

\bibitem{Hofmann1977}
K.~H. Hofmann.
\newblock Bundles and sheaves are equivalent in the category of {B}anach
  spaces.
\newblock In {\em {$K$}-theory and operator algebras ({P}roc. {C}onf., {U}niv.
  {G}eorgia, {A}thens, {G}a., 1975)}, pages 53--69. Lecture Notes in Math.,
  Vol. 575, 1977.

\bibitem{LAZAR2018448}
A.~J. Lazar.
\newblock A selection theorem for banach bundles and applications.
\newblock {\em Journal of Mathematical Analysis and Applications},
  462(1):448--470, 2018.

\bibitem{MRW:Grpd}
P.~S. Muhly, J.~N. Renault, and D.~P. Williams.
\newblock Equivalence and isomorphism for groupoid {$C^\ast$}-algebras.
\newblock {\em J. Operator Theory}, 17(1):3--22, 1987.

\bibitem{MW2008}
P.~S. Muhly and D.~P. Williams.
\newblock Equivalence and disintegration theorems for {F}ell bundles and their
  {$C^*$}-algebras.
\newblock {\em Dissertationes Math.}, 456:1--57, 2008.

\bibitem{RW:Morita}
I.~Raeburn and D.~P. Williams.
\newblock {\em Morita equivalence and continuous-trace {$C^*$}-algebras},
  volume~60 of {\em Mathematical Surveys and Monographs}.
\newblock American Mathematical Society, Providence, RI, 1998.

\bibitem{Wil2007}
D.~P. Williams.
\newblock {\em Crossed products of {$C{^\ast}$}-algebras}, volume 134 of {\em
  Mathematical Surveys and Monographs}.
\newblock American Mathematical Society, Providence, RI, 2007.

\bibitem{Wil2019}
D.~P. Williams.
\newblock {\em A tool kit for groupoid {$C^*$}-algebras}, volume 241 of {\em
  Mathematical Surveys and Monographs}.
\newblock American Mathematical Society, Providence, RI, 2019.

\end{thebibliography}

\end{document}